\documentclass{article}
\usepackage{graphicx} % Required for inserting images
\usepackage{amsmath,amsfonts, amsthm, amssymb, xcolor, enumitem}
\usepackage{tabularx, longtable, array, verbatim}

\numberwithin{equation}{section}

\usepackage{hyperref}

\usepackage[vlined,ruled,algo2e]{algorithm2e}
\SetAlCapHSkip{0em}
\SetCustomAlgoRuledWidth{\textwidth}

\title{An $SU(2n)$-valued nonlinear Fourier transform}
\author{Michel Alexis, Lars Becker, Diogo Oliveira e Silva, Christoph Thiele}

\newtheorem{theorem}{Theorem}[section]
\newtheorem{lemma}[theorem]{Lemma}
\newtheorem{remark}[theorem]{Remark}
\newtheorem{cor}[theorem]{Corollary}
\newtheorem{definition}[theorem]{Definition}

\makeatletter
\let\c@algocf\c@theorem

\makeatother

\def\T{\mathbb{T}} 
\def\D{\mathbb{D}}
\def\Z{\mathbb{Z}}
\def\C{\mathbb{C}}
\def\R{\mathbb{R}}
\def\contract{\mathcal{C}}

\newcommand{\Id}{\operatorname{Id}}

\newcommand{\spce}{\mathbf{L}_{+} ^{*}}
\newcommand{\phl}{\mathcal{P}_{\spce}}

\newcommand{\dense}{\mathbf{D}}
\newcommand{\ddense}{\mathbf{E}}
\newcommand{\adj}{\operatorname{adj}}
\newcommand{\outr}{\mathcal{O}}
\newcommand{\of}{\omega} %outer scalar function
\newcommand{\weights}{\mathbf{P}}
\newcommand{\weight}{P}
\renewcommand{\epsilon}{\varepsilon}
\newcommand{\trc}{\operatorname{tr}}
\newcommand{\grp}{\mathcal{G}} %matrix group
\newcommand{\ad}{\mathrm{Ad}} %conjugate by z
\newcommand{\aarrow}{\curvearrowright}
\renewcommand{\Re}{\mathrm{Re}}
\newcommand{\op}{\mathcal{A}} %antisymmetric operator
\newcommand{\nlf}{M} %denotes things which are NLFTs
\newcommand{\metric}{d} %denotes rescalings of NLFTs
\newcommand{\con}{\mathcal{C}}
\newcommand{\sqrtt}{\mathrm{sqrt}}

\newcommand{\ma}[1]{\textcolor{blue}{MA : #1}}
\newcommand{\lb}[1]{\textcolor{blue}{LB : #1}}

\begin{document}

\date{}\maketitle
\begin{abstract}
    We define a nonlinear Fourier transform which maps sequences of contractive  $n \times n$ matrices to $SU(2n)$-valued functions on the circle $\T$. We characterize the image of finitely supported sequences and square-summable sequences on the half-line, and construct an inverse for $SU(2n)$-valued functions whose diagonal $n \times n$ blocks are outer matrix functions.   As an application, we relate this nonlinear Fourier transform with quantum signal processing over $U(2n)$ and multivariate quantum signal processing. 
\end{abstract}
\tableofcontents

\section{Introduction}

In \cite{Alexis+24}, a close connection was pointed out between an important algorithm in quantum signal processing and the $SU(2)$ nonlinear Fourier series. This led to a flow of ideas in both directions.

On the one hand, techniques for nonlinear Fourier series such as Riemann--Hilbert factorization 
were used for computational tasks in quantum signal processing \cite{Alexis+25}. Remarkable improvement in computational performance was subsequently obtained, for example using fast Toeplitz solvers in \cite{ni2024fast} and most recently the discovery of a fast inverse nonlinear Fourier transform (NLFT) of complexity order $N\log(N)^2$ using Riemann--Hilbert factorization to cut a signal in half followed by down-sampling of the data towards half the amount on each piece before iterating \cite{ni2025inverse}. For more information about how quantum signal processing relates to the NLFT, see for instance \cite{laneve25}, and for more on the former, see \cite{lin25_survey} for recent developments and \cite{rossi_survey} for its role in other quantum algorithms.   

On the other hand, quantum signal processing (QSP) motivated a particular and rather stringent analytic setup for the $SU(2)$ nonlinear Fourier series, whose study up till recently had been eschewed in favor of its more famous $SU(1,1)$ counterpart \cite{MTT, KoSiRu, Si17, BeGu24}; see also \cite{BeCo, AKNS} for other NLFTs. This stringent $SU(2)$ setup has led to an existence and uniqueness result for the inverse NLFT \cite{Alexis+25} in a subspace of $L^2$ and to 
results in the theory of one-sided orthogonal polynomials \cite{alexisb+25}, a theory that mirrors  that of the $SU(1,1)$ NLFT \cite{simon}. The present paper continues this drive by establishing an analogous existence and uniqueness result in a subspace of $L^2$ for an $SU(2n)$-valued nonlinear Fourier series.

  We recall the $SU(2)$-valued NLFT of a finitely supported  sequence 
  $(F_j)_{j\in \Z}$ of contractive complex numbers, that is, numbers of modulus less than one. Write for a complex number $z$ on the unit circle $\T$, 
  % \lb{I am putting my comment here again. I think $\ad$ is usually used for the action on the Lie algebra (which is where the coefficients live for the continuous NLFT, I guess). But here the matrices are in the Lie group, and $\ad$ is a group homomorphism and not a Lie algebra homomorphism. ChatGPT estimates that about 70 percent of textbooks use $\ad$ only for the action on the Lie algebra. } \ct{I followed https://ncatlab.org/nlab/show/adjoint+action where lower case ad is for the Lie algebra and upper case Ad is the group action.
  % Not sure how typical this is. I also find $\cdot$ for conjugation, but this is maybe too hard to
  % distinguish from the product. Maybe we dont need Ad, by the time we have written $Ad(Z)^jX$ we also
  % have written $Z^jXZ^{-j}$} 
 %   \begin{equation}\label{defineAd}
 %   \ad (z) X:=ZXZ^{-1}\ ,    
 %   \end{equation}
   \begin{equation}\label{defineZ}
       Z:=
    \begin{pmatrix}
        z^{\frac 12}& 0 \\ 0 & z^{-\frac 12}
    \end{pmatrix}\ .
   \end{equation}
The fractional power $z^{\frac 12}$ is formal as the forthcoming expressions simplify towards involving only integer powers of $z$.
The $SU(2)$-valued NLFT  of $(F_j)_{j \in \Z}$ is then
    \begin{equation}\label{def:NLFTsu2}
    \mathcal{F}(F)(z):=\prod\limits_{j \in \mathbb{Z}}^{\curvearrowright} Z^j
    \begin{pmatrix}
        \sqrt{1-|F_j|^2}& F_j \\ -\overline{F_j} & \sqrt{1-|F_j|^2}
    \end{pmatrix} Z^{-j}
    \end{equation}
   \begin{equation}\label{def:NLFTalt}
   =Z^{j_0}\left(\prod\limits_{j_0\le j\le j_1}^{\curvearrowright} 
    \begin{pmatrix}
        \sqrt{1-|F_j|^2}& F_j \\ -\overline{F_j} & \sqrt{1-|F_j|^2}
    \end{pmatrix} Z\right)Z^{-j_1-1}
    \end{equation}
for $j_0$ the first and $j_1$ the last nonzero entries of the sequence $F_j$.
Here the noncommutative product is a finite product  in increasing order
of $j$ from left to right.   Note that, within the brackets of \eqref{def:NLFTalt}, one has
an alternating product between elements of two one-parameter subgroups of $SU(2)$, which
makes the expression attractive for a quantum computer. The elements of one subgroup are
determined by the nonlinear Fourier coefficients $F_j$, while the elements of the other subgroup
are determined by the argument $z$.

The diagonal and the off-diagonal terms in the above expressions play very different roles.
For example, the off-diagonal elements of the coefficient matrices carry the full information
of the coefficient matrix, while the diagonal elements carry no additional information.
It is then natural to maintain the $2\times 2$ block structure when going to $SU(2n)$.
The off-diagonal blocks of the coefficient matrices will be rather general contractive matrices,
while the diagonal blocks will carry essentially no additional information (see Theorem \ref{factorization-theorem} below). As for the analog of \eqref{defineZ}, we  generalize $Z$ to have diagonal blocks $z^{\pm \frac 1 2}$ times the identity.
%In Appendix \ref{section:qsv}, we discuss how our model relates to QSP over $U(N)$ and QSP in two variables.

Fix a dimension $n\ge 1$
and denote by $\mathcal{M}$ the set of complex $n\times n$ matrices.
For $F\in \mathcal{M}$, write $\|F\|_2$ for the Hilbert--Schmidt norm and $\|F\|_\infty$
for the operator norm, that is, the largest singular value of $F$.  
We call $F$ contractive if $\|F\|_\infty<1$ and 
denote by $\con$ the set of contractive matrices in $\mathcal{M}$. 
%\[
%    \con := \{F \in \mathbb{C}^{n\times n} : \|F\|_\infty < 1\}.
%\]

Let $H^\infty(\D; \mathcal{M})$ denote the analytic Hardy space with values in $\mathcal{M}$,
that is, the set of bounded measurable functions $A:\T\to \mathcal{M}$ whose entries have
analytic extensions in the sense of  scalar $H^\infty(\D)$ \cite{garnett}. As usual, we shall identify measurable functions which only differ on a set of measure zero. We shall also identify functions in $H^\infty(\D; \mathcal{M})$ with their analytic extensions to the unit disc $\D$. In denoting these spaces, we will drop the $\mathcal{M}$ and just write $H^{\infty} (\D)$, as whether a function is matrix- or scalar-valued will be clear from context. We define $H^p (\D)$ similarly for all $0<p<\infty$. 

A matrix-valued function\footnote{In what follows, matrix-valued functions denote equivalence classes of measurable functions, identified up to almost everywhere (a.e.) equality.} $B:\T\to \con$ is called Szeg\H o if 
\begin{equation}\label{eq:Szego_condition}
\int\limits_{\T} \log \det (\Id - B B^*) > - \infty \, ,
\end{equation}
where we adopt the convention that integrals over $\T$ are with respect to the uniform probability measure on  $\T$. 
We denote by $\mathbf{S}$ the set of Szeg\H o functions. 

An $H^{p} (\D)$ function $O$ is called outer if its determinant satisfies the logarithmic mean value property
\[
\log \lvert \det O(0)\rvert =  \int\limits_{\T} \log \lvert \det O \rvert \, .
\]
We write matrices in $SU(2n)$
as $2\times 2$ block matrices with blocks
in $\mathcal{M}$.
We denote by $\grp_0$ and $\grp_1$ the groups of upper and lower triangular complex matrices with positive diagonal entries, respectively.
Let $\Z_2:=\{0,1\}$.

\begin{definition}
    For $\alpha:\Z_2\to \Z_2$, 
denote by $\mathbf{B}_\alpha$ the set of measurable $SU(2n)$-valued matrix functions on $\T$
 \[
    \begin{pmatrix}
    A & B \\ C & D
\end{pmatrix}\, ,
\]
such that $B$ and $C$ are Szeg\H o, $A^*$ and $D$ are outer and normalized so that $A(\infty)$
is in $\mathcal{G}_{\alpha(0)}$
and $D(0)$
is in $\mathcal{G}_{\alpha(1)}$. 
\end{definition}
Note that the function $\alpha$ encodes one of four possible normalizations for the blocks $A$ and $D$. 
Our first theorem states that every Szeg\H o function $B$ can be uniquely extended to a matrix function in $\mathbf{B}_\alpha$. 

\begin{theorem}\label{factorization-theorem}
    For each $\alpha:\Z_2\to \Z_2$ and $B \in \mathbf{S}$, there is a unique 
    $Y_\alpha(B) \in \mathbf{B}_\alpha$
    such that the upper right block of $Y_\alpha(B)$ is equal to $B$.
\end{theorem}
We present a short proof of Theorem \ref{factorization-theorem} in Section \ref{sec:B_to_Y_outer}. The main ingredient in the proof is the  known spectral factorization theorem
for matrix functions, for which we
refer to the simple proof in \cite{Barclay_cty_spec}
%\cite{Lowdenslager} 
and further references therein.

%\begin{cor}
%    Given a contractive $F\in \C^{n\times n}$, there exists unique $n \times n$ matrices $E$,$G$,$H$, with $E$ and $H$ upper triangular with positive diagonal entries, such that the matrix 
%    \[
%    Y(F) = \begin{pmatrix}
%        E & F \\ G & H
%    \end{pmatrix}
%    \]
%    is in $SU(2n)$. We denote the space of such matrices as $SU'(2n)$.
%\end{cor}
%The matrices in $SU'(2n)$ are the building blocks of our nonlinear Fourier Transform (NLFT).
\begin{definition}[Forward finite matrix NLFT]
    Let $\alpha:\Z_2\to \Z_2$.
    Let $F=(F_j)_{j\in \Z}$ be a sequence of contractive matrices,
    identified with constant elements in $\mathbf{S}$, and assume only finitely many $F_j$ are nonzero.
Given $z \in \T$, define $Z$ analogously to  \eqref{defineZ} for block matrices.
Define the $\alpha$-$SU(2n)$ nonlinear Fourier transform  of $F$ as     \begin{equation}\label{def:NLFT_alpha}
    \mathcal{F}_\alpha(F)(z):=\prod\limits_{j \in \mathbb{Z}}^{\curvearrowright} Z^j Y_\alpha(F_j) Z^{-j}. 
    \end{equation}
Here the product is in increasing order
of $j$ from left to right and it is finite because $Y_\alpha(0)$ is the identity.    
\end{definition}

Next, we extend the definition to certain sequences with infinite support. When $n=1$, multilinear expansion extends the NLFT to the set of $\ell^p$ sequences for $1 \leq p<2$; see \cite[Lectures 1.3--1.4]{TT03} and \cite[Theorem 2.5]{tsai}. However, motivated by the $\ell^2$ theory which was applied in quantum signal processing \cite{Alexis+24, Alexis+25}, we do not generalize this process for $n \geq2$, but instead jump directly to the larger space of square summable sequences using an approximation argument. 
This extension of the NLFT requires a good target space with a suitable metric, which we are about to define. 
Let $\ell^2(\Z ; \con)$ denote the space of sequences $(F_j)_{j\in \Z}$ of 
contractive matrices 
such that $\sum \|F_j\|_2 ^2 < \infty$. We endow $\ell^2(\Z ; \con)$ with the  $\ell^2$ metric.
Analogously we define $\ell^2(\Z_{\ge 0} ; \con)$ and $\ell^2(\Z_{<0};\con)$. In what follows, we denote the identity matrix by $\Id$.
\begin{definition}
    \label{defH}
    Let $\alpha:\Z_2\to \Z_2$. The space $\mathbf{H}_\alpha ^+$ consists of all matrix functions
    \begin{equation}\label{labeling_M_intro}
        \nlf: \mathbb{T} \to SU(2n), \qquad M = \begin{pmatrix}
            A&B\\
            C&D
        \end{pmatrix}
    \end{equation}
    satisfying the following properties:
    \begin{enumerate}
        \item \label{item:1st_prop_H} $A^*,C^*, B, D  \in H^2(\mathbb{D})$;
        \item \label{item:2nd_prop_H} $A(\infty)\in \mathcal{G}_{\alpha(0)}$ and $D(0)\in \mathcal{G}_{\alpha(1)}$;
        \item \label{item:3rd_prop_H}  if there exist $I_1 ^*, I_2 \in H^2 (\D)$, both  unitary a.e.\ 
    on $\T$,
    and there exists $\nlf '$ of the form \eqref{labeling_M_intro} satisfying Properties \ref{item:1st_prop_H} and \ref{item:2nd_prop_H}, for which 
    \begin{equation}
        \label{eq:inner_factor}
       \nlf = \nlf' \begin{pmatrix} I_1 & 0 \\ 0 & I_2 \end{pmatrix} \, ,
    \end{equation}  
    then $I_1 = I_2 =\Id$. 
    \end{enumerate}
  On $\mathbf{H}_\alpha ^+$ we define the metric 
    \[
        \metric(\nlf, \nlf'):= \left(\int_\mathbb{T} \|\nlf - \nlf'\|^2_2 \right)^{\frac12} + \lvert \log  \det A(\infty) - \log \det A'(\infty)\rvert.
    \]
     % \[
     %     + \lvert \log  \det A(\infty) - \log \det A'(\infty)\rvert  + \lvert \log  \det D(0) - \log \det D'(0)\rvert\,.
     % \]
    %\lb{The $D$ and $A$ term are equal. We just need one of them, and should change the definition and all places where it is used.}
\end{definition}
Our next result identifies $\mathbf{H}_{\alpha} ^+$ as the image of the square summable data on the right half-line under the $\alpha$-$SU(2n)$ NLFT.
\begin{theorem}\label{thm:l2_extension}
   The map $\mathcal F_\alpha$ extends to a homeomorphism from $\ell^2 (\mathbb{Z}_{ \geq 0} ; \con)$ onto $\mathbf{H}_{\alpha} ^+$. 
\end{theorem}
In the proof of Theorem \ref{thm:l2_extension}, we show injectivity of $\mathcal{F}_{\alpha}$ into $\mathbf{H}_{\alpha} ^+$ and surjectivity onto the seemingly larger space $\mathbf{U}_\alpha^+$, defined as in Definition \ref{defH} but with $SU(2n)$ replaced by $U(2n)$ in \eqref{labeling_M_intro}. A curious consequence of our proof is that the spaces $\mathbf{H}_{\alpha} ^+$ and $\mathbf{U}_{\alpha} ^+$ in fact coincide. It would be interesting to have a direct proof of this fact. While we claim the equality here, we will prove it in Corollary \ref{cor:H_equals_U}, avoiding any semblance of circular reasoning.

In what follows, let $\mathbf{S}^{\epsilon}$ denote those elements $B \in \mathbf{S}$ such that
\[
\|B(z)\|_{\infty} \leq 1- \epsilon
\]
for almost every $z \in \T$. Motivated by the application to QSP of the $SU(2)$-valued NLFT, we show that the nonlinear Fourier coefficients can be uniquely recovered whenever $A$ (or $D$) is outer, and furthermore that this process is ``stable'', that is, Lipschitz continuous, whenever the function $B$ is bounded away from $1$.

\begin{theorem}\label{thm:NLFT_outer_inverse}
For every $\alpha:\Z_2\to \Z_2$ and $B \in \mathbf{S}$, there exists a unique $F \in \ell^2 (\Z ; \con)$
such that
\[\mathcal{F}_\alpha(F)=Y_\alpha(B)\,.\]
For this $F$, we have 
the nonlinear Plancherel identity 
\begin{equation}\label{eq:Plancherel_QSP}
\sum_{j \in \mathbb{Z}} \log  \det(\Id - F_jF_j^*) =\int_{\mathbb{T}} \log \det(\Id-BB^*) \, .
\end{equation}
Furthermore, for every $\epsilon > 0$, there exists a constant $C_{\epsilon, n}$ for which  we have the Lipschitz bounds
\begin{equation}\label{eq:Lip_bds}
\sup\limits_{j\in\mathbb Z} \|F_j-F_j '\|_{\infty} \leq C_{\epsilon, n} \left (\int\limits_{\T}\|B-B'\|_2 ^2 \right )^{\frac 1 2} 
\end{equation}
for all $B, B'$ in $\mathbf{S}^{\epsilon}$. 
\end{theorem}
An alternate characterization of the sequence $F$ in Theorem \ref{thm:NLFT_outer_inverse} is that it is the unique $F \in \ell^2 (\Z; \con)$ for which $B$ is the upper right block of $\mathcal{F}_{\alpha} (F)$ and the nonlinear Plancherel identity \eqref{eq:Plancherel_QSP} holds, since  \eqref{eq:Plancherel_QSP} can hold if and only if the diagonal blocks of $\mathcal{F}_{\alpha} (F)$ are outer functions. It is an interesting open question whether the $\ell^\infty$-norm on the left side of  \eqref{eq:Lip_bds} can be upgraded to the $\ell^2$-norm.

We emphasize that there are elements in $\mathbf{H}_\alpha ^+$ such that $A^*$ is not outer, and hence not  every element of
$\ell^2(\Z_{\ge 0}; \con)$
arises in Theorem \ref{thm:NLFT_outer_inverse}. We also point out that the triangular nature of the diagonal $n \times n$-blocks of elements of $\mathbf{S}$ allows for a more precise, component-wise version of the Plancherel identity; see Remark \ref{rem:comp plan}.

 In Appendix \ref{section:qsv}, we relate the NLFT to  QSP over $SU(2^n)$, and give an application to multivariate QSP.

As this paper proposes a higher dimensional model of the NLFT, we are forced to make several notational choices in our exposition. For the convenience of the reader, we include a glossary at the very end of the paper.

\section{Analytic matrix-valued functions and the proof of Theorem \ref{factorization-theorem}}\label{sec:B_to_Y_outer}

\subsection{Inner and outer functions}

Here, we outline the basic theory of inner and outer functions. Useful references for this section are \cite{garnett} (scalar case) and \cite{Roos} (matrix case). 

We discuss the scalar case first. Recall that any bounded analytic function on $\D$ has pointwise a.e.\@ defined boundary values on the unit circle $\T$ \cite[Theorem 3.1, Chapter 2]{garnett}. An  inner function is a bounded analytic function $i$ on $\D$ such that $|i(z)|=1$ for almost every $z$ on the unit circle $\T$. Inner functions may be further factored into a  Blaschke product, that is, a convergent product of functions of the form 
\[b_j(z)=\frac{z_j-z}{1-z\overline{z_j}}\] with $z_j\in\mathbb D$, and a singular inner function $s$, that is, a nonvanishing inner function.  A bounded analytic function $o$ on $\D$ is called outer if
\begin{equation}\label{eq:outer_defn}
\log |o (0)|=  \int\limits_{\T} \log |o| \, .
\end{equation}
The inner-outer factorization for scalar-valued functions 
\cite[Corollary 5.7]{garnett} ensures that any bounded analytic function $f:\D\to\C$ factors as
\[f=b_f \cdot s_f \cdot o_f \]
where $b_f$ is a Blaschke product, $s_f$ is a singular inner function, and $o_f$ is an outer function. All three functions are unique up to multiplication by a unimodular constant. Inner-outer factorization is analogous to the polar representation of a complex number, with inner functions carrying the phase information, and outer functions carrying the modulus information.

%\ma{definition of contractive here interferes with how we defined contractive in the introduction, where we asked for strict inequality. Maybe write ``weakly contractive'' ?} \lb{How about `$1$-bounded'. And is this really used anywhere except in the sentence about the determinant below? We can just write $\|A\|_{H^\infty} \le 1$ there. Also, should we use the notation $\|\cdot\|_{H^\infty}$ or just $\|\cdot\|_\infty$?}

%{\it contractive} if $A(z)A^\ast(z)\leq \Id$ for all $z\in\mathbb D$ (or, equivalently, if $\|A(z)\|_{\infty}\leq 1$ for all $z\in\mathbb D$), and 
%The space of bounded analytic mvfs on $\mathbb D$ whose determinant does not vanish identically will be (temporarily) denoted by $\mathcal H^\infty$, and its subspace of contractive mvfs by $\mathcal S$.

We now to turn to the matrix case. An inner function $I$ is an element of the unit ball of $H^{\infty} (\D)$, that is, 
\begin{equation}\label{eq:unit_ball}
\|I\|_{H^\infty}:=\sup_{z\in\mathbb D} \|I(z)\|_{\infty} \leq 1 \, ,
\end{equation}
which has unitary boundary values a.e.\@ on $\mathbb T$, that is, 
\[
I I^* = \Id \, .
\]
It turns out that any $I \in H^{\infty}(\D)$ satisfying \eqref{eq:unit_ball} is inner if and only if $\det I$ is a scalar inner function; see \cite[Theorem 4.10]{Roos}.
%Analogous to the scalar case, a {\it Blaschke--Potapov factor} is an inner analytic mvf  of the form 
%\begin{equation}
%\label{eq:Blaschke_factor}
%b(z)=\Id-\mathcal P+\frac{z_0-z}{1-z\overline{z_0}} \mathcal P  \, ,
%\end{equation}
%for some $z_0\in\mathbb D$ and some orthogonal projection $\mathcal P$.
In this vein, we call an analytic matrix-valued function (mvf) outer if $\det A$ is a scalar outer function.
As discovered by Ginzburg \cite{ginzburg} and  explained in detail in \cite[Ch.~4]{Roos}, a bounded analytic mvf can 
factored into an inner mvf and an outer mvf, both unique up to a constant unitary factor.

% This will be used in Lemma \ref{lem plancherel} below.

\subsection{Spectral Factorization}
Let $\weights$ denote the set of measurable mvfs $\weight$ on $\T$ which are pointwise a.e.\ positive definite and hermitian, and  which satisfy
\begin{equation}\label{eq:Szego_std}
\int\limits_{\T} \log \det \weight  > - \infty \, .
\end{equation}
Equation \eqref{eq:Szego_std} is sometimes called the Szeg\H o condition in the complex analysis literature. We take $\weight = \Id - B B^*$ in \eqref{eq:Szego_condition}, hence we also refer to the latter as  the Szeg\H o condition. 

The following lemma is known as the F\'ejer--Riesz, or spectral factorization, theorem for mvfs.
It states that any positive mvf is the ``modulus'' of an outer mvf. In order to obtain the normalization needed for this paper, we use the QR factorization of Lemma \ref{lem:QR}. For more details on spectral factorization, see the exposition of  \cite{Barclay_cty_spec}.
\begin{lemma}[\cite{wiener}]\label{lem:outer_func_construction}

%Let $\weight$ be a positive definite measurable hermitian matrix-valued functions $\weight$ on $\T$ with
%\[
%\int\limits_{\T} \log \lvert\det \weight \rvert > - \infty \, .
%\]
Let $a \in \mathbb Z_2$. If $\weight \in \weights$, then there exists a unique outer mvf $\outr_a $ on $\D$ for which
   $\outr _a (0) \in \grp_a$ and the boundary values of $\outr _a$ satisfy
   \begin{equation}\label{eq:spec_fact}
   \weight = \outr _a ^* \outr_a .
   \end{equation}
\end{lemma}

%\begin{theorem}[\cite{Barclay_cty_spec},\cite{HeLo58}]
%Let $w$ be a positive matrix-valued function in $L^1 (\T)$. There exists an analytic mvf $b \in H^2(\D)$ with $\det b(0)\neq 0$ for which $w=b^\ast b$ if and only if  
%\begin{equation}\label{eq:Szego_w}
%\int_{\mathbb T} \textup{tr} \log w >-\infty \, .
%\end{equation}
%If \eqref{eq:Szego_w} holds, we can choose $b$ so that 
%\[\int_{\mathbb T} \log\lvert\det b\rvert =\log\lvert\det b(0)\rvert.\]
%\end{theorem}

\begin{proof}
To see uniqueness, let $\outr_a$ and $\outr_a '$ both satisfy \eqref{eq:spec_fact}, and assume both belong to $\grp_a$ at the origin. Then $U:=  \outr_a ' \outr_a ^{-1}$ must be an outer mvf, and  unitary on $\T$. It follows that $U$ is constant. Evaluation at the origin yields $U = \outr_a (0) ^{-1} \outr_a ' (0) $ is triangular with positive diagonal entries. The only  unitary matrix possible is $U=\Id$, and so $\outr_a = \outr_a '$.

As for existence, by \cite[Theorem 3.3]{Barclay_cty_spec} (originally proved by Masani--Wiener \cite{wiener}), there exists an outer mvf $\outr$ satisfying \eqref{eq:spec_fact}. By outerness, $\outr$ is invertible at the origin. Thus QR factorization yields $\outr(0)=QR$, where  $Q$ is unitary and  $R \in \grp_a$. 
Define ${\outr}_a:=Q^{-1}\mathcal O$, and note $\outr_a (0) = R$. Since $Q$ is unitary, using the notation $Q^{-*} := (Q^*)^{-1}$,
    \[\outr_a ^\ast \outr_a=\mathcal O^\ast Q^{-\ast} Q^{-1}\mathcal O=\mathcal O^\ast \mathcal O=\weight\, .\]
This completes the proof of existence. 
\end{proof}

In what follows, for $p\in\{1, 2, \infty\}$, the $L^p$ norm of an mvf $T$ is given by
\begin{equation}\label{defn:Lp_norm_def}
\|T\|_{L^p} ^p : = \int\limits_{\T} \|T(z)\|_p ^p \,  \, ,
\end{equation}
where when $p=1$, the trace norm $\|\nlf\|_{1}$ of $\nlf$ is defined as the sum of its singular values.

%Given $\epsilon> 0$, define the space $\weights_{\epsilon}$ to be the set of $\weight \in \weights$ for which the eigenvalues of $\weight$ are all between $\epsilon$ and $\epsilon^{-1}$. We endow $\weights_{\epsilon}$ with the $L^1 (\T; \C^{n^2})$ topology.

Given $P$ as in Lemma \ref{lem:outer_func_construction} and $a  \in \Z_2$, denote by  $\outr_a(P)$ the outer mvf  described by Lemma \ref{lem:outer_func_construction}.
    Spectral factorization is not continuous in $L^1$ since convergence $\|\weight_m-\weight\|_{L^1}\to 0$ does not imply $\|\outr_a (\weight_m)-\outr_a (\weight)\|_{L^2}\to 0$. 
The latter convergence does hold if we further assume $\|\log\det \weight_m - \log\det \weight\|_{L^1}\to 0$ (or any of the equivalent conditions listed in \cite[Prop.\@ 4.2]{Barclay_cty_spec}); see \cite[Theorem 3.5]{Barclay_cty_spec}. The following lemma shows spectral factorization satisfies an $L^1 \to L^2$ Lipschitz bound under the further assumption that  the eigenvalues of $\weight$ are bounded above and below. 
\begin{lemma}[\cite{Barclay_cty_spec, ESS_Quant_cty_spec}] \label{lem:cty_spec_factors}
Let  $\epsilon>0$ and $a \in \Z_2$.
There exists a constant $C_{\epsilon, n}$ such that for all $\weight, \weight' \in \weights$ whose eigenvalues lie in $[\epsilon, \epsilon^{-1}]$ a.e., we have the Lipschitz  bound 
\begin{equation}\label{eq_Lip}
\|\outr_a (\weight) - \outr_a (\weight')\|_{L^2} \leq C_{\epsilon, n} \| \weight - \weight ' \|_{L^1} \, . 
\end{equation}

\end{lemma}

\begin{proof}[Proof of Lemma \ref{lem:cty_spec_factors}]
Let $a,\epsilon,\weight,\weight'$ be given as in the lemma.
The Lipschitz bound \eqref{eq_Lip} with the canonical $\outr$ of \cite{Barclay_cty_spec} in place of $\outr_a$ follows from \cite[Theorem 1.5]{ESS_Quant_cty_spec}. Indeed, our assumptions imply that $\weight$ and $\ell_{\weight}:=\log\det \weight-n\log_+\|\weight\|_{\infty}$ are bounded (with bounds depending on $\varepsilon, n$), and that 
%as long as $\weight\in\mathcal \weights_\varepsilon$, and that
\[\left\|\log\frac{\det \weight'}{\det \weight}\right\|_{L^1}\leq C_{\varepsilon,n} \|\weight'-\weight\|_{L^1}.\]
 The Lipschitz bound \eqref{eq_Lip} then follows from Lipschitz continuity of the QR factorization proved in Lemma \ref{lem:QR}. 
\end{proof}

\subsection{Determinants of unitary block matrices}
As in \eqref{labeling_M_intro}, we will often label the blocks  of a given $2n \times 2n$ mvf $\nlf$ on $\T$  as
\begin{equation}\label{eq:convention}
\begin{pmatrix}
    A & B \\ C & D
\end{pmatrix} := \nlf \, .
\end{equation}
%In this paper, 
This notational convention extends to mvfs  $\nlf_{-}$, $\nlf_{+}$, $\nlf'$ , $\nlf_j$, $\ldots$, whose upper left blocks will be labeled $A_{-}, A_{+}, A', A_{j}, \ldots  $, respectively. In what follows, we let $U(m)$ denote the set of $m \times m$ unitary matrices.

%The following lemma gives us the determinant of a block unitary matrix in terms of its diagonal blocks.
\begin{lemma}\label{lem:det_holo_formula}
    If $\nlf$ is an a.e.\ $U(2n)$-valued matrix function on $\T$, with $A, D$ invertible a.e.\ on $\T$, then  
    a.e.\ on $\T$ we have
    \begin{equation}\label{eq:det_formula_Y}
    \det \nlf = \frac{\det D}{\det A^*} = \frac{\det A}{\det D^*}\, . 
    \end{equation}
\end{lemma}
\begin{proof}
Since $\nlf$ is unitary, we have $A^*B+ C^*D=0$, or equivalently,  
\begin{equation}\label{eq_CstarForLater}
 -A^* B D^{-1}=C^* \, .
\end{equation}
We compute
\[
\nlf
\begin{pmatrix}
        \Id  & 0 \\ -D^{-1}C & \Id
    \end{pmatrix}  = \begin{pmatrix}
        A & B \\ C & D
    \end{pmatrix}\begin{pmatrix}
        \Id  & 0 \\ -D^{-1}C & \Id
    \end{pmatrix} = \begin{pmatrix}
        A - BD^{-1}C & B \\ 0  & D 
    \end{pmatrix} \, .
\]
Taking determinants of both sides yields
\[
\det \nlf = \det (D) \det (A - B D^{-1} C) = 
 \frac{\det D}{\det A^*} \det (A^* A + C^* C) = \frac{\det D}{\det A^*} \, ,
\]
where we inserted $\Id=A^{-*}A^*$ and recalled \eqref{eq_CstarForLater} in the second step, and used again that $\nlf$ is unitary in the last step. The second equality in \eqref{eq:det_formula_Y} follows similarly.
\end{proof}
%The next lemma shows that if the matrix $Y$ in Lemma \ref{lem:det_holo_formula} has diagonal blocks which are appropriate outer functions, then $\det Y = 1$. 

%Old version:

\begin{lemma}\label{lem:det_1_Y}
    Let $\nlf$
    be an a.e.\ $U(2n)$-valued matrix function on $\T$. If $A^*$ and $D$ are outer mvfs on $\D$ with positive determinant at $z=0$, then $\nlf$ is a.e.\ $SU(2n)$-valued.  
\end{lemma}
\begin{proof}
We must show $\det \nlf=1$ a.e.\ on $\T$.
By Lemma \ref{lem:det_holo_formula},  
\[
\det \nlf=\frac{\det D}{\det A^*} \, .
\]
Thus $\det \nlf$ extends to an outer function on $\D$ which is positive at $z=0$. On $\T$, it has modulus $1$, because 
\[
\left|{\det \nlf}\right|^2=\frac{\det D^* D}{\det A A^*} = \frac{\det (\Id  - B^* B)}{\det (\Id - B B^*)}  = 1 \, ,
\]
where the last equality followed from the fact that  $\Id - B B^*$ and $\Id - B^* B$  are positive and have the same singular values. But any outer function with modulus $1$ on $\T$ must be constant, and combined with the positivity of $\det M$ at $0$, we get $\det M=1$ everywhere. 
\end{proof}

Later on in the proof of Lemma \ref{lem:injectivity_half_line} and near \eqref{outer_func_mean_value_fails}, we will also need the following result for inner mvfs.

\begin{lemma}\label{lem:cst_det_matrix_inner}
    Let $I$ be an inner mvf. If $\det I$ is constant, then $I$ is constant.
\end{lemma}
\begin{proof}
   Since $\det I$ is constant on  $\D$, then $I^{-1}$ is bounded and analytic on $\D$ by the adjugate\footnote{Recall $(\adj I)_{i j} := (-1)^{i +j} M_{j i}$, where $M_{xy}$ denotes the $xy$-th minor of $I$.} formula
  \begin{equation}\label{eq:adjugate_2}
I^{-1} = \frac{1}{\det I} \adj I \, .
\end{equation}
    Because  $I^{-1}$ agrees with $I^*$ on $\T$, we get that $I^*$ extends to a bounded analytic function on  $\D$. But for $I$ and $I^*$ to both extend to bounded analytic functions on $\D$, we must then have that $I$ is constant.   
\end{proof}

\subsection{Proof of Theorem \ref{factorization-theorem}}

We are now ready to prove Theorem \ref{factorization-theorem}. Fix $\alpha : \Z_2 \to \Z_2$.

%Denote by $\tilde{\mathbf{B}}$ the set of measurable $U(2n)$
%valued matrix functions
% \[
%    \begin{pmatrix}
%    A & B \\ C & D
%\end{pmatrix}
%\]
%such that $B$ and $C$ are Szeg\H o and $A^*$ and $D$ are outer and $A(\infty)$
%is in $\mathcal{G}_{\alpha(0)}$
%and $D(0)$
%is in $\mathcal{G}_{\alpha(1)}$.
%Note that unlike the definition of $\mathbf{B}_{\alpha}$, we do not require that elements of $\tilde{\mathbf{B}}$ have determinant $1$.

We begin with existence. Given $B \in \mathbf{S}$, define $A^*$ and $D$ to be the unique solutions of the spectral factorization problems 
\begin{equation}\label{eq:spec_fact_A}
A A^* = \Id - B B^* \, , \qquad A^*(0) \in \grp_{1-\alpha(0)} \, , 
\end{equation}
and
\begin{equation}\label{eq:spec_fact_D}
 D ^* D   = \Id- B^* B \, , \qquad D(0) \in \grp_{\alpha(1)}  \, ,
\end{equation}
whose existence and uniqueness is guaranteed by Lemma \ref{lem:outer_func_construction}. Then define 
\begin{equation}\label{eq:C_unique}
C := -D ^{-*} B ^* A  \, , 
\end{equation}
and finally set \begin{equation}\label{eq:sety}
    Y_{\alpha} (B) := \begin{pmatrix}
    A & B \\ C & D
    \end{pmatrix} \, .
\end{equation}
We have
\[
    Y_{\alpha} (B)^* Y_{\alpha} (B) =  \begin{pmatrix} A^*A + C^*C & A^*B + C^*D\\ B^*A + D^*C & B^* B + D^* D\end{pmatrix}.
\]
% \[
% Y (B) Y (B)^*  = \begin{pmatrix}
%     A A^* + B B^* & A C^* + B D^* \\ C A^*  + D B^* & C C^* + D D^* 
% \end{pmatrix} \, . 
% \]
The bottom right block equals the identity matrix by \eqref{eq:spec_fact_D}. The off-diagonal blocks are zero by \eqref{eq:C_unique}. For the upper left block, we compute with
\eqref{eq:C_unique}
\begin{equation}
    \label{eq:upperleft}
    A^*A + C^*C = A^*(\Id + B D^{-1} D^{-*} B^*) A     
\end{equation}
To show that this equals the identity matrix, it suffices to show that
\begin{equation}\label{eq:goal_id}
\Id + B D^{-1}D^{-*} B^* =(AA^*)^{-1} \, .
\end{equation}
We compute with \eqref{eq:spec_fact_A} and \eqref{eq:spec_fact_D}
\[
 (\Id+B(D^* D)^{-1} B^*) A A^*= (\Id+B(D^* D)^{-1} B^*)(\Id-BB^*)\]\[
=\Id -BB^*+B(D^* D)^{-1} (\Id-B^*B)B^*=\Id \, ,
\]
yielding \eqref{eq:goal_id}. Thus $Y_{\alpha}(B) \in U(2n)$. 
% \[
% A^*  B D^{-1} (D^*)^{-1} B^* A = A^* (A A^*)^{-1} A - A^* A = \Id - A^* A
% \]
% Plugging in the definition of $C$ into the above yields
% \[
% Y(B) Y(B) ^* = \begin{pmatrix}
%     A A^* + B B^* & -A A^* B D^{-1} + B D^* \\ -(D^*)^{-1} B^* A A^*  + D B^* & (D^*)^{-1} B^* A  A^* B D^{-1} + D D^* 
% \end{pmatrix} \, . 
% \]
% By Lemma \ref{lem:outer_func_construction} and \eqref{def:A_D_outer}, we have 
% \begin{equation}\label{eq:A_to_B}
% A A^* = \Id - B B^* \, , 
% \end{equation}
% and
% \begin{equation} \label{eq:D_to_B}
% D^* D = \Id - B^* B \, . 
% \end{equation}
% By \eqref{eq:spec_fact_A}, the top left block equals $\Id$. The top right block equals
% \[
% (-A A^* B + B D^* D) D^{-1} \, ,
% \]
% which by substituting for $A A^*$ and $D^* D$ using \eqref{eq:A_to_B} and \eqref{eq:D_to_B},
% equals
% \[
% [-(\Id -B B^*) B + B(\Id - B^* B)] D^{-1} = 0 \, .
% \]
% The bottom left block similarly vanishes. As for the bottom right block, substituting again for $A A^*$ using \eqref{eq:A_to_B}, we may write it as
% \[
% (D^*)^{-1} B^* B D^{-1}  -  (D^*)^{-1} B^* B B^* B D^{-1}+ D D^*
% \]
% \[
% = (D^*)^{-1} B^* B (\Id-B^* B)  D^{-1}  + D D^* \, .
% \]
% Substituting for $\Id -  B^* B$ using \eqref{eq:D_to_B},  
% the above becomes
% \[
% (D^*)^{-1} B^* B D^* D  D^{-1}  + D D^* =  (D^*)^{-1} B^* B D^*  + D D^* 
% \]
% \[
% = (D^*)^{-1}[ B^* B + D^* D] D^* = (D^*)^{-1} D^* = \Id \, .
% \]
% Thus, $Y(B)$ is unitary on $\T$.
We are left with checking $\det Y_{\alpha}(B) = 1$. Because $A(\infty) \in \grp_{\alpha(0)}$ and $D(0) \in \grp_{\alpha(1)}$, they have positive determinant. Combining this with outerness of  $A^*$ and $D$, the claim that $\det Y_{\alpha} (B) =1$ now follows from Lemma \ref{lem:det_1_Y}. This concludes the proof of existence.

We now prove uniqueness. For a matrix 
\eqref{eq:sety} to be unitary a.e.\ on $\T$, we necessarily have  a.e.\ 
\[
Y_{\alpha} (B)Y_{\alpha} (B)^*=Y_{\alpha} (B)^*Y_{\alpha} (B)=\Id \, ,
\]
from which \eqref{eq:spec_fact_A}, \eqref{eq:spec_fact_D} and \eqref{eq:C_unique} follow. Since the outer mvfs $A^*$ and $D$ solve the spectral factorization problems \eqref{eq:spec_fact_A} and \eqref{eq:spec_fact_D}, respectively, then Lemma \ref{lem:outer_func_construction} implies $A$ and $D$ are unique. And then \eqref{eq:C_unique}, which follows from unitariness of $Y_{\alpha} (B)$, implies $C$ is unique. This concludes the proof of Theorem \ref{factorization-theorem}. 

%Finally, we prove the claim on the polynomial degree. If $B \in \mathbf{S}$ is a trigonometric polynomial of degree at most $d$, then by Lemma \ref{lem:outer_func_construction} and \eqref{eq:spec_fact_A} and \eqref{eq:spec_fact_D}, $A^*$ and $D$ are analytic polynomials of degree at most $2d$. 

%As for $C$, by \eqref{eq:C_unique} the matrix function  
%\[
% z^{d} C^* = -A^* (Bz^d) D^{-1} 
%\]
%must be analytic on $\D$, since $A^*$, $B z^d$ and $D^{-1}$ are all analytic on $\D$, the last of which follows from the fact that $D$ is outer. This in particular shows that the bounded measurable function $C$ on $\T$ has frequency support in $(-\infty, d]$.
%
%Similarly, because $Y_{\alpha} (B) Y_{\alpha} (B)^* = \Id$, the matrix function
%\[
%Cz^{d} = -D (z^{-d}B)^* A^{-*}
%\]
%is analytic on $\D$ because $D, A^{-*}$ and $z^{d} B^*$ are all analytic on $\D$ as well. This implies $C$ has frequency support in $[-d, \infty)$.
%
%Thus $C$ is a trigonometric polynomial of degree at most $d$, and this concludes the proof of Theorem \ref{factorization-theorem}.

\section{The NLFT and basic properties}

\subsection{Cholesky and \texorpdfstring{$QR$}{QR} factorization}

Recall that given a positive definite hermitian matrix $\weight$ and $a \in \Z_2$, there exists a unique $U \in \grp_{\alpha}$ for which
\begin{equation}\label{eq:cholesky}
\weight = U ^* U   \, .  
\end{equation}
This follows from Lemma \ref{lem:outer_func_construction} for the constant matrix function $\weight$.
The factorization \eqref{eq:cholesky} is called the Cholesky factorization \cite[Section 9.4]{schatzman2002numerical}.  

\begin{lemma}\label{lem_cholesky}
    For each $\epsilon > 0$ and $a \in \Z_2$, the map from $P$ to $U$ as in the Cholesky factorization \eqref{eq:cholesky} is Lipschitz continuous on the set of positive definite hermitian matrices with all eigenvalues in the interval $[\epsilon, \epsilon^{-1}]$.
\end{lemma}

\begin{proof}
    The map $\Phi$ sending $U$ to $ U^* U$ from $\grp_a$ to the set of positive definite hermitian matrices is a quadratic polynomial in the entries and one-to-one.
    Hence its derivative $D \Phi$ has full rank at every point. For if not, there is  a point $U$ and a triangular matrix $V$ for which  
\[ \left. \partial_t\Phi(U+tV)\right|_{t=0} =0.\]
But then $\Phi(U+tV)$ is a quadratic polynomial in $t$ with vanishing linear term, defined on a small neighborhood of $t=0$. In particular it is even and thus not injective. This contradicts injectivity of $\Phi$.

The inverse function theorem now implies that the inverse function $\Phi^{-1}$ is continuously differentiable, hence Lipschitz continuous on the compact set of positive hermitian matrices with all eigenvalues in the interval $[\epsilon, \epsilon^{-1}]$.
\end{proof}

We remark that the eigenvalues and the singular values of a positive definite hermitian matrix coincide. As a corollary of Lemma \ref{lem_cholesky}, we record the following standard facts about the $QR$ factorization \cite[Section 12.1]{schatzman2002numerical}.  

\begin{lemma}\label{lem:QR}
   Let $a \in \Z_2$. For every invertible matrix $A$, there exists a unique $Q \in U(n)$ and a unique $R \in \grp_{a}$ such that
    \[
    A = QR.
    \]
    The map $A \mapsto (Q, R)$ is Lipschitz continuous on the set of matrices $A$ with all singular values in the interval $[\epsilon, \epsilon^{-1}]$.
\end{lemma}

\begin{proof}
  Let $A\in\mathcal M$ be an invertible matrix. 
  
  We first show existence of $R$ and $Q$.
  The matrix $A^*A$ is positive definite hermitian, and so 
  there exists a unique Cholesky factorization, $A^\ast A=R^\ast R,$ with $R \in \grp_a$. In particular, $R$ is invertible. Let $Q:=AR^{-1}$.
  Then $Q$ is unitary:
 \[Q^\ast Q=R^{-\ast}A^\ast A R^{-1}=R^{-\ast} R^\ast R R^{-1}=\Id.\]
%This proves existence. 

We next show uniqueness. Assume we have two factorizations $A=Q_1R_1=Q_2R_2$ as in the statement of the lemma. Then, for $j\in \{1,2\}$,
\[R_j^\ast R_j = R_j ^\ast Q_j ^\ast Q_j R_j= A^\ast A\ ,\]
yielding a Cholesky factorization of $A^* A$. But the Cholesky factorization is unique, so $R_1=R_2$,  hence $Q_1=Q_2$ as well.

Finally, note that the maps $A\mapsto A^\ast A\mapsto R$ are locally Lipschitz continuous, as  is $A\mapsto AR^{-1}$. This finishes the proof since the space of matrices $A$ with all singular values in $[\epsilon, \epsilon^{-1}]$ is compact.
\end{proof}

\subsection{Basic properties of the nonlinear Fourier transform}

For $\alpha :\Z_2 \to \Z_2$, let $SU_{\alpha} (2n)$ denote the set of all matrices
\begin{equation}\label{eq:label_SU_alpha}
\begin{pmatrix}
    E & F \\ G & H
\end{pmatrix} \in SU(2n) 
\end{equation}
for which $E \in \grp_{\alpha(0)}$ and $H \in \grp_{\alpha(1)}$. 
%Throughout this paper, we always denote elements of $SU_{\alpha}$ using \eqref{eq:label_SU_alpha}.
Theorem \ref{factorization-theorem} applied with constant matrix functions provides a parametrization of $SU_\alpha(2n)$  by the upper right block $F$ in  \eqref{eq:label_SU_alpha}.
We denote by $Y_\alpha(F)$ the unique $SU_\alpha(2n)$ matrix with upper right block $F$. When $\alpha$ is evident from context, we also adopt the notational convention 
\begin{equation}\label{eq:notation_conv}
\begin{pmatrix}
    E_j & F_j \\ G_j & H_j
\end{pmatrix}:= Y_{\alpha} (F_j),
\end{equation}
where $F_j$ denotes an $n \times n$ contractive matrix; the subscript $j$ may sometimes be dropped. Finally, we will use  the notation
\begin{equation}\label{defineAd}
    \ad (z) X:=ZXZ^{-1}\,,  
\end{equation}
where $Z$ is the block version of \eqref{defineZ}.

Similarly to \cite[Lemma 1]{TT03}, \cite[Lemma 2.2]{tsai} and \cite[Theorem 2]{Alexis+24}, the following algebraic properties and symmetries of the NLFT hold.

\begin{lemma}
\label{lem:properties}
Let $\alpha:\Z_2 \to \Z_2$ and let $F=(F_j)_{j\in\mathbb Z}$ denote a finitely supported sequence of contractive matrices. The following properties of the \textup{NLFT}  hold.
\begin{enumerate}[label=\alph*)]
\item \label{item:dirac} \underline{Dirac delta sequence}: If $F_j=0$ for all $j\neq 0$, then  
\[
\mathcal{F}_{\alpha} (F) = Y_\alpha(F_0).
%\begin{pmatrix}
%    E_m & F_m z^m \\ G_m z^{-m} & H_m
%\end{pmatrix}
\]
\item \label{item:ordered_mult} 
\underline{Ordered multiplicativity}: If the support of $F$ is entirely to the left of the support of $F'$, then
\[
\mathcal{F}_{\alpha} (F+F') = \mathcal{F}_{\alpha} (F) \mathcal{F}_{\alpha} (F') \, .
\]
\item \label{item:cconjugation}
\underline{Complex conjugation}: For the  conjugated sequence $\overline{F}$ we have 
\[\mathcal{F}_{\alpha}(\overline{F})(z)=\overline{\mathcal{F}_{\alpha} (F) (\overline z)} \, .\]

\item \label{item:translation} \underline{Translation}: Let $m\in\mathbb Z$. Define the translation map
\[
    T: (F_j)_{j \in \mathbb{Z}} \mapsto (F_{j-1})_{j \in \mathbb{Z}}.
\]
Then 
\[
    \mathcal{F}_{\alpha} \circ T^m = \ad(z)^m \circ \mathcal{F}_{\alpha}.
\]

\item \label{item:phaserotation} 
\underline{Phase rotation}: If $c\in\mathbb \T$, then 
\[
\mathcal{F}_{\alpha} (cF) = \ad(c) \circ \mathcal{F}_{\alpha} (F).
\]

\item \label{item:modulation} \underline{Modulation}: Given $\theta\in\mathbb R$, for the modulated sequence $(e^{ij\theta}F_{j})_{j\in \Z}$ we have
\[
\mathcal{F}_{\alpha} ((e^{ij\theta}F_{j})_{j \in \Z} )(z)
=\mathcal{F}_{\alpha} (F) (e^{i\theta}z).
\]
\item \label{item:matrixconjugation} 
\underline{Matrix conjugation}: If $U$ is unitary and diagonal, then 
\begin{equation}\label{eq_Uconj}
    \mathcal{F}_{\alpha} (UFU^{-1}) =\begin{pmatrix}
    U & 0 \\ 0 & U
\end{pmatrix}\mathcal{F}_{\alpha} (F) \begin{pmatrix}
    U^{-1} & 0 \\ 0 & U^{-1}
\end{pmatrix}.
\end{equation}
\end{enumerate}

\end{lemma}
\begin{proof}
Properties \ref{item:dirac} and \ref{item:ordered_mult} follow directly from the definition of the NLFT. Property \ref{item:cconjugation} follows by observing that $Y_{\alpha} (\overline{F}) = \overline{Y_{\alpha} (F)}$. Property \ref{item:translation}  follows from 
\[
    \ad(z) \mathcal{F}_\alpha(F) = \ad(z) \Big[\prod_{j \in \mathbb{Z}}^{^{\curvearrowright}} \ad(z)^j Y_\alpha(F_j) \Big]
    = \prod_{j \in \mathbb{Z}}^{\curvearrowright} \ad(z)^{j+1} Y_\alpha(F_j)
    \]
    \[
    = \prod_{j \in \mathbb{Z}}^{\curvearrowright} \ad(z)^{j} Y_\alpha(F_{j-1}),
\]
using that $\ad(z)$ is a group homomorphism, and induction on $m$.
Similarly, property \ref{item:phaserotation} holds by observing for $c \in \mathbb{T}$
\begin{equation}\label{eq:modulate_F}
\ad(c) Y_\alpha(F) = Y_\alpha(cF).
\end{equation}
For property \ref{item:modulation}, we observe using \eqref{eq:modulate_F} that
\[
\mathcal{F}_{\alpha} (F) (e ^{i \theta} z) = \prod\limits_{j\in\mathbb Z} ^{\aarrow} \ad(e^{i \theta} z) ^j Y_{\alpha} (F_j) = \prod\limits_{j\in\mathbb Z} ^{\aarrow} \ad(z) ^j  \left [ \ad (e^{i \theta}) ^j Y_{\alpha} (F_j) \right ] 
\]
\[
=  \prod\limits_{j\in\mathbb Z} ^{\aarrow} \ad(z) ^j    Y_{\alpha} (e ^{i j \theta} F_j) = \mathcal{F_{\alpha}} ( (e^{i j \theta} F_j)_{j \in \Z}) \, . 
\]
Property \ref{item:matrixconjugation} is proved similarly. Indeed, with a slight abuse of notation, we may formally write
\[
\ad (U) Y_{\alpha} (F_j) = Y_{\alpha} (U F_j U^{-1}),
\]
and hence
\[
\mathcal{F}_{\alpha} (U F U^{-1}) = \prod\limits_{j \in \Z} ^{\aarrow} \ad (z) ^j Y_{\alpha} (U F_j U^{-1}) =  \prod\limits_{j \in \Z} ^{\aarrow} \ad (z) ^j \ad (U) Y_{\alpha} (F_j ) 
\]
\[= \ad (U) \Big[\prod\limits_{j \in \Z} ^{\aarrow} \ad (z) ^j  Y_{\alpha} (F_j ) \Big]= \ad (U) \mathcal{F}_{\alpha} (F) \, ,
\]
which we identify with the right side of \eqref{eq_Uconj}.
\end{proof}

The following reflection symmetry intertwines the various NLFTs.
\begin{lemma}[Reflection symmetry]
     \label{item:reflection_2} 
Let $\alpha:\Z_2 \to \Z_2$ and let $F=(F_j)_{j\in\mathbb Z}$ denote a finitely supported sequence of contractive matrices. Let $F_{-} := (F_{-j})_{j \in \Z}$ denote the reflected sequence. Then
\[
\begin{pmatrix}
    0 & \Id \\ \Id & 0 
\end{pmatrix}\mathcal{F}_{\alpha} (F) ^* \begin{pmatrix}
    0 & \Id \\ \Id & 0 
\end{pmatrix} = \mathcal{F}_{\beta} (F_{-} ^* ) \, , 
\]
where $\beta(x) := 1-\alpha(1-x)$. 
\end{lemma}
\begin{proof}
We compute 
\[
\begin{pmatrix}
    0 & \Id \\ \Id & 0 
\end{pmatrix} \mathcal{F}_{\alpha} (F)^* (z) \begin{pmatrix}
    0 & \Id \\ \Id & 0 
\end{pmatrix} =\begin{pmatrix}
    0 & \Id \\ \Id & 0 
\end{pmatrix} \left [ \prod\limits_{j \in \Z} ^{\aarrow} \ad (z) ^j Y_{\alpha} (F_j) \right ]^* \begin{pmatrix}
    0 & \Id \\ \Id & 0 
\end{pmatrix} 
\]
\[= \begin{pmatrix}
    0 & \Id \\ \Id & 0 
\end{pmatrix}  \prod\limits_{j \in \Z} ^{\aarrow} \left [\ad (z) ^{-j} Y_{\alpha} (F_{-j})  \right ] ^* \begin{pmatrix}
    0 & \Id \\ \Id & 0 
\end{pmatrix} 
\]
\[
=  \prod\limits_{j \in \Z} ^{\aarrow} \left [  \begin{pmatrix}
    0 & \Id \\ \Id & 0 
\end{pmatrix} \ad (z) ^{-j} \big[Y_{\alpha} (F_{-j}) ^*\big]  \begin{pmatrix}
    0 & \Id \\ \Id & 0 
\end{pmatrix} \right ]  
\]
\[
=  \prod\limits_{j \in \Z} ^{\aarrow}   \ad (z) ^{j} \Big[\begin{pmatrix}
    0 & \Id \\ \Id & 0 
\end{pmatrix} Y_{\alpha} (F_{-j}) ^*  \begin{pmatrix}
    0 & \Id \\ \Id & 0 
\end{pmatrix}\Big]  =  \prod\limits_{j \in \Z} ^{\aarrow}  \ad (z) ^{j} Y_{\beta} (F_{-j} ^*)   \, ,
\]
which we recognize as $\mathcal{F}_{\beta} (F_{-} ^*) (z)$.   
\end{proof}

 The NLFT defined in \eqref{def:NLFT_alpha} has the following Fourier analytic properties. 
 
\begin{lemma}\label{lem_MultilinearExpansion}
       Let $c,d\in\Z$.
       Let $F$ be a sequence supported on the interval $[c,d]$ and let $\alpha : \Z_2 \to \Z_2$. Then the $n \times n$ matrix functions $A, B,C, D$, defined by 
       \begin{equation}\label{eq:abcd}
      \begin{pmatrix}
          A & B \\ C & D 
      \end{pmatrix} := \mathcal{F}_{\alpha} (F) \, ,
       \end{equation}
       have frequency support on $[c-d, 0]$, $[c,d]$, $[-d,-c]$, $[0, d-c]$, respectively. Furthermore, the formulas 
    \begin{equation}\label{eq:constant_Fourier_A_D}
    A (\infty) = \prod\limits_{j\in\mathbb Z}^{\aarrow} E_j \, , \qquad 
    D(0) = \prod\limits_{j\in\mathbb Z}^{\aarrow} H_j  \,
    \end{equation}
    hold,   from which it follows that
    \begin{equation}
    \label{eq:constant_Fourier_A_D_2}
    A(\infty) \in \grp_{\alpha(0)} \, , \qquad D(0) \in \grp_{\alpha(1)} \, . 
    \end{equation}
    Furthermore, if $c=0$, then
    \begin{equation}\label{eq:layerstripping_1}
    B(0) D^{-1} (0) = F_0 H_0 ^{-1} \, , \qquad  C (\infty) A(\infty)^{-1} = G_0 E_0 ^{-1} \, .
    \end{equation}
\end{lemma}

\begin{proof}
By the translation 
property \ref{item:translation} of Lemma \ref{lem:properties},  and in particular noting that translating $F$ by $m$ to the left means multiplying $B$ and $C$ by $z^{-m}$ and $z^m$, respectively, we can assume without loss of generality that $c=0$.

We now proceed by induction on $d$. For the base case when $d = 0$, we have 
\[
\mathcal{F}_{\alpha}(F) (z) = \begin{pmatrix}
    E_0 & F_0  \\ G_0 & H_0
\end{pmatrix} \, .
\]
so all the conclusions of the lemma  hold.

Now suppose as our induction hypothesis that the claim holds for all nonnegative integers up to $d$. We show it must hold for $d +1$. By definition of the NLFT, if
%By Property \ref{item:ordered_mult} of Lemma \ref{lem:properties}, if 
% \[
% \begin{pmatrix}
%     A_{\Delta_{-}} & B_{\Delta_{-}} \\ C_{\Delta_{-}} & D_{\Delta_{-}} 
% \end{pmatrix} := \mathcal{F}_{\alpha} ( (F_j \mathbf{1}_{\{j  \leq \Delta \}} )_{j \in \Z} ) \, , \qquad \begin{pmatrix}
%     A_{\Delta^+} & B_{\Delta^+} \\ C_{\Delta^+} & D_{\Delta^+} 
% \end{pmatrix} := \mathcal{F}_{\alpha} ( (F_j \mathbf{1}_{\{j >\Delta \}} )_{j \in \Z} ) \, , 
% \]
\[
\begin{pmatrix}
    A_{d} & B_{d} \\ C_{d} & D_{d} 
\end{pmatrix} := \mathcal{F}_{\alpha} ( (F_j \mathbf{1}_{\{j  \leq d \}} )_{j \in \Z} ) \, ,  
\]
then
\[
\begin{pmatrix}
    A & B \\ C & D
\end{pmatrix}= \begin{pmatrix}
    A_{d} & B_{d} \\ C_{d} & D_{d} 
\end{pmatrix} \begin{pmatrix}
    E_{d+1} & z^{d+1} F_{d+1} \\ z^{-d-1} G_{d+1} & H_{d+1} 
\end{pmatrix}  
\]
\begin{equation}\label{eq:induction}
= \begin{pmatrix}
    A_{d} E_{d+1} + z^{-d-1} B_{d}  G_{d+1} & z^{d+1} A_{d}  F_{d+1} + B_{d} H_{d+1} \\ C_{d} E_{d+1} +  z^{-d-1} D_{d}  G_{d+1} & z^{d+1} C_{d}  F_{d+1} + D_{d} H_{d+1} 
\end{pmatrix} \, . 
\end{equation}
The claims about the frequency supports follow by inspection, using the induction hypothesis.
% By the induction hypothesis, $A_{\Delta_{-}}$ has frequency support in $[-\Delta, 0]$ and $B_{\Delta_{-}}$ has frequency support in $[0, \Delta]$, whereas $A_{\Delta_{+}}$ has frequency support at $\{0\}$ and $C_{\Delta_{+}}$ has frequency support at $\{-\Delta\}$. Thus \eqref{eq:induction} yields $A$ has frequency support on $[-(\Delta+1), 0]$. 

Noting that $z^{-d-1}B_{d} G_{d+1}$ has Fourier support on $[-d-1, -1]$, we have
\[
A (\infty) = A_{d} (\infty) E_{d+1} (\infty) = \left ( \prod\limits_{j=0}^{d} E_j \right ) E_{d+1}  \, , 
\]
which proves \eqref{eq:constant_Fourier_A_D} for $A$, and \eqref{eq:constant_Fourier_A_D} for $D$ follow similarly. 

% The claim about the frequency support of $B$ follows again by noting that $A_{\Delta_{-}} B_{\Delta_{+}}$ is the product of trigonometric polynomials with frequency support in $[-\Delta, 0]$ and $[\Delta+1]$ and hence has frequency support in $[1, \Delta+1]$, and similarly $B_{\Delta_{-}} D_{\Delta_{+}}$ is the product of functions with frequency support in $[0,\Delta]$ and $\{0\}$, and so has frequency support at $[0, \Delta]$. The claim for the frequency support of $C$ follows similarly. 

To see \eqref{eq:layerstripping_1}, we again use \eqref{eq:induction} to write
\begin{equation}\label{eq:BD_inv_induct}
BD^{-1} = (z^{d+1} A_{d}  F_{d+1} + B_{d} H_{d+1}) (z^{d+1} C_{d}  F_{d+1} + D_{d} H_{d+1})^{-1} \, .
\end{equation}
By the induction hypothesis, both $z^{d+1}A_{d}$ and $z^{d+1}C_d$ vanish at $0$, giving
\[
B(0)D^{-1} (0) = ( B_d(0)H_{d+1})(D_d(0) H_{d+1})^{-1} = B_{d} (0) D_{d}(0)^{-1}\, ,
\]
which by  induction shows the first equality of \eqref{eq:layerstripping_1}. The second equality of \eqref{eq:layerstripping_1} follows similarly.
\end{proof}

\subsection{The layer stripping algorithm}\label{sec_LayerStripping}
In this section we define layer stripping, that is, the sequential recovery of nonlinear Fourier coefficients from the NLFT. This is  sometimes known as {\it peeling} in the QSP literature.

We  need some preliminary lemmas. 
% Given $\weight\in \weights$, as defined just before Lemma \ref{lem:outer_func_construction}, and $a \in \Z_2$, set 
% \[
% \mathcal{Q}_{a} (\weight) := (\mathcal{O}_{1-a} (P) )^* \, ,
% \]
% where $\mathcal O$ was introduced in Lemma \ref{lem:outer_func_construction},
% so that
% \[
% P = \mathcal Q_{a} \mathcal Q_{a} ^*   \, , \qquad \mathcal Q_{a} (\infty) \in \grp_{1-a} \, .
% \]
Given a positive definite hermitian matrix $\weight \in \mathcal{M}$ and $a \in \Z_2$, denote by $\sqrtt_a(P)$ the unique solution of
\begin{equation}
    \label{e:sqrt}
    \sqrtt_a(P) \sqrtt_a(P)^* = P, \qquad \sqrtt_a(P) \in \grp_{a},
\end{equation}
that is, the upper or lower triangular Cholesky factor of $P$, depending on $a$. 
\begin{lemma} \label{lem:layer_strip_recover}
    Let $\alpha: \Z_2 \to \Z_2$. The map   \begin{equation}\label{eq:strip_cst_forward_2}
    \begin{pmatrix}
        E & F \\ G & H
    \end{pmatrix} \mapsto F H^{-1}
    \end{equation}
is a homeomorphism from $SU_{\alpha} (2n)$ onto $\mathcal{M}$, with inverse given by 
% \begin{equation}\label{eq:inverse_strip}
% S_{\alpha} (K):= \begin{pmatrix}
%   \mathcal{Q}_{\alpha(0)}  [ (\Id + K K^*)^{-1}]  & K \mathcal{Q}_{\alpha(1)}   [ (\Id + K^* K)^{-1} ] \\
%   - K^* \mathcal{Q}_{\alpha(0)} [ (\Id + K K^*)^{-1}]& \mathcal{Q}_{\alpha(1)}   [ (\Id + K^* K)^{-1} ]  
% \end{pmatrix} \, ,
% \end{equation}
\begin{equation}\label{eq:inverse_strip}
S_{\alpha} (K):= \begin{pmatrix}
  \sqrtt_{\alpha(0)}  [ (\Id + K K^*)^{-1}]  & K \sqrtt_{\alpha(1)}   [ (\Id + K^* K)^{-1} ] \\
  - K^* \sqrtt_{\alpha(0)} [ (\Id + K K^*)^{-1}]& \sqrtt_{\alpha(1)}   [ (\Id + K^* K)^{-1} ]  
\end{pmatrix} \, .
\end{equation}
The map $S_{\alpha}$ is Lipschitz continuous on the set of matrices with singular values bounded by $\epsilon^{-1}$,  with Lipschitz constant only depending on $\epsilon$ and $n$.
\end{lemma}

\begin{proof}
By definition of $SU_{\alpha} (2n)$,  the matrix $H$ is triangular with positive entries along the diagonal, and therefore $H$ is invertible. Thus \eqref{eq:strip_cst_forward_2} is a well-defined map into $\mathcal{M}$.

We now show that $S_{\alpha}$ defines a map into $SU_{\alpha}(2n)$. Given $K \in  \mathcal{M}$, we label
\begin{equation}\label{eq:Talpha}
\begin{pmatrix}
    E' & F' \\ G' & H'
\end{pmatrix} := S_{\alpha} (K) \, ,
\end{equation}
and so we may then write
\begin{equation}\label{eq:B_C_strip_inverse}
F' = K H \, , \qquad  G' = -K^* E \, .
\end{equation}
We first check that $S_{\alpha} (K)$ is unitary. Using \eqref{eq:B_C_strip_inverse}, we compute
\begin{equation}\label{eq:S_unitary_check}
S_{\alpha} (K) ^* S_{\alpha}(K)  = \begin{pmatrix}
    (E')^* (\Id + K K^*) E' & 0 \\ 0 & (H')^* (\Id +  K^* K) H' \\
\end{pmatrix} = \Id \, ,
\end{equation}
where in the last step we used \eqref{e:sqrt}.
% \begin{equation}\label{eq:relate_K_A_D}
% (\Id + K K^*)^{-1} = E' (E')^* \, , \qquad  (\Id + K^* K)^{-1} = H'(H')^* \, ,
% \end{equation}
% which follows from
% \eqref{eq:inverse_strip} and the definitions of the map $\mathcal{Q}_{j}$ for $j=1,2$.
Thus $S_{\alpha} (K)$ is a.e. $U(2n)$-valued. Again by \eqref{e:sqrt}, we have that  $E' \in \grp_{\alpha(0)}$ and $H' \in \grp_{\alpha(1)}$, and Lemma \ref{lem:det_1_Y} shows that $S_{\alpha}(K)$ must then be a.e.\ $SU(2n)$-valued.

It is immediate that $S_{\alpha}$ is the right inverse of \eqref{eq:strip_cst_forward_2}, and so we now show that $S_{\alpha}$ is the left inverse of \eqref{eq:strip_cst_forward_2}. Given 
\begin{equation}
    \label{e:unitarity_assu}
\begin{pmatrix}
    E & F \\ G & H 
\end{pmatrix} \in SU_{\alpha} (2n) \, ,
\end{equation}
define
\begin{equation}\label{eq:defn_K}
K := F H^{-1} 
\end{equation}
and use the labeling of $S_{\alpha} (K)$ as in \eqref{eq:Talpha}. 
We have 
\[
    H' (H')^* = (\Id + K^* K)^{-1} = (H^{-*}(H^* H + F^* F) H^{-1})^{-1} = H H^*,
\]
where we used the unitarity assumption \eqref{e:unitarity_assu}. By uniqueness of the Cholesky factorization, $H = H'$. Since $F' = KH'$, this immediately gives $F = F'$, and as a simple corollary of Theorem \ref{factorization-theorem}, the $F$ block uniquely determines any $SU_\alpha$ matrix. 
% By Definition \eqref{eq:defn_K} of $K$,
% \[
%  K K^*= F H^{-1} H^{-*} F^*  = F (\Id - F^* F  )^{-1} F^*  \, .
% \]
% Multiplying both sides on the right by $\Id - F F^*$ yields
% \[
% (\Id - F F^*) K K^*=  F (\Id - F^* F  )^{-1} F^* - F F^* F (\Id - F^* F  )^{-1} F^*
% \]
% \[
% = F (\Id - F^* F) (\Id- F^* F)^{-1} F^* = F F^*  \, .
% \]
% Using the idenity $E E^*  = \Id - F F^*$, this becomes
% \[
% K K^* = E^{-*} E^{-1} (\Id - E E^*) = E^{-*} E^{-1} -\Id \,  .
% \]
% Combining this last computation with \eqref{eq:relate_K_A_D} yields 
% \[
% (E')^* E' = E^* E \, .
% \]
% By the uniqueness of the spectral factorization in Lemma \ref{lem:outer_func_construction}, $E = E'$.
% Similarly, one obtains that $H= H'$. We may then write
% \[
% F' = K H' = (F H^{-1}) H  = F \, ,
% \]
% and similarly 
% \[
% G' = -K^* E = - (F H^{-1})^* E = - (-G E^{-1}) E = G \, . 
% \]
Thus, $S_{\alpha}(F H^{-1})$ must equal \eqref{e:unitarity_assu}, i.e., $S_{\alpha}$ is the left inverse of \eqref{eq:strip_cst_forward_2}.

The continuity of the map \eqref{eq:strip_cst_forward_2} is immediate, whereas the continuity of $S_{\alpha}$ follows from Lemma \ref{lem:cty_spec_factors}. Thus both maps are homeomorphisms.

We now show $S_{\alpha}$ is Lipschitz continuous on the set of matrices $K$ with singular values at most $\epsilon^{-1}$. Because 
\[
\|K\|_{\infty} = \|K^*\|_{\infty} \leq \epsilon^{-1} \, ,
\]
it suffices to show Lipschitz continuity of the map
\begin{equation}\label{eq:K_to_OQ}
K \mapsto \sqrtt_{j} ( (\Id + K^* K)^{-1})
\end{equation}
for $j\in\{0, 1\}$. The singular values of $K$ belong to $[0, \epsilon^{-1}]$, hence  the singular values of $(\Id + K^* K)^{-1}$ lie in $[(1+\epsilon^2)^{-1},1]$. Lipschitz continuity of the map \eqref{eq:K_to_OQ} now follows from Lemma \ref{lem_cholesky}.
\end{proof}

The following is an immediate consequence of Lemma \ref{lem:layer_strip_recover} and Identity \eqref{eq:layerstripping_1}. 
\begin{lemma}
    \label{layerstripping}
Let $\alpha :\Z_2 \to \Z_2$. Suppose $F$ is a sequence of $n \times n$ contractive matrices with finite support within $[0, \infty)$, and let 
\[
\begin{pmatrix}
    A & B \\ C & D
\end{pmatrix} := \mathcal{F}_{\alpha} (F) \, .
\]
Then 
\begin{equation}\label{eq:stripping_def}
\begin{pmatrix}
    E_0 & F_0 \\ G_0 & H_0
\end{pmatrix}  = S_{\alpha} (B(0) D(0)^{-1}) \, .
\end{equation}
\end{lemma}

We will now define the Layer Stripping Algorithm, which, given the NLFT $M:\T \to U(2n)$  of some finitely supported sequence $F$  on $[0, \infty)$, returns $F$. The recovery of $F$ for general $M$ reduces to this case by shifting (Lemma \ref{lem:properties} \ref{item:translation}). We continue using the block labeling convention of  \eqref{eq:convention}.

\begin{algorithm2e}[H]
\vspace*{1mm}
\KwIn{An mvf $M: \T \to U(2n)$ satisfying the constraints 
\begin{equation} \label{eq:Mj_space}
M  \in \begin{pmatrix}
    H^2 (\D^*)  & H^2 (\D) \\ H^2 (\D^*) & H^2 (\D)
\end{pmatrix} \, , \qquad A^*(0) \in \mathcal{G}_{\alpha(0)} \, , \qquad D(0) \in \mathcal{G}_{\alpha(1)} \, .
\end{equation}}
\KwOut{A coefficient sequence $(F_j)_{j \in \Z}$ supported on $[0, \infty)$.}
\vspace*{2mm}

\begin{enumerate}\itemsep1pt
\item Set $M_0 := M$ and $F_j = 0$ for all $j <0$.\\[1mm]
\item \label{item:loop_alg} For $j \geq 0$ repeat the following:\\[1mm]
Suppose we are given $M_j:\T \to U(2n)$ satisfying  \eqref{eq:Mj_space}. 
\\[1mm] 
Set $F_j$ to be the upper right block of $S_{\alpha} (B_j(0) D_j(0)^{-1})$.\\[1mm]
Set
\begin{equation}\label{eq:Mj_to_next}
M_{j+1} (z) := \ad(z)^{-1}[ S_{\alpha} (B_j (0) D_j (0)^{-1})^{-1} M_j (z)] \, .
\end{equation}
\item Return the sequence $(F_j)_{j \in \mathbb{Z}}$.
\end{enumerate}
\caption{Layer Stripping Algorithm \label{algo:layer_stripping}
}
\end{algorithm2e}

The following lemma shows that the iteration of Step \ref{item:loop_alg} in the Layer Stripping Algorithm \ref{algo:layer_stripping} is well-defined, and eventually stabilizes if the input $M$ is a polynomial. 
\begin{lemma}
    \label{lem:layer_stripping_finite}
   Let $\alpha :\Z_2 \to \Z_2$. If $M_j$ is a map $\T \to U(2n)$ satisfying  \eqref{eq:Mj_space}, then so is $\nlf_{j+1}$, as defined in \eqref{eq:Mj_to_next}.
 
 If, additionally, $M_{j}$ is a Laurent polynomial of degree at most $d$ for some $d\geq 1$, then $M_{j+1}$ is a Laurent polynomial of degree at most $d-1$. If  $\nlf_j$ is constant unitary, then $\nlf_{j+1} = \Id$. 
\end{lemma}
\begin{proof}
We begin with the first statement of the lemma. We label the blocks of the mvfs $M_j$ and $M_{j+1}$ as in \eqref{eq:convention}.  Taking $E_j, F_j, G_j, H_j$ to be the blocks of $S_{\alpha} (B_j (0) D_j (0)^{-1})$ as in \eqref{eq:stripping_def}, we define  
\begin{equation}\label{eq_DefABCD''''}
\begin{pmatrix}    
A ' (z) &  B'(z) \\ C '(z)  & D ' (z)
\end{pmatrix} := \begin{pmatrix}
    E_j ^* & G_j ^* \\ F_j ^* & H_j ^*
\end{pmatrix} \begin{pmatrix}    
A_j (z) & B_j(z) \\ C_j(z) & D_j(z)
\end{pmatrix} 
\end{equation}
\[= \begin{pmatrix}    
E_j ^* A_j (z) + G_j ^* C_j(z) & E_j ^* B_j(z) + G_j ^* D_j(z) \\ F_j ^* A_j(z) + H_j ^* C_j(z) & F_j ^* B_j(z) + H_j ^* D_j(z)
\end{pmatrix} \, . 
\]
Thus $A'$, $C'$ have frequency support in $(-\infty,0]$, while $B'$, $D'$ have frequency support in $[0, \infty)$.
We claim that 
\begin{equation}\label{eq:vanishing_BC}
B'(0) = 0 \, , \qquad C'(\infty) = 0 \, .
\end{equation}
To see the first part of the claim,
\[
B'(0)=E_j ^* B_j(0) + G_j ^* D_j(0) = E_j ^* ( B_j(0) D_j(0)^{-1} +  E_j  ^{-*} G_j ^* ) D_j(0) \, , 
\]
which by unitarity of the matrix \eqref{eq:stripping_def}, and then the definition of $S_{\alpha}$, equals
\[
 E_j ^* ( B_j(0) D_j(0)^{-1} - F_j H_j  ^{-1} ) D_j(0) = 0 \, .
\] Similarly, 
\[
C'(\infty)=H_j ^* [  H_j ^{-*} F_j ^* + C_j(\infty) A_j(\infty)^{-1} ] A_j(\infty) \,
\]
which will then vanish if we show 
\begin{equation}\label{eq:AC_BD_cst_Fourier}
C_j(\infty) A_j(\infty)^{-1} = -(B_j(0) D_j(0)^{-1})^* \, .
\end{equation}
But because the matrix \eqref{eq:convention} is unitary,  

\[
 B_j^*(z)A_j(z)+D_j^*(z)C_j(z)=0 
\]
for all $z \in \T$. Observe the left side is analytic in $\D^*$. Evaluating at $z = \infty$ yields
\eqref{eq:AC_BD_cst_Fourier}.
Thus, $C' (\infty)=0$, completing the proof of Claim \eqref{eq:vanishing_BC}. It follows that $B'$ and $C'$ have frequency support in $[1,\infty)$ and $(-\infty, 1]$, respectively.

Next we check that $A'(\infty) \in \grp_{\alpha(0)}$ and $D'(0) \in \grp_{\alpha(1)}$.
Multiplying both sides of \eqref{eq_DefABCD''''} by the inverse of the constant unitary matrix and then evaluating the bottom right entries at $z=0$ reads $G_jB'(0)+H_jD'(0)=D_j(0)$.
We have already checked that $B'(0)=0$, and so $D'(0)=H_j^{-1}D_j(0)$ belongs to  $\grp_{\alpha(1)}$ because both $H_j$ and $D_j(0)$ do.
Reasoning similarly for the top left entry of \eqref{eq_DefABCD''''} and using that $C'(\infty)=0$, we also conclude that $A'(\infty)=E_j^{-1}A_j(\infty) \in \grp_{\alpha(0)}$.

Noting that 
\begin{equation}\label{eq:Mj+1_to_A'}
\nlf_{j+1}= \ad(z^{-1}) \begin{pmatrix}
    A' &  B' \\  C' &  D'
\end{pmatrix} \, ,
\end{equation}
we now see that  \eqref{eq:Mj_space} holds for $\nlf_{j+1}$. %Furthermore, by \eqref{eq:Mj+1_to_A'} and noting the constant matrix in \eqref{eq_DefABCD''''} has determinant $1$, we then obtain  
   %\begin{equation}\label{eq:det_conserved_layer_stripping}
   %    \det \nlf_{j+1} = \det \nlf_j \, .
   %\end{equation}
    Because $\nlf_j: \T \to U(2n)$, then $\nlf_{j+1}: \T \to U(2n)$.

We now check the polynomial statement. Assume that $\nlf_{j}$ is also a Laurent polynomial of degree at most $d$, for some $d \geq 1$. From \eqref{eq_DefABCD''''}, and the fact that $B'$ and $(C')^*$ have vanishing means, it follows that both have frequency support on $[1, d]$. Thus \eqref{eq:Mj+1_to_A'} implies $B$ and $C^*$ have frequency support on $[0, d-1]$. We now check that the degree of the Laurent polynomials $A',  D'$ is at most $d-1$. 
 Since the right side of \eqref{eq_DefABCD''''} is the product of two $U(2n)$ matrices, it follows that 
\[\begin{pmatrix}    
A' (z) & B'(z) \\ C'(z) & D' (z)
\end{pmatrix} \in U(2n)\text{ if }z\in\mathbb T,
\]
and so 
\[(B')^*(z)B'(z)+(D')^*(z)D'(z)=\Id \text{ if }z\in\mathbb T.\]
Writing $B'(z)=\sum_{\ell=0}^{d-1} B'_\ell z^{\ell}$ and 
$D'(z)=\sum_{\ell=0}^{d} D'_{\ell} z^{\ell}$, we see that for $z\in\mathbb T$
\[\sum_{k=0}^{d-1}\sum_{\ell=0}^{d-1} (B_{\ell} ')^* B_k' z^{k-\ell}+\sum_{k=0}^{d}\sum_{\ell=0}^{d} (D_{\ell}')^* D_k' z^{k-\ell}=\Id.\]
The left side coefficient of $z^d$ arises from choosing $(\ell,k)=(0,d)$ on the second sum, and by equating it to the right side we get $(D_0')^* D_d '=0$. Since $D_0'=D'(0)$ is invertible, it follows that $D_d'=0$ and so $D'$ has degree at most $d-1$.
Analogous reasoning via the identity
\[(A')^*(z)A'(z)+(C')^*(z)C'(z)=\Id \text{ if }z\in\mathbb T\]
reveals that the degree of $A'$ is also at most $d-1$. This proves the polynomial claim.
%\dos{Bit of a conflicting notation between conjugate transpose $X^*$ and the matrix-valued reflected function $A^*(z):=\overline{A(\overline{z^{-1}})}.$}
%\ma{On $\T$, for a scalar function $a(z)$, note that $\overline{a(\overline{\frac{1}{z}})} = \overline{a(z)}$. I see a transpose implicit in the $*$ operation too (which is trivial for scalar functions). So in higher-dimensions, I always think of $*$ as both conjugate transpose and the higher-dimensional analogue of $*$ that we know from scalar complex analysis.} 

We are left with the statement for when $\nlf_j$ is constant. In this case, Lemma \ref{lem:det_1_Y} reveals that $\nlf_j$ is a constant $SU_{\alpha}(2n)$-valued mvf. Thus Lemma \ref{lem:layer_strip_recover} yields  
\[
\begin{pmatrix}
    E_j & F_j \\ G_j & H_j
\end{pmatrix} = S_{\alpha} \left (B_j (0) D_j  ^{-1}(0) \right ) = \nlf_j
\]
and so \eqref{eq:Mj_to_next} yields $\nlf_{j+1} = \Id$.
\end{proof}

\subsection{The image of finitely supported sequences}
 We continue the labeling convention \eqref{eq:convention} in what follows.
\begin{lemma} \label{lem:bijection_finite_supp}
Let $\alpha:\mathbb Z_2\to \mathbb Z_2$.
    Then $\mathcal{F}_{\alpha}$ is a bijection from the set of  finitely supported sequences  of contractive $n \times n$ matrices onto the set of Laurent polynomials 
    \begin{equation}\label{eq:unitary_valued}
    \nlf: \T \to U(2n)
    \end{equation}
    satisfying
   \[
   A(\infty)\in\mathcal G_{\alpha(0)} \, , \qquad  D(0)\in\mathcal G_{\alpha(1)} \, .
   \] 
  One may replace $U(2n)$ by $SU(2n)$ in \eqref{eq:unitary_valued}.
  Furthermore, if $B$ and $C^*$ are analytic, then $\mathcal{F}_{\alpha}$ has inverse given by the Layer Stripping Algorithm \ref{algo:layer_stripping}.
\end{lemma}

% \lb{Did you change $SU(2n)$ to $U(2n)$ here? If you want to make the point here that every such matrix valued function in $U(2n)$ automatically has determinant $1$ then we should point that out explicitly}
% \ma{changing it back to $SU(2n)$ now that we have your lemma!}

\begin{proof}
   By Lemma \ref{lem_MultilinearExpansion},  $\mathcal{F}_{\alpha}$ maps into either of the spaces specified by the lemma. We now prove bijectivity.

    Given a $2n\times2n$ mvf $\nlf$ as in the statement of the lemma, recalling property \ref{item:translation} of Lemma \ref{lem:properties}, a shift reduces matters to the case when $B$ and $C^*$ are analytic. There exists $d$ sufficiently large so that the left and right block columns of $\nlf$, as in \eqref{eq:convention},  have frequency support on $[-d, 0]$ and $[0, d]$, respectively. We show that each such element has a unique preimage, namely the output of the Layer Stripping Algorithm \ref{algo:layer_stripping}. %We have that $A(z)$ is a matrix-valued Laurent polynomial in $z^{-1}$ since $A(\infty)$ has finite diagonal entries. Let $d$ denote the degree of $A$, so that $A(z)=\sum_{j=-d}^0 A_j z^j$. If $B(z)=\sum_{j=d_-}^{d_+} B_j z^j$, with $B_{d_-}\neq 0\neq B_{d_+}$, then from $AA^\ast+BB^\ast=\Id$ it follows that $d=d_+-d_-.$ One similarly checks that $D(z)=\sum_{j=0}^d D_j z^j$ and $C(z)=\sum_{j=-d_+}^{-d_-} C_j z^j$, where $-d_--(-d_+)=d.$ 

{\bf Existence.}
Given a matrix $\nlf_0:=M$ satisfying the above conditions, define $F$ to be the sequence output by the Layer Stripping Algorithm \ref{algo:layer_stripping}, and for $j \geq 1$, define $\nlf_j$ as in \eqref{eq:Mj_to_next}.  We now show by (descending) induction on $0 \leq j \leq d+1$ that 
\begin{equation}\label{eq:induct_nlft}
\nlf_{j} = \mathcal{F}_{\alpha} ( (F_{j+k} \mathbf{1}_{\{k \geq 0\} } )_{k \in \Z})  \, . 
\end{equation}
When $j=d+1$, both sides of  \eqref{eq:induct_nlft} equal $\Id$.
Indeed, Lemma \ref{lem:bijection_finite_supp} reveals $\nlf_{d+1} = \Id$ and $F_k =0$ for all $k \geq d+1$. Now assuming \eqref{eq:induct_nlft} holds for $0 < j\leq d+1$, the translation symmetry \ref{item:translation} of Lemma \ref{lem:properties}, the recurrence \eqref{eq:Mj_to_next} and the identity \eqref{eq:stripping_def} all reveal that \eqref{eq:induct_nlft} also holds for $j-1$.
 Taking $j=0$ in \eqref{eq:induct_nlft} then completes the existence proof.

{\bf Uniqueness.} Assume $\nlf_0 :=\nlf$ is an NLFT of some finitely supported sequence $F$. After shifting using the translation property \ref{item:translation} of Lemma \ref{lem:properties}, we can assume without loss of generality that $F$ is supported on $[0, d]$. By induction and using \eqref{eq:stripping_def} together with the translation property \ref{item:translation} of Lemma \ref{lem:properties}, it follows that for each $j \geq 0$, the mvf $\nlf_j$ defined in \eqref{eq:Mj_space} is the NLFT of $(F_{k+j+1} \mathbf{1}_{ \{k \geq 0\} })_{k \in \Z}$. Lemma \ref{layerstripping} then  reveals that $F_j$ is uniquely determined by $\nlf_j$. Since all the $\nlf_{j}$ are determined by $\nlf_0$, we obtain that $F$ is uniquely determined by $\nlf_0 = \nlf$.   
%By the induction hypothesis and Lemma \ref{lem:layer_stripping_finite}, it is enough to check that $F_0$ can be uniquely recovered from  a given matrix of the form \eqref{eq_ABCD} satisfying \ref{item:finLaurent}--\ref{item:finG}.
%In turn, this follows from \eqref{eq:stripping_def}, where $F_0$ is the upper right entry of $S_{\alpha}(B(0)D(0)^{-1})$.
\end{proof}

\section{The \texorpdfstring{$\ell^2$}{l2} theory}

In this section, we fix a function $\alpha:\Z_2 \to \Z_2$. In the sequel, we will consider $2n\times 2n$ mvfs $\nlf$, which will also often, but not always, be the NLFT of some sequence. We continue to use the notational convention \eqref{eq:convention} to denote the blocks of such $\nlf$. 

\subsection{Norms, metrics  and spaces}

We recall that the Hilbert--Schmidt norm is given by
$$
    \|A\|^2_2 = \trc(A^* A) = \sum_{i = 1}^n \lambda_i(A)^2\,,
$$
where $(\lambda_i (A) )_{i=1}^n$ denotes the singular values of the matrix $A$, which are by definition nonnegative. %Recall also that the operator norm $\|A\|_{\infty}$ of a matrix $A$ is the maximum of all its singular values. \lb{Check if this is used anywhere}
We define the $\ell^2$ norm of a sequence $F=(F_j)_{j \in \mathbb{Z}}$ of matrices as
$$
    \| (F_j)_{j \in \mathbb{Z}} \|_{\ell^2}^2 := \sum_{j \in \mathbb{Z}} \|F_j\|^2_2 = \sum_{j \in \mathbb{Z}} \sum_{i = 1}^n \lambda_{i}(F_j)^2\,.
$$
To motivate the definition of the space of mvfs on the torus into which the NLFT maps, we start with the following Plancherel type lemma.

\begin{lemma}
    \label{lem plancherel}
    Let $\alpha: \Z_2 \to \Z_2$ and let $F $ be a finitely supported sequence of contractive matrices. Denote its \textup{NLFT} by 
    \[
        \begin{pmatrix}
      A & B \\ C & D 
  \end{pmatrix}:= \mathcal{F}_\alpha(F) \,.
   \]
   Let $z_1, \dotsc, z_k$ be the zeros of $\det A^*$ in $\D$, counted with multiplicity. Then 
   \begin{equation}
        \label{plancherel}
        \sum_{j \in \mathbb{Z}} \log \det E_j  = \sum_{j \in \mathbb{Z}} \log \det H_j  = \frac{1}{2}\sum_{j \in \mathbb{Z}} \log  \det(\Id - F_jF_j^*) 
   \end{equation}
   \begin{equation}
       \label{plancherel_ineq}
       = \log  \det(A(\infty)) =\int_{\mathbb{T}} \log \lvert\det(A(z))\rvert + \sum_{i = 1}^k \log |z_i|\,.
   \end{equation}
   
   % Let $z_1', \dotsc, z_\ell'$ be the zeros of $\det D$ in the unit disc, counted with multiplicity. Then we also have
   % \begin{equation}
   %      \label{plancherelD}
   %      \sum_{j \in \mathbb{Z}} \log\lvert \det H_j \rvert = \frac{1}{2}\sum_{j \in \mathbb{Z}} \log \lvert \det(I - F_jF_j^*)\rvert = \log \lvert \det(D(0)) \rvert
   % \end{equation}
   % \begin{equation}
   %     \label{plancherel_ineqD}
   %     =\int_{\mathbb{T}} \log \lvert\det(D(z))\rvert+ \sum_{i = 1}^\ell \log |z_i'|\,.
   % \end{equation} 
\end{lemma}
 \begin{remark}
    \label{rem:AD}
     By Lemma \ref{lem:det_holo_formula} and analytic continuation, we have 
    \[
        \det D = \det A^*
    \]
    on $\mathbb{D}$, so that analogs of \eqref{plancherel} and \eqref{plancherel_ineq} also hold with $D$ in place of $A$. 
 \end{remark}
 Further note that by contractivity and the maximum principle, all terms in the equations of Lemma \ref{lem plancherel} are nonpositive.

\begin{proof}
    The relations $E_jE_j^* + F_j F_j^* = \Id$ and $H_j^*H_j + F_j^*F_j = \Id$, which follow from unitarity of $Y_\alpha(F_j)$,
    imply \eqref{plancherel}. The equality with the first term of \eqref{plancherel_ineq} follows from \eqref{eq:constant_Fourier_A_D}.

    We turn to the final identity. Since the sequence $F$ is finitely supported, Lemma \ref{lem_MultilinearExpansion} shows that $\det A^*$ is a polynomial. Hence, it factors as an outer function $o$ times a finite Blaschke product, that is,
    \[
        \det A^* (z) =  o(z) \prod_{j=1}^k \frac{z - z_j}{1 - z \bar z_j} \, . 
    \]
     By the mean value property for $\log \lvert o\rvert$, which holds because $o$ is outer, and since Blaschke products are unimodular on $\T$, we have
    \[
        \log \lvert \det A(\infty) \rvert - \sum_{j=1}^k \log \lvert z_j \rvert = \log \lvert o(0) \rvert = \int_\mathbb{T} \log \lvert o(z)\rvert  = \int_\mathbb{T} \log \lvert \det A(z) \rvert \, . 
    \]
    Noting that $\lvert \det A(\infty)\rvert = \det A(\infty)$ then completes the proof. 
\end{proof}

\begin{remark}
    \label{rem:comp plan}
    The Plancherel identities \eqref{plancherel} and \eqref{plancherel_ineq} can be strengthened to a componentwise statement for the diagonal entries of $A$ and $D$. Suppose that we are in the setting of Lemma \ref{lem plancherel}, and fix some $1 \le m \le n$. Let $k_m$ be the number of zeros of $A_{mm}^*$ in the unit disc, and denote these zeros (with multiplicity) by $z_1^{(m)}, \dotsc, z_{k_m}^{(m)}$. Then
    $$
        \sum_{j \in \mathbb{Z}} \log (E_{j})_{mm} = \log A_{mm}(\infty) = \int_\mathbb{T} \log \lvert A_{mm}(z)\rvert  + \sum_{i=1}^{k_m} \log \lvert z_j^{(m)}\rvert\,,
    $$
    and a similar statement holds for $D$. 
\end{remark}

% In particular, \eqref{plancherel} implies 
% \[
%      \lvert\log  \det(A(\infty)) \rvert = \frac{1}{2} \sum_{j\in\mathbb{Z}} \log \lvert \det(I - F_j F_j^*)\rvert
%     = \frac{1}{2} \sum_{j\in\mathbb{Z}} \sum_{i = 1}^n \log \lvert 1 - \lambda_i(F_j)^2 \rvert\,,
% \]
% which is finite if the sequence $F_j$ is in $\ell^2$, a similar computation also shows finiteness of $\log\lvert\det(D(0))\rvert$.
% This motivates the following definition of the target space for the nonlinear Fourier transform.

We now define the target space for the NLFT on $\ell^2$.
Let $\mathbf{L}_\alpha$ be the space of all functions $M: \mathbb{T} \to SU(2n)$,
\begin{equation}\label{eq:matrix_spaces_ntn}
    M = \begin{pmatrix}
        A&B\\
        C&D
    \end{pmatrix} 
    \in
    \begin{pmatrix}
        H^2(\mathbb{D}^*) & L^2(\mathbb{T})\\
        L^2(\mathbb{T}) & H^2(\mathbb{D})
    \end{pmatrix} \, ,
\end{equation}
such that $A^*(0) \in \mathcal{G}_{\alpha(0)}$ and $D(0) \in \mathcal{G}_{\alpha(1)}$. 
We equip $\mathbf{L}_\alpha$ with the metric $\metric$ introduced in Definition \ref{defH}.
We further denote by $\mathbf{L}^{+}_{\alpha}$ the subspace of all functions with $B, C^* \in H^2(\mathbb{D})$.

We will also need to work in the space $\mathbf{U}_\alpha^+$, which is nearly identically defined  as the space $\mathbf{H}_{\alpha} ^+$ in Definition \ref{defH}, except that we replace $SU(2n)$ by $U(2n)$ in \eqref{labeling_M_intro}. Note that all NLFTs in the space $\mathbf{U}_{\alpha} ^+$ are automatically in $\mathbf{H}_{\alpha} ^+$ because NLFTs are always $SU(2n)$-valued. It will turn out that both spaces are equal; see Corollary \ref{cor:H_equals_U}.  

\begin{lemma}
    \label{lem:complete}
    The metric space $(\mathbf{L}_\alpha, \metric)$ is complete, and its subspace $\mathbf{L}^{+}_{\alpha}$ is closed in $\mathbf{L}_\alpha$.
\end{lemma}

\begin{proof}
    If $M_\ell$ is a Cauchy sequence in $(\mathbf{L}_\alpha, \metric)$, then the blocks $A_\ell ^*, D_\ell$ are Cauchy in $H^2 (\mathbb{D})$ and $B_\ell, C_\ell$ are Cauchy in $L^2(\mathbb{T})$. Thus there exists an $L^2$ limit 
    \[
        M := \begin{pmatrix}A&B\\C&D\end{pmatrix}
        \in
    \begin{pmatrix}
        H^2(\mathbb{D}^*) & L^2(\mathbb{T})\\
        L^2(\mathbb{T}) & H^2(\mathbb{D})
    \end{pmatrix}.
    \]
    Clearly, $M \in SU(2n)$ a.e.\ on $\mathbb{T}$. From $H^2(\mathbb{D}^*)$-convergence of $A_\ell$ to $A$ it follows that $A_\ell(\infty)$ converges to $A(\infty)$. So $A(\infty)$ belongs to the closure of $ \grp_{\alpha(0)}$, which consists of upper or lower triangular matrices, depending on value of $\alpha(0)$, with nonnegative diagonal. 
    Since
    \[
        \lvert\log \det A_\ell(\infty) - \log \det A_m(\infty)\rvert \le \metric(M_\ell, M_m),
    \]
    the sequence $\log \det A_\ell(\infty)$ is Cauchy and thus converges to some real number. Therefore  
    \[
        \det A(\infty) = \lim_{\ell \to \infty} \det A_\ell(\infty) = \exp( \lim_{\ell \to \infty} \log \det A_\ell(\infty)) > 0\,.
    \]
    In particular, $A(\infty)$ has nonzero diagonal, showing that in fact $A(\infty) \in \mathcal{G}_{\alpha(0)}$. By Lemma \ref{lem:det_holo_formula} and since $M_\ell(z)\in SU(2n)$ for $z \in \mathbb{T}$, we have that 
    \[
        \det D_\ell(0) = \det A_\ell(\infty)
    \]
    for all $\ell$, from which it also follows that 
    \[
        \det D(0) = \det A(\infty) > 0.
    \]
    Like $A(\infty)$, the matrix $D(0)$ is contained in the closure of $\mathcal{G}_{\alpha(1)}$ by $H^2(\mathbb{D})$ convergence, and since it has positive determinant it is in fact in $\mathcal{G}_{\alpha(1)}$.

    The space $\mathbf{L}_{\alpha}^{+}$ is closed in $\mathbf{L}_\alpha$ because $H^2(\mathbb{D})$ and $H^2(\mathbb{D}^*)$ are closed in $L^2(\mathbb{T})$.
\end{proof}

\begin{remark}
While the space $\mathbf{H}_\alpha ^+$ introduced in Definition \ref{defH} is a subspace of $\mathbf{L}_\alpha ^+$, it is not closed in $\mathbf{L}_\alpha ^+$. For the $SU(2)$ case, consider  
$$
    A_\ell^*(z) = D_\ell(z) = \tfrac1{10} (z + \tfrac12)\,, \qquad \ell \ge 0\,,
$$
and the Poisson extension of $\sqrt{1 - |D|^2}$ times a single Blaschke factor
$$
    C_\ell^*(z) = B_\ell(z) = \exp\left(\frac{1}{2} \int_\mathbb{T} \frac{z + \zeta}{z - \zeta} \log \lvert 1 - \tfrac1{100}|\zeta + \tfrac12|^2 \rvert \right) \frac{z - x_\ell}{1 - z\bar x_\ell}\,.
$$
Suppose that $x_\ell \to -\frac{1}{2}$ as $\ell\to\infty$, but that $x_\ell \ne -\frac{1}{2}$ for all $\ell \ge 0$. Then $M_\ell$ belongs to $\mathbf{H}_\alpha ^+$ for each $n$. However,  $M_\ell$ converges in $(\mathbf{L}_\alpha, \metric)$ to an mvf $\nlf$ where both $B, D$ vanish at $-\frac12$, and consequently have a common inner factor.
\end{remark}

\begin{lemma}
    \label{lem_continuous} Let $c \le d$. 
    The NLFT $\mathcal{F}_\alpha$ defines a continuous map from the space of sequences supported in $[c,d]$ with the $\ell^2(\mathbb{Z})$ topology into $\mathbf{L}_{\alpha}$. 
\end{lemma}

\begin{proof}
    By definition, the NLFT and its value at $0$ are polynomial expressions in the $d-c + 1$ coefficients. 
\end{proof}

\subsection{Extension to half-line \texorpdfstring{$\ell^2$}{l2} sequences}

Our goal here is to extend the NLFT to square summable sequences of matrices supported on the positive half-line. 
We start with two technical lemmas. 

\begin{lemma}
    \label{lem:dist_Id}
    For finitely supported sequences $F$, it holds that
    \begin{equation}
        \label{eq:dist_Id}
        \metric(\mathcal{F}_\alpha(F), \Id) \le \sum_{j \in \mathbb{Z}} \lvert\log  \det(E_j)\rvert + 2 \Big(\sum_{j \in \mathbb{Z}} \lvert\log  \det(E_j)\rvert\Big)^{\frac12}\, .
    \end{equation}
    %where $\Id \in \mathbf{L}_\alpha$ denotes the function that equals  the identity matrix everywhere. 
\end{lemma}

\begin{proof}
    We denote the blocks of the nonlinear Fourier transform of $F$ by
    \[
        \begin{pmatrix}
            A&B\\C&D
        \end{pmatrix}:= \mathcal{F}_\alpha(F) \,.
    \]
    By \eqref{eq:constant_Fourier_A_D}, we have
    \[
        \lvert \log  \det A(\infty) - \log  \det \Id  \rvert = \sum_{j\in \mathbb{Z}} | \log  \det(E_j) | \, .
    \]
    This controls the second  term in $\metric(\mathcal{F}_\alpha(F), \Id)$ by the first summand in \eqref{eq:dist_Id}.
For the first summand in $\metric(\mathcal{F}_\alpha(F), \Id)$, we write using unitarity of $\mathcal{F}_{\alpha} (F)$ 
\[
\lVert\mathcal{F}_\alpha(F) - \Id\rVert_2^2 = \trc (2 \Id -  \mathcal{F}_\alpha(F) - \mathcal{F}_\alpha(F) ^* ) \, .
\]
Expressing this in terms of $A$ and $D$, and integrating on $\T$, yields
\[
\frac{1}{2} \int\limits_{\T} \lVert\mathcal{F}_\alpha(F) - \Id\rVert_2^2 =  \int\limits_{\T} \Re \trc (\Id -  A) + \int\limits_{\T} \Re \trc (\Id - D) \, .
\]
Using the mean value property of the entries of $A$ and $D$, this equals
\[
\Re \trc (\Id -  A (\infty) ) + \Re \trc (\Id - D (0)) \, .
\]
Since $A(\infty)$ is upper triangular with positive diagonal, and using that for all $x > 0$ we have $1 - x \le -\log x$, the first term is bounded by
\[
 \sum\limits_{i=1}^n 1 - \lambda_{i} (A(\infty)) \leq  \sum\limits_{i=1}^n  \lvert \log \lambda_{i} (A(\infty))\rvert = \lvert\log \det A(\infty)\rvert = \sum\limits_{j\in\mathbb Z} \lvert\log \det E_j \rvert \, .
\]
A similar argument applies to $D$.
\end{proof}

% \begin{remark}
%     In the proof of Lemma \ref{lem:dist_Id}, we have made use of our convention that all $E_j$ are upper triangular with positive diagonal. We were free to make up a convention on how to choose $E_j$, because only $E_jE_j^*$ is fixed. Here it was, however, important to make this choice of `square root' in some manner that is continuous at the identity and picks the identity as the `square root' of the identity matrix.
% \end{remark}
In what follows, we continue using the notational convention \eqref{eq:convention}.

%\lb{We could change the next lemma back to the (commented out) earlier version, since we now only use it for that case. (We removed another application in the meantime). But looking at it now maybe the current statement is just more natural, so I would be ok with just keeping it as it is.} \ma{Diogo and I vote that we keep it as is.}
\begin{lemma}
    \label{lem:submult}
% If $(F_j)_{j\in\mathbb{Z}}$ and $(F_j')_{j \in \mathbb{Z}}$ are finitely supported sequences such that the support of $F$ is entirely to the left of the support of $F'$, then their nonlinear Fourier transforms
% \[
%     M = \mathcal{F}_\alpha(F) \qquad \text{and} \qquad M' = \mathcal{F}_\alpha(F')
% \]
% satisfy
% \[
%     d(MM', M) = d(M', \Id).
% \]
Suppose that $M, M', MM' \in \mathbf{L}_\alpha$ and that
\begin{equation}\label{eq:mult_vanishing_prop}
        (BC') (\infty) = 0 \qquad \text{and} \qquad  (C B')(0) = 0\,.
    \end{equation}
Then
\begin{equation}\label{eq:dist_ids}
     \metric(MM', M) = \metric(M', \Id) \qquad \text{and} \qquad \metric(MM', M') = \metric(M, \Id)\,.
\end{equation}
In particular, \eqref{eq:mult_vanishing_prop} holds if $M = \mathcal{F}_{\alpha} (F)$ and $M' = \mathcal{F}_{\alpha} (F')$ for $F, F'$ finitely supported sequences such that the support of $F$ is entirely to the left of the support of $F'$.
\end{lemma}

\begin{proof}
    We only prove the first identity in \eqref{eq:dist_ids}, as the second follows similarly.
    Because multiplication by unitaries preserves singular values and hence the Hilbert--Schmidt norm,  noting that $M$ is unitary then yields
    \begin{equation}
        \label{eq product l2}
        \int\limits_{\mathbb{T}} \|MM' - M\|_2^2 =\int\limits_{\mathbb{T}} \lVert M' - \Id\rVert_2^2 \,.
    \end{equation}
    Because of the assumption \eqref{eq:mult_vanishing_prop}, the upper left block $A''$ of $MM'$ satisfies
    \[
        A''(\infty) = A(\infty) A'(\infty) + (B C')(\infty) = A(\infty) A'(\infty)\,.
    \]
    % Let us denote the upper left blocks of $M, M'$ by $A, A'$ and the one of $MM'$ by $A''$. 
    % Because the support of $F$ is to the left of the support of $F'$, we have 
    % \[
    %     A''(\infty) = \prod_{j \in \mathbb{Z}}^\aarrow E_j \cdot \prod_{j \in \mathbb{Z}}^\aarrow E_j' = A(\infty) A'(\infty)\,.
    % \]
    Hence we obtain, for the first logarithmic term in $d(MM', M)$,
    \begin{equation}
        \label{eq product log1}
        \lvert \log \det(A(\infty) A'(\infty))  - \log  \det(A(\infty)) \rvert = \lvert \log  \det(A'(\infty)) \rvert\,.
    \end{equation}
    Adding the square root of \eqref{eq product l2} to \eqref{eq product log1} yields the desired identity.

    For the last statement, if $F$ is supported to the left of $F'$, then $MM' = \mathcal{F}_{\alpha} (F+F')$ by Lemma \ref{lem:properties} \ref{item:ordered_mult}, and then by \eqref{eq:constant_Fourier_A_D}, we have
    \[
    A''(\infty) = A(\infty) A'(\infty) \, , \qquad D''(0) = D(0) D'(0) \, , 
    \]
    which implies \eqref{eq:mult_vanishing_prop}.
\end{proof}

As in the linear theory, the nonlinear Plancherel identity allows us to extend the NLFT to square summable sequences. 

\begin{lemma}
    \label{lem:ext_l2}
    The map $\mathcal{F}_{\alpha}$ extends to a continuous map from $\ell^2(\mathbb{Z}_{\geq 0} ;  \con)$ into $\mathbf{L}_{\alpha} ^{+}$.
\end{lemma}

\begin{proof}
    Let $F \in \ell^2(\mathbb{Z}_{\geq 0}; \con)$.
    Combining the previous two lemmas and \eqref{plancherel} reveals that the sequence  $\mathcal{F}_\alpha(F\mathbf{1}_{[0,N]})$ is Cauchy in $(\mathbf{L}_{\alpha}^{+}, \metric)$. By completeness of $\mathbf{L}_{\alpha}^+$ as shown in  Lemma \ref{lem:complete}, the sequence has a limit, which we define to be the nonlinear Fourier transform of $F$. When $F$ is a finitely supported sequence, then $F \mathbf{1}_{[0,N]}$ is constant in $N$ for  sufficiently large $N$, and so our extension of the NLFT coincides with the previous definition \eqref{def:NLFT_alpha} for finitely supported sequences. 

    In order to prove continuity, the following consequence of the previous two lemmas is helpful. We have, by definition,
    \[
        \metric(\mathcal{F}_\alpha(F),\mathcal{F}_\alpha(F \mathbf{1}_{[0,N]}))= \lim_{M \to \infty} \metric(\mathcal{F}_\alpha(F\mathbf{1}_{[0,M]}), \mathcal{F}_\alpha(F \mathbf{1}_{[0,N]}))\,.
    \]
    Using property \ref{item:ordered_mult} from Lemma \ref{lem:properties} and Lemma \ref{lem:submult}, this is at most
    \[
        \lim_{M \to \infty} \metric(\mathcal{F}_\alpha(F\mathbf{1}_{(N,M]}),\Id)\,,
    \]
    and using  Lemma \ref{lem:dist_Id} we then bound this by 
    \begin{equation}
        \label{eq:tail}
        \frac 1 2 \sum_{j = N}^\infty \lvert \log  \det(\Id - F_j F_j^*) \rvert  +  \Big( \sum_{j = N}^\infty \lvert \log  \det(\Id - F_j F_j^*) \rvert  \Big)^{\frac12}\, .
    \end{equation}
     The function $z \mapsto -\log(1-z)$ is convex, so for $0 \le z \le 1 - e^{-1}$ we have 
    \[
        - \log(1 - z) \le (1 - e^{-1})^{-1} z \le 2z\,.
    \]
    Applying this to $z = \lambda_i(F_j)^2$ yields that, if $\|F_j\|_\infty \le (1 - e^{-1})^{\frac12}$, then 
    \begin{equation}
        \label{eq:det_to_HS}
        \lvert \log  \det(\Id - F_j F_j^*) \rvert = \sum_{i = 1}^n -\log(1 - \lambda_i(F_j)^2) \le 2 \|F_j\|_2^2\,.
    \end{equation}
    
    We now prove continuity. Fix $F$ and $0 < \epsilon < (1 - e^{-1})^{\frac12}$, and assume $F'$ satisfies
    \begin{equation}\label{eq:delta_cty}
    \|F-F'\|_{\ell^2} < \delta 
    \end{equation}
    for some $\delta >0$ to be specified later. Choose $N$ sufficiently large such that 
    \begin{equation}
        \label{eq F upper}
        \sum_{j = N}^\infty \|F_j\|_2^2 < \frac{\epsilon^2}{10^4}\,.
    \end{equation}
    We may now apply \eqref{eq:det_to_HS} to all $j \ge N$, and  combining it with \eqref{eq:tail} yields
    \begin{equation}
        \label{eq:tailF}
        \metric(\mathcal{F}_\alpha(F), \mathcal{F}_\alpha(F \mathbf{1}_{[0,N]})) \le  \sum_{j = N}^\infty \|F_j\|_2^2  + 2 \Big( \sum_{j = N}^\infty \|F_j\|_2^2 \Big)^{\frac12} \le \frac{\epsilon}{10}\,.
    \end{equation}
    By Lemma \ref{lem_continuous}, there exists $\delta$ sufficiently small so \eqref{eq:delta_cty} implies 
    \begin{equation}
        \label{eq F' upper}
        \sum_{j = N}^\infty \|F_j'\|_2^2 < \frac{\epsilon^2}{9\cdot 10^3}
    \end{equation}
    and
    \begin{equation}
        \label{eq:finite_cont}
        \metric(\mathcal{F}_\alpha(F \mathbf{1}_{[0,N]}), \mathcal{F}_\alpha(F' \mathbf{1}_{[0,N]})) < \frac{\epsilon}{10}\,.
    \end{equation}
    With the same argument used to prove \eqref{eq:tailF} but now starting from \eqref{eq F' upper} rather than \eqref{eq F upper}, we obtain
    \begin{equation}
        \label{eq:tailF'}
        \metric(\mathcal{F}_\alpha(F'), \mathcal{F}_\alpha(F' \mathbf{1}_{[0,N]})) \le \frac{\epsilon}{9}\,.
    \end{equation}
    The triangle inequality and \eqref{eq:tailF},\eqref{eq:finite_cont}, \eqref{eq:tailF'} together yield that 
    \[
        \metric(\mathcal{F}_\alpha(F), \mathcal{F}_\alpha(F')) < \epsilon
    \]
    whenever \eqref{eq:delta_cty} holds. This completes the proof of continuity. 
\end{proof}
% \ma{can we add some more details to this paragraph, since there's something I don't get? In particular, isn't the one-point compactification sending the complete metric space $\C$ to the Riemann Sphere a uniformly continuous map whose image is not complete?} \lb{Yes, it seems like the explanation is wrong.} \ma{Thanks! Also just realized my supposed counter-example there isn't a counter-example, namely it isn't uniformly continuous. But we both agree to change the explanation.} 
Note that, in the previous proof, $\delta$  depended on $F$. This is necessary, since the map $\mathcal{F}_\alpha$ is not uniformly continuous. 
%This follows for example from the fact that $\ell^2(\mathbb{Z} \cap [0, \infty))$ is complete \ma{also, technically we should be working with $\ell^2 (\Z_{\geq 0}; \con)$, which isn't complete, right?}, but $\mathbf{H}_{\alpha} ^{+}$ , onto which the NLFT maps homeomorphically (see Lemma \ref{lem:inverse_NLFT_cts} below), is not.
 Indeed, already for $n = 1$ and real-valued sequences supported only at $0$, the map 
\[
    T: F_0 \mapsto \begin{pmatrix} \sqrt{1 - F_0^2} & F_0\\ -F_0 & \sqrt{1 - F_0^2}\end{pmatrix}
\]
is not uniformly continuous from $(-1,1)$ into $\mathbf{L}_\alpha^+$. For example, for every $n \ge 1$
\[
    d(T(1 - 2^{-n}), T(1 - 2^{-n-1})) \ge c > 0.
\]

By continuity, the basic properties of the nonlinear Fourier transform extend to half-line sequences.
\begin{lemma}
    \label{lem:propertiesl2}
    The NLFT defined in Lemma \ref{lem:ext_l2} satisfies all the properties listed in Lemmas \ref{lem:properties} and  \ref{lem_MultilinearExpansion}. It satisfies analogs of the  Plancherel identities \eqref{plancherel} and  \eqref{plancherel_ineq}, namely
    \begin{equation}
        \label{eq hl plancherel}
    \sum\limits_{j \in \Z} \log  \det H_j=  \sum\limits_{j \in \Z} \log \det E_j= \log \det A(\infty) \leq  \int\limits_{\T} \log \lvert\det A(z)\rvert + \sum\limits_{i \in \Z} \log \lvert z_i\rvert \, ,
    \end{equation}
    where $(z_i)_{i \in \Z}$ denotes the zeros of $\det A^*$ and $\det D$. Equality in \eqref{eq hl plancherel} holds if and only if the singular inner factor of $\det A^*$ is trivial.
\end{lemma}

\begin{proof}
    Everything but the inequality in \eqref{eq hl plancherel} follows from straightforward limiting arguments. For the inequality, repeat the proof of Lemma \ref{lem plancherel}, but now using inner-outer factorization to write $\det A^*  = o b s$, where $o$ is outer, $b$ is an infinite Blaschke product, and  $s$ is singular inner. Then
    \[
    \log \lvert \det A (\infty)\rvert   =  \log \lvert o (\infty)\rvert + \log \lvert b(\infty)\rvert + \log \lvert s (\infty)\rvert  
    \]
    \[
    \leq \int\limits_{\T} \lvert \log o\rvert  + \sum\limits_{i\in\mathbb Z}  \log\lvert z_i\rvert  = \int\limits_{\T} \log \lvert \det A^* \rvert +  \sum\limits_{i\in\mathbb Z}  \log\lvert z_i\rvert \, . \qedhere
    \]
\end{proof}

\subsection{Layer stripping for the half-line}

Having defined the NLFT on $\ell^2(\mathbb{Z}_{\geq 0} ; \con)$, we now show that it is a homeomorphism between suitable spaces using the Layer Stripping Algorithm \ref{algo:layer_stripping}.

As a first step, we extend Lemmas \ref{layerstripping} and \ref{lem:layer_stripping_finite} to  infinite sequences.
\begin{lemma}\label{layerstripping_infinite}
    If
    \[
    \begin{pmatrix}
        A & B \\ C & D
    \end{pmatrix} = \mathcal{F}_{\alpha} (F)
    \]
    for some $F \in \ell^2 (\Z_{\geq 0}; \con)$, then \eqref{eq:stripping_def} holds.
\end{lemma}
\begin{proof}
    By the definition of $\mathcal{F}_\alpha$ given in Lemma \ref{lem:ext_l2}, we have
    \begin{equation}
        \label{eq:dL_lmimit}
        \mathcal{F}_\alpha(F)= \lim_{\ell \to \infty} \mathcal{F}_\alpha(F \mathbf{1}_{[0,\ell]}) =: \lim_{\ell \to \infty} \begin{pmatrix} A_0 ^\ell & B_0 ^\ell \\ C_0 ^\ell & D_0 ^\ell\end{pmatrix}\,,
    \end{equation}
    where the limit is in $(\mathbf{L}_{\alpha}, \metric)$.
     Lemma \ref{layerstripping} states that
    $$
        Y_\alpha(F_0) = S_\alpha(B_0^{\ell} (0) D_0 ^{\ell}(0)^{-1})
    $$
    for all $\ell$. Then \eqref{eq:stripping_def} follows from continuity of $S_\alpha$, and  
    $$
        \lim_{\ell \to \infty} B_0 ^{\ell} (0) D_0 ^{\ell} (0)^{-1} = B_0(0) D_0 (0)^{-1}\,.
    $$
    The latter holds since by the convergence \eqref{eq:dL_lmimit} in $\mathbf{L}_\alpha$, we have that $B_0 ^{\ell} (0) \to B(0)$, that $D_0 ^{\ell} (0) \to D(0)$, and that $\det(D_{0} ^{\ell} (0))$ remains bounded away from zero as $\ell \to \infty$. 
\end{proof}

\begin{lemma}
    \label{lem:layer_stripping}
If $M_j \in \mathbf{U}_\alpha^+$, then $M_{j+1} \in \mathbf{U}_\alpha^+$, where $M_{j+1}$ is defined as in \eqref{eq:Mj_to_next}.
If $\nlf = \mathcal{F}_{\alpha} (F)$ for some $F \in \ell^2 (\Z_{\geq 0}; \con)$, then $F$ is the output of the Layer Stripping Algorithm \ref{algo:layer_stripping}.
\end{lemma}

\begin{proof}
    Let $\nlf_j \in \mathbf{U}_\alpha^+$, and assume to the contrary that $M_{j+1} \not\in \mathbf{U}_\alpha^+$. By Lemma \ref{lem:layer_stripping_finite}, $\nlf_{j+1}$ is unitary a.e.\ and Properties \ref{item:1st_prop_H}--\ref{item:2nd_prop_H} hold in Definition \ref{defH} for $\nlf=\nlf_{j+1}$. For our assumption to hold, Property \ref{item:3rd_prop_H} must then fail, i.e., there exists a factorization $M_{j+1} = M' I$ as in \eqref{eq:inner_factor}.
    Then, using that $\ad(z)$ is a group homomorphism which preserves the block diagonal matrix functions $I$,
    \begin{equation}
        \label{eq right multiplication}
        M_j =  Y_\alpha(F_j) \ad(z)(M_{j+1}) = Y_\alpha(F_j) \ad(z)(M' I) = Y_\alpha(F_j) \ad(z)(M') I.
    \end{equation}
    This is a factorization of the form \eqref{eq:inner_factor}, showing that $M_j \not\in \mathbf{U}_\alpha^+$, contradicting our assumption. Thus $\nlf_{j+1} \in \mathbf{U}_{\alpha} ^+$. 

We turn to the statement about the Layer Stripping Algorithm. By Lemma \ref{layerstripping_infinite}, it will follow from the claim that $M_j= \mathcal{F}_\alpha( (F_{k-j}\mathbf{1}_{[j,\infty)}(k))_{k \in \mathbb{Z}})$ for all $j\geq 0$. But the latter claim  follows by induction on $j \geq 0$, Lemma \ref{layerstripping} and the symmetries of Lemma \ref{lem:properties}. 
%We claim that for all $j \geq 0$, if there exists $F \in \ell^2 (\Z_{\geq 0 } ; \con)$ such that  $M_j= \mathcal{F}_\alpha( (F_{k-j}\mathbf{1}_{[j,\infty)}(k))_{k \in \mathbb{Z}})$, then 
%    \begin{equation}
%        \label{eq_defFn}
%        Y_\alpha(F_j) = S_\alpha(B_j(0) D_j(0) ^{-1}) \, , 
%    \end{equation}
%    and $M_{j+1}$
%    is the NLFT of the sequence $(F_{k-(j+1)} \mathbf{1}_{[j+1,\infty)}(k))_{j \in \mathbb{Z}}$. The statement about the Layer Stripping Algorithm \ref{lem:layer_stripping} will follow once we show the claim. To prove the claim, without loss of generality take $j=0$.  
%
%    The statement about $M_{j+1}$ now follows  by combining the two properties \ref{item:ordered_mult} and \ref{item:translation} of the NLFT for half line sequences, as stated in Lemma \ref{lem:propertiesl2}. 
\end{proof}

By Lemma \ref{lem:layer_stripping}, the Layer Stripping Algorithm \ref{algo:layer_stripping}
has a well-defined output for any input $\nlf \in  \mathbf{U}_\alpha^+$. Thus, given such an $\nlf$
which \textit{a priori} is not known to be a nonlinear Fourier transform, 
we associate to it the sequence $F_0, F_1, \dotsc$ of coefficient matrices obtained by applying the Layer Stripping Algorithm \ref{algo:layer_stripping}.
We also define $M_0 :=M$ and, for $ j \geq 0$, define $M_{j+1}$ as in \eqref{eq:Mj_to_next}. We will use the definitions $F_j$ and $M_j$ in the remainder of the subsection without explicitly mentioning them again. We will need the inequality from the following lemma which, in combination with the Plancherel identity \eqref{plancherel}--\eqref{plancherel_ineq}, will yield that the only inner factors of a nonlinear Fourier transform as in \eqref{eq:inner_factor} must be trivial.

\begin{lemma}
   Let $\nlf \in \mathbf{U}_{\alpha} ^+$. For all $m \geq 0$, we have
    \begin{equation}
        \label{eq:plancherel_bound}
        \sum_{j =0}^{m-1} \lvert \log  \det E_j \rvert \le \lvert \log  \det A(\infty)\rvert\,.
    \end{equation}
\end{lemma}

\begin{proof}
    We write
    \[
        M = \mathcal{F}_\alpha(F_j \mathbf{1}_{[0,m)}(j)) \ad(z^{m})(M_{m}) =: \begin{pmatrix} A' & B' \\ C' & D' \end{pmatrix}\begin{pmatrix} A_{m} & B_{m} z^{m} \\ C_{m}z^{-m} & D_{m} \end{pmatrix}.
    \]
    By Lemma \ref{lem_MultilinearExpansion}, the block $B'$ has frequency support in $[0,m-1]$. By Lemma \ref{lem:layer_stripping}, we have $\nlf_{m} \in \mathbf{U}_{\alpha} ^+$, and so $z^{-m} C_{m}(z)$ vanishes at $\infty$ to order $m$. Hence
    $$
        A(\infty) = A'(\infty) A_{m}(\infty) =   \left ( \prod_{j=0}^{m-1} E_j \right ) A_{m}(\infty) ,
    $$
    where we also used \eqref{eq:constant_Fourier_A_D}.
   Noting that all logarithms in \eqref{eq:plancherel_bound} are nonpositive, the lemma now follows by taking determinants together with the maximum principle which ensures $\det A_{m}( \infty) \le 1$.
\end{proof}

\begin{lemma}\label{lem:injectivity_half_line}
    The map $\mathcal{F}_\alpha$ is injective from $\ell^2(\mathbb{Z}_{\geq 0} ; \con)$ into $\mathbf{H}_\alpha^+$.
\end{lemma}

\begin{proof}
We claim that $\mathcal{F}_{\alpha}$ maps $\ell^2 (\Z_{\geq 0}; \con)$ into $\mathbf{H}_{\alpha} ^+$. Then the claim and Lemma \ref{lem:layer_stripping} together show that Layer Stripping is a left inverse for $\mathcal{F}_{\alpha}$, implying the injectivity of the latter.

    We now show that $\mathcal{F}_{\alpha}$ maps $\ell^2 (\Z_{\geq 0}; \con)$ into $\mathbf{H}_{\alpha} ^+$. Let $\nlf = \mathcal{F}_\alpha(F)$. By definition, $\nlf$ is a limit of $SU(2n)$-valued functions, and so must be a.e.\ $SU(2n)$-valued. Again using limits, Property \ref{item:2nd_prop_H} of Definition \ref{defH} holds, whereas Property \ref{item:1st_prop_H} holds by Lemma \ref{lem_MultilinearExpansion} for $(c,d)=(0,\infty)$. We must now show Property \ref{item:3rd_prop_H} holds. Suppose that $M = M' I$ is a factorization as in \eqref{eq:inner_factor}. %Then  the block $I_2(0)$ of $I(0)$ is invertible, because $D(0) \in \grp_{\alpha(1)}$.  
    
    Note that the Layer Stripping Algorithm \ref{algo:layer_stripping} applied to $M$ and $M'$ produces the same coefficient sequence. Indeed, the right side of \eqref{eq:stripping_def} clearly does not change if $M'$  is multiplied by the diagonal block matrix $I$. Furthermore, as the computation \eqref{eq right multiplication} shows, the stripping procedure also commutes with right multiplication by block diagonal matrix functions.

    It follows from \eqref{eq:plancherel_bound} that 
    $$
        \sum_{j =0}^\infty \lvert \log  \det E_j \rvert \le \lvert \log  \lvert\det A'(\infty)\rvert \rvert 
        = \lvert \log \lvert \det A(\infty) \rvert \rvert - \lvert \log \lvert \det I_1 (\infty) \rvert \rvert \,.
    $$
    By the Plancherel identity \eqref{eq hl plancherel},
    % finite case: \eqref{plancherel}--\eqref{plancherel_ineq}
    we must have $\lvert \det I_1 (\infty)\rvert = 1$. By the maximum principle, $\det I_1$ is constant. Since $I_1 ^*$ is an inner function, it follows  from Lemma \ref{lem:cst_det_matrix_inner} that $I_1$ is constant. Similarly, but using the version of \eqref{plancherel}--\eqref{plancherel_ineq} for $D$ as in Remark \ref{rem:AD}, we also obtain  that $I_2$ is constant. Since $A(\infty), A'(\infty) \in \grp_{\alpha(0)}$ and $D(0), D'(0) \in \grp_{\alpha(1)}$, it follows that $I_1 \in \grp_{\alpha(0)}$ and $I_2 \in \grp_{\alpha(1)}$, and so both must equal $\Id$.
\end{proof}

We next  verify surjectivity of the nonlinear Fourier transform onto $\mathbf{U}_\alpha^+$. 

\begin{lemma}\label{lem surj}
    The map $\mathcal{F}_{\alpha}$ is surjective from $\ell^2(\mathbb{Z}_{\geq 0} ; \con)$ onto $\mathbf{U}_\alpha^+$.
\end{lemma}

\begin{proof}
    Let $M \in \mathbf{U}_\alpha^+$, and let $F$ be the sequence produced by the layer stripping algorithm applied to $M$. By \eqref{eq:plancherel_bound}, we have 
    \begin{equation}
        \label{eq:F_squaresum}
        \sum_{j \in \mathbb{Z}\cap[0,\infty)} \frac{1}{2} \lvert \log \det(\Id - F_jF_j^*) \rvert = \sum_{j \in \mathbb{Z} \cap [0,\infty)} \lvert \log E_j \rvert < \infty\,,
    \end{equation}
    and hence $F \in \ell^2(\mathbb{Z}_{\geq 0}; \con)$. Defining
    \[
        M' := \mathcal{F}_\alpha(F),
    \]
     it then remains to show that $M' = M$. Introducing the notation 
    \[
        M_{\le N} := \mathcal{F}_\alpha(F\mathbf{1}_{[0,N]}),
    \]
     we then have $M_{N+1} = \ad(z^{-N-1})(M_{\le N}^{-1} M)$. We also denote analogously
    \[
        M'_{N+1} :=  \ad(z^{-N-1})(M_{\le N}^{-1} M') = \mathcal{F}_\alpha((F_{j + N + 1} \mathbf{1}_{\{j \geq 0\}})_{j \in \mathbb{Z}}),
    \]
    where the second equality follows from the multiplicativity property \ref{item:ordered_mult} of Lemma \ref{lem:properties}. Denote the blocks of $M_N '$ by $A_N '$, $B_N '$, $C_N '$ and $D_N '$. We have, for all $N$,
    \begin{equation}
        \label{eq surj matrix}
        {M'}^{-1} M = \ad(z^{N+1})({M'_{N}}^{-1} M_{N}) 
        = \begin{pmatrix}
            * & z^{N+1}({A'}_{N}^*B_{N} + {C'}_{N}^* D_{N})\\
            * & *
        \end{pmatrix}\,.
    \end{equation}
    Because $\nlf_{N} \in \mathbf{U}_{\alpha} ^+$ by Lemma \ref{lem:layer_stripping}, the matrices $B_{N}$ and $C_{N}^*$ have frequency support in $[0, \infty)$. Also recalling that $A_N '$ and $D_N '$ have frequency support in $(-\infty, 0]$ and $[0, \infty)$ by Lemmas \ref{lem:propertiesl2} and \ref{lem_MultilinearExpansion}, we then have that the upper right entry of the matrix \eqref{eq surj matrix} has frequency support in $[N+1, \infty)$. Since the left side of \eqref{eq surj matrix} is independent of $N$, it follows that its upper right entry is zero. %A similar argument applies to the lower left entry. 
    In particular, we have for all $N$:
    $$
        B_{N} = - {A'_{N}}^{-*} {C'_{N}}^* D_{N}\,.
    $$
    By continuity of $\mathcal{F}_\alpha$, the functions $M'_{N}$ converge to the identity in $L^2(\mathbb{T})$. In particular, $A_{N}'$ converges to the identity and $C'_{N}$ converges to zero. By passing to a subsequence if necessary, this convergence is pointwise a.e. Combining this with the boundedness of $D_N$ it follows that, along the chosen subsequence, pointwise a.e.\ we have 
    \begin{equation}\label{eq:BN_to_0}
        B_N  \to 0.
    \end{equation}
    This alongside with boundedness of $B_N$ and  dominated convergence yield that $B_{N}$ converges to zero in $L^2(\mathbb{T})$. A similar argument applies to $C_{N}$. 
    From the Layer Stripping Algorithm \ref{algo:layer_stripping}, it follows that for $N < L$ 
    \begin{equation}\label{eq:MN_to_ML}
        M_N = \mathcal{F}_\alpha( (F_{j+N} \mathbf{1}_{\{0\leq j < L-N\}})_{j \in \mathbb{Z}}) \ad(z^{L-N})(M_L). 
    \end{equation}
    By unitarity of $\ad(z^{L-N})(M_L)$, as in the proof of \eqref{eq product l2}, we have
    \[
        \Big(\int_{\T} \|M_N - \ad(z^{L-N})(M_L)\|_2^2 \Big)^{1/2} \le \metric(\mathcal{F}_\alpha( (F_{j+N} \mathbf{1}_{\{0\leq j < L-N\}})_{j \in \mathbb{Z}}), \Id) \, ,
    \]
    which, by Lemma \ref{lem:dist_Id}, is at most
    \[
        \sum_{j =N}^{L-1} \lvert\log\det(\Id - F_jF_j^*)\rvert + 2\Big(\sum_{j =N}^{L-1} \lvert\log\det(\Id - F_jF_j^*)\rvert\Big)^{\frac12} \, .
    \]
    Since the diagonal blocks of $\nlf_N$ and $\ad(z^{L-N}) \nlf_L$ are $A_N, D_N$ and $A_{L}, D_{L}$, respectively, then  \eqref{eq:F_squaresum} implies that $A_{N}^*$ and $D_{N}$ are Cauchy in $H^2(\mathbb{D})$. Hence they converge to $H^2 (\D)$ functions $I_1 ^*$ and $I_2$. Thus
    $$
        M_{N} \to \begin{pmatrix}
            I_1 & 0\\
            0 & I_2
        \end{pmatrix} \qquad \text{in $L^2(\mathbb{T})$.}
    $$
    Since all $M_{N}$ are unitary on the torus, the same applies to $I_1$ and $I_2$ a.e.  Taking limits in \eqref{eq surj matrix} and using that $M_N' \to \Id$ yields
    \begin{equation*}
        M = M' \begin{pmatrix}
            I_1 & 0\\
            0 & I_2
        \end{pmatrix} = M' I\,.
    \end{equation*}
    Recalling that $M \in \mathbf{U}_{\alpha}^+$, then Property \ref{item:3rd_prop_H} of  Definition \eqref{defH} must hold, i.e.,  $I_1 = I_2 = \Id$. Thus $M = M'$ as needed. 
\end{proof}

As a corollary, we obtain $\mathbf{U}_\alpha^+ = \mathbf{H}_\alpha^+$.
\begin{cor}
 \label{cor:H_equals_U}
   The spaces $\mathbf{H}_{\alpha} ^+$, $\mathbf{U}_{\alpha} ^+$ and $\mathcal{F}_{\alpha} (\ell^2 (\Z_{\geq 0} ; \con))$ are all equal. 
\end{cor}
\begin{proof}
    By definition, we have $\mathbf{H}_{\alpha} ^+ \subseteq \mathbf{U}_{\alpha} ^+$. But Lemmas \ref{lem:injectivity_half_line} and \ref{lem surj} yield
    \[
    \mathbf{U}_{\alpha} ^+ = \mathcal{F}_{\alpha} (\ell^2 (\Z_{\geq 0} ; \con)) \subseteq \mathbf{H}_{\alpha} ^+ \, .
    \]
    Thus all three spaces must be equal. 
\end{proof}

We finally show continuity of the inverse map on $\mathbf{H}_\alpha^+$.

\begin{lemma}\label{lem:inverse_NLFT_cts}
    The inverse map $\mathcal{F}_\alpha^{-1}$ is continuous from the space $\mathbf{H}_\alpha^+$ into $\ell^2(\mathbb{Z} ; \con)$.
\end{lemma}

\begin{proof}
    By the continuity of $S_\alpha$ from Lemma \ref{lem:layer_strip_recover}, the finite Layer Stripping Algorithm \ref{algo:layer_stripping} truncated after $N+1$ steps   
    \[
        \mathbf{H}_\alpha^+ \to \ell^2(\mathbb{Z} \cap [0,N]; \con), \quad M \mapsto (F_0, \dotsc, F_N)\, ,
    \]
    is continuous for each $N$. 

    Fix $\nlf \in \mathbf{H}_{\alpha} ^+$ and fix  $\epsilon > 0$. Let $\nlf' \in \mathbf{H}_{\alpha} ^+$, and assume $d(\nlf, \nlf')< \delta$. We show that for $\delta$ sufficiently small, that the layer-stripping outputs $F$ and $F'$ are close in $\ell^2$. First pick $N$ sufficiently large such that 
    \begin{equation}
        \label{eq:fj small}
        \Big(\sum_{j > N} \|F_j\|_2^2\Big)^{\frac12} < \epsilon\, .
    \end{equation}
    By the triangle inequality,  this implies
    \begin{equation} 
        \label{eq:homeo1}
        \Big(\sum_{j > N} \|F_j - F_j'\|_2^2\Big)^{\frac 12} <\epsilon + \Big(\sum_{j >N} \|F_j'\|_2^2\Big)^{\frac12}\,.
    \end{equation}
    By  continuity of the finite layer stripping algorithm, there exists $0 < \delta < \epsilon$ such that, whenever $\metric(M,M') < \delta$,  
    \begin{equation}
        \label{eq small n 2}
        \sum_{j \le N} \lvert\log \det(\Id - F_j F_j^*) - \log \det(\Id - F_j' {F_j'}^*) \rvert < \epsilon.
    \end{equation}
    By \eqref{plancherel} and the definition of the metric $\metric$, we also have when $\metric(M, M') < \delta$ 
    \[
        2\epsilon > 2\delta > \Big\lvert\sum_{j \ge 0} \log \det(\Id - F_j F_j^*) - \log \det(\Id - F_j' {F_j'}^*)\Big\rvert\,.
    \]
    Combining with \eqref{eq small n 2},
    \begin{equation}
        \label{eq:homeo2}
        3 \epsilon > \sum_{j > N} \lvert \log \det(\Id - F_j'{F_j'}^*)\rvert - \sum_{j > N} \lvert \log \det(\Id - F_j{F_j}^*)\rvert\,.
    \end{equation}
    For any contractive matrix $F'$, using that $e^x \ge 1 + x$ and that all singular values of $F'$ are at most $1$,
    $$
        \|F'\|_2^2 = \sum_{i=1}^n \lambda_i(F')^2 \le -  \sum_{i=1}^n \log(1 - \lambda_i(F')^2) = \lvert  \log \det(\Id - F'{F'}^*)\rvert\,.
    $$
    Applying this to the matrices $F_j'$ in \eqref{eq:homeo2} and combining with \eqref{eq:homeo1} yields
    $$
        \Big(\sum_{j > N} \|F_j - F_j'\|_2^2\Big)^{\frac12} \le \epsilon + \Big(3\epsilon + \sum_{j > N} \lvert \log \det(\Id - F_jF_j^*)\rvert\Big)^{\frac12}\,.
    $$
    By \eqref{eq:fj small}, we have $\|F_j\|_2 < \epsilon$ for $j > N$ . For $\epsilon$ sufficiently small, this implies 
    $$
        \lvert \log \det(\Id - F_j F_j^*)\rvert \le 2 \|F_j\|_2^2\,.
    $$
    Thus we finally obtain 
    $$
        \Big(\sum_{j > N} \|F_j - F_j'\|_2^2\Big)^{\frac12} \le \epsilon + \Big(3\epsilon + 2 \sum_{j > N} \|F_j\|_2^2 \Big)^{\frac12}
        < \epsilon + (3\epsilon + 2 \epsilon^2)^{\frac12}\,. 
    $$
    This completes the proof of continuity of the inverse map.    
\end{proof}
Computing the first nonlinear Fourier coefficient of an element $\nlf \in \mathbf{H}_{\alpha} ^+$ is a Lipschitz continuous process if we assume that the singular values of $D(0)$ are bounded from below.
\begin{lemma}\label{lem:first_layer_strip_Lip}
    Let $\epsilon > 0$, and let $\nlf, \nlf' \in \mathbf{H}_{\alpha} ^+$. If all the singular values of $D(0), D'(0)$ are at least $\epsilon$, then there exists a constant $C_{\epsilon, n}$ for which we have the Lipschitz bound
    \[
    \|F_0 - F_0 '\|_{\infty} \leq C_{\epsilon, n} \|\nlf - \nlf'\|_{L^2} \, .
    \]
\end{lemma}
\begin{proof}
In what follows, we let the constant $C_{\epsilon, n}$ change line to line.
By \eqref{eq:notation_conv}, the lemma will follow from 
\begin{equation}
    \label{e:Lip_Ya}
\|Y_{\alpha}(F_0) - Y_{\alpha}(F_0 ')\|_{\infty} \leq C_{\epsilon, n} \|\nlf - \nlf '\|_{L^2} \, ,
\end{equation}
 or, equivalently rewritten  using \eqref{eq:stripping_def}, 
\begin{equation}\label{eq:stripping to RH_factor}
\|S_{\alpha}( B (0) D (0)^{-1}) - S_{\alpha}( B ' (0) D ' (0)^{-1})\|_{\infty} \leq C_{\epsilon, n} \|\nlf - \nlf '\|_{L^2} \, .
\end{equation}
We claim the uniform bound
\begin{equation}\label{eq:uniform_bd}
\| B (0) D (0)^{-1}\|_{\infty}, \| B ' (0)   D ' (0)^{-1}\|_{\infty} \leq \epsilon^{-1} 
\end{equation}
and the difference bound
\begin{equation}\label{eq:difference_bd}
\| B (0) D (0)^{-1} -  B ' (0)   D ' (0)^{-1}\|_{\infty} \leq C_{\epsilon, n} \|\nlf - \nlf '\|_{L^2} \, . 
\end{equation}
Using the Lipschitz continuity of $S_{\alpha}$ on the set of matrices with singular values bounded by $C_{\epsilon, n}$, recall Lemma \ref{lem:layer_strip_recover}, estimates \eqref{eq:uniform_bd}--\eqref{eq:difference_bd} imply \eqref{eq:stripping to RH_factor}. 

To see \eqref{eq:uniform_bd}, the mean value property  applied to $B$ and unitarity of  $\nlf$ together imply
\[
\|B(0)D (0) ^{-1}\|_{\infty} \leq \|D (0) ^{-1}\|_{\infty} \leq \epsilon^{-1} \, ,
\]
and similarly for $B_+ (0) D_+ (0) ^{-1}$.

As for \eqref{eq:difference_bd}, its left side is at most 
\[
 \|B (0) - B '(0)\|_{\infty} \|D (0) ^{-1}\|_{\infty} +\|B ' (0)\|_{\infty} \|D (0) ^{-1} - D ' (0) ^{-1}\|_{\infty} \,.
\]
By unitarity of $\nlf '$ and the mean value property, we have $\|B ' (0)\|_{\infty} \leq 1$. Combined with rewriting the difference of inverses as the difference times the inverses, the left side of \eqref{eq:difference_bd} is at most
\[
 \|B (0) - B '(0)\|_{\infty} \|D(0)^{-1}\|_{\infty} + \|D(0)^{-1}\|_{\infty} \|D'  (0)  - D  (0) \|_{\infty} \|D' (0)^{-1}\|_{\infty}  \, .
\]
Using the lower bound on the singular values of $D(0)$ and $D'(0)$, this is at most 
\[
\epsilon^{-1} \|B (0) - B '(0)\|_{\infty} + \epsilon^{-2}  \|D ' (0)  - D  (0) \|_{\infty}  \, .
\]
Each term is then controlled using the mean value property, e.g.,
\[
\|B (0) - B '(0)\|_{\infty} \leq \int\limits_{\T} \|B - B ' \|_{\infty} \leq  \|B - B '\|_{L^2} \leq  \|\nlf - \nlf '\|_{L^2} \, .
\]
This completes the proof of the lemma.
\end{proof}

\subsection{The left half-line}

As a consequence of the reflection symmetry in Lemma \ref{item:reflection_2}, the results from the previous section apply just as well to sequences supported on the left half-line. For the convenience of the reader, we provide the resulting statements explicitly.  

Let $\mathbf{L}^-_{\alpha, 0}$ be the space of  functions 
\[
    M = \begin{pmatrix} A & B\\ C & D\end{pmatrix} \in \mathbf{L}_\alpha
\]
such that $B \in H^2(\mathbb{D}^*)$, and $C \in H^2(\mathbb{D})$, and $B(\infty) = C(0) = 0$. Equipping $\mathbf{L}^-_{\alpha, 0}$ with the metric $\metric$, it is then a closed subspace of $\mathbf{L}_\alpha$. 

    We define $\mathbf{H}_{\alpha,0} ^-$ to be the space of all matrix functions
    \begin{equation}\label{labeling_M_left}
        \nlf: \mathbb{T} \to SU(2n), \qquad M = \begin{pmatrix}
            A&B\\
            C&D
        \end{pmatrix}
    \end{equation}
    satisfying the following properties:
    \begin{enumerate}
        \item \label{item:1st_prop_H-} $A^*,z^{-1}B^*, z^{-1} C, D  \in H^2(\mathbb{D})$;
        \item \label{item:2nd_prop_H-} $A(\infty)\in \mathcal{G}_{\alpha(0)}$ and $D(0)\in \mathcal{G}_{\alpha(1)}$;
        \item  if there exist $I_1 ^*, I_2 \in H^2 (\D)$, both  unitary a.e.\ 
    on $\T$,
    and there exists $\nlf '$ of the form \eqref{labeling_M_intro} satisfying Properties \ref{item:1st_prop_H-} and \ref{item:2nd_prop_H-}, for which 
    \begin{equation}
        \label{eq:inner_factor_H-}
       \nlf = \begin{pmatrix} I_1 & 0 \\ 0 & I_2 \end{pmatrix} \nlf'  \, ,
    \end{equation}  
    then $J = K =\Id$. 
    \end{enumerate}
As with Corollary \ref{cor:H_equals_U}, the space $\mathbf{H}_{\alpha, 0} ^{-}$ may be equivalently defined by replacing $SU(2n)$ by $U(2n)$ in \eqref{labeling_M_left}. We also equip $\mathbf{H}^-_{\alpha, 0}$ with the metric $\metric$.

% Define $\mathbf{H}^-_{\alpha,0}$ to be the subset of  functions 
% $$
%     M = \begin{pmatrix} A & B\\ C & D \end{pmatrix} \in \mathbf{L}^-_{\alpha, 0}
% $$
% for which  there exists no  inner factorization of the following kind:
% \begin{equation}
%     \label{eq:inner_factor_right}
%     \begin{pmatrix} A & B\\ C & D \end{pmatrix} = \begin{pmatrix} J & 0 \\ 0 & K \end{pmatrix} \begin{pmatrix} A' & B' \\ C' & D' \end{pmatrix} 
% \end{equation} 
% such that ${A'}^*, {B'}^*, C'$ and $D'$ have analytic extensions to $\mathbb{D}$, both matrices on the right hand side belong to  $SU(2n)$ for $z \in \mathbb{T}$, and $J^*$ and $K$ are inner matrix functions on $\mathbb{D}$.  
% We also equip $\mathbf{H}^-_{\alpha, 0}$ with the metric $d$. 

\begin{cor}
    \label{cor:left_half}
    The map $\mathcal F_\alpha$ extends to a homeomorphism 
    from $
        \ell^2(\mathbb{Z}_{<0} ; \con) $ to $ \mathbf{H}_{\alpha, 0} ^{-}\,.
    $
\end{cor}

% The inverse map can be computed using a straightforward modification of the layer stripping algorithm. We state it here for convenience: \lb{Maybe drop this part, is it used anywhere?}

% \begin{itemize}
%     \item Set $M_{\le -1} = M$.
%     \item Iteratively, for $n \ge 1$:
%     \begin{itemize}
%         \item Set $F_{-n} = \tilde T_{-n}(M_{\le -n})$, where $\tilde T_{-n}$ is defined below.
%         \item Set $M_{\le -n-1} = M_{-n-1}(M_{\le -n})$, where $M_{-n-1}$ is defined below. 
%     \end{itemize}
% \end{itemize}
% The map $\tilde T_{-n}$ is now given by 
% $$
%     \tilde T_{-n}(M) = \tilde T(A(\infty)^{-1} B(\infty))\,,
% $$
% \lb{$\tilde T$ is the map sending $E^{-1} F$ to the block $SU'(2n)$ matrix with upper row $E \ F$.}
% and the map $M_{-n-1}$ by
% $$
%     M_{-n -1}(M) = M \begin{pmatrix} z^{-n/2} & 0\\ 0 & z^{n/2} \end{pmatrix} \mathbf{F}_{-n}^* \begin{pmatrix}  z^{n/2} & 0\\ 0 & z^{-n/2}  \end{pmatrix}\,.
% $$

\subsection{Extension to full line \texorpdfstring{$\ell^2$}{l2} sequences}
We define for a general sequence $F \in \ell^2(\mathbb{Z}; \con)$ its nonlinear Fourier transform by
\begin{equation}\label{def:NLFT_full_line}
    \mathcal{F}_\alpha(F) := \mathcal{F}_\alpha(F \mathbf{1}_{(-\infty,0)}) \mathcal{F}_\alpha(F \mathbf{1}_{[0,\infty)})\,,
\end{equation}
where the nonlinear Fourier transform on the right side is the one for half-line sequences  defined in the previous subsection.
This coincides with the previous definitions for half-line sequences and, in light of property \ref{item:ordered_mult} from Lemma \ref{lem:properties},  with the one  for finitely supported sequences. 

\begin{lemma}
    The map $\mathcal F_\alpha$ defined in \eqref{def:NLFT_full_line} is continuous from $\ell^2(\mathbb{Z}; \con)$ into $\mathbf{L}_\alpha$.
\end{lemma}

\begin{proof}
    Because of Definition \eqref{def:NLFT_full_line}, Lemma \ref{lem:ext_l2} and Corollary \ref{cor:left_half}, it only remains to verify that  multiplication defines a continuous map from $\mathbf{H}_{\alpha,0}^- \times \mathbf{H}_\alpha^0$ into $\mathbf{L}_\alpha$.
    Let $(M_-, M_+)$ and $(M_-', M_+')$ be elements of $\mathbf{H}_{\alpha,0}^- \times \mathbf{H}_\alpha^0$. Then we estimate
    $$
        \Big(\int_\mathbb{T} \|M_- M_+ - M_-' M_+'\|_2^2\Big)^{\frac12} 
    $$
    $$
        \le \Big(\int_\mathbb{T} \|M_- M_+ - M_- M_+'\|_2^2 \Big)^{\frac12} + \Big(\int_\mathbb{T} \|M_- M_+' - M_-' M_+'\|_2^2 \Big)^{\frac12}.
    $$
    Since the Hilbert--Schmidt norm remains invariant under multiplication by unitary matrices, this equals
    \begin{equation}
        \label{eq:full_line_cont}
        \Big(\int_\mathbb{T} \|M_+ - M_+'\|_2^2 \Big)^{\frac12} + \Big(\int_\mathbb{T} \|M_-  - M_-'\|_2^2 \Big)^{\frac12}.
    \end{equation}
    Denote
    $$
         \begin{pmatrix} A & B\\ C& D \end{pmatrix}:= M_- M_+ \, .
    $$
    Then, by definition of the spaces $\mathbf{H}_{\alpha, 0} ^{-}$ and $\mathbf{H}_{\alpha} ^{+}$, we have $B_-(\infty) = C_-(0) = 0$ and $B_+(0)$ and $C_+(\infty)$ are finite. Matrix multiplication along with the frequency supports dictated by the spaces $\mathbf{H}_{\alpha, 0} ^{-}$ and $\mathbf{H}_{\alpha} ^{+}$ then gives the identity
    \begin{equation}\label{eq:mult_mean_value}
        A(\infty) = A_-(\infty)A_+(\infty).
    \end{equation}
    Thus, adapting similar notation  for the functions $M_+'$ and $M_-'$,
    $$
        \lvert \log\det(A(\infty)) - \log \det(A'(\infty)) \rvert
    $$
    \begin{equation}
        \label{eq:full_line_cont2}
        \le \lvert \log\det(A_-(\infty))- \log\det(A'_-(\infty))\rvert + \lvert \log\det(A_+(\infty))- \log\det(A'_+(\infty))\rvert\,.
    \end{equation}
    Adding \eqref{eq:full_line_cont} and \eqref{eq:full_line_cont2} shows that
    \[
        \metric(M_-M_+, M'_-M_+') \le \metric(M_-, M_-') + \metric(M_+, M_+'),
    \]
    which yields the required continuity. 
\end{proof}

%\lb{What else do we need here? All properties still carry over by continuity, we could state that. We could also give a remark that any other definition (that is as limit in $\mathbf{L}$ of finitely supported sequences) gives the same result. This is a direct consequence of continuity also, so maybe it is not worth stating it.}
\begin{remark}
    The analog of Lemma \ref{lem:propertiesl2} holds for the map $\mathcal{F}_{\alpha}$ on $\ell^2 (\Z; \con)$. 
\end{remark}

\section{The Riemann--Hilbert factorization problem}\label{sec:RH}

Our goal in this section is to prove Theorem \ref{thm:NLFT_outer_inverse}. Namely, we show that every $\nlf \in \mathbf{B}_{\alpha}$ is the NLFT of a unique sequence $F \in \ell^2 (\Z; \contract)$; then, for each $\epsilon >0$, we show that the inverse map $M\mapsto F$ is Lipschitz continuous on the set $\mathbf{B}_\alpha ^{\epsilon}$ of $\nlf \in \mathbf{B}_{\alpha}$
whose upper right block $B$ satisfies $\|B\|_{L^{\infty}} < 1- \epsilon$.

Let $L^2 (\T ; \C^{m^2})$ denote the Hilbert space of $m \times m$ mvfs $X$ for which 
\[
\|X\|_{L^2} ^2 = \int\limits_{\T} \|X\|_2 ^2 < \infty \, .
\]
The inner product of $X, Y \in L^2 (\T; \C^{m^2})$ is given by
\begin{equation}\label{eq:inner_prod_def}
\langle X, Y \rangle = \int\limits_{\T} \sum\limits_{1 \leq i,j \leq m} X_{i j} \overline{Y_{i j}} \, .
\end{equation}
Given an mvf $Q$ in $L^{\infty} (\T; \C^{m^2})$, the following duality formulae hold:
\begin{equation}\label{eq:duality}
\langle X, Q Y \rangle = \langle Q^* X , Y \rangle \, , \qquad \langle X, Y Q \rangle = \langle X Q^*  , Y \rangle \, .
\end{equation}

\subsection{Overview of the proof of Theorem \ref{thm:NLFT_outer_inverse}}

By the definition of the NLFT \eqref{def:NLFT_full_line} and its injectivity on the half-lines, we achieve the aforementioned  goals by showing that $\nlf \in \mathbf{B}_{\alpha}$ factors uniquely as a product
\begin{equation}\label{eq:factorization_0}
\nlf = \nlf_{-} \nlf_+ \, ,
\end{equation}
for $\nlf_{-} \in \mathbf{H}_{\alpha, 0} ^-$ and $\nlf_+ \in \mathbf{H}_{\alpha} ^+$. This  is known as a Riemann--Hilbert factorization problem. 

In Subsection \ref{subsection:defn_M} below, we define an unbounded operator $\op$ on a certain Hilbert space $\spce$ and prove that $i\op$ is self-adjoint. As a consequence, the equation
\[
    (\Id + \op) X = \Id
\]
has a unique solution $X \in \spce$. We will then obtain $\nlf_+$ by solving
\[
X = \nlf_+ ^* \begin{pmatrix}
    A_+ (\infty) & 0 \\ 0 & D_+ (\infty)
\end{pmatrix} \, ,
\]
where $A_+$ and $D_+$ are defined as in the convention in \eqref{eq:convention}.
% satisfies
% \begin{equation}\label{eq:fixed_pt_0}
% (\Id + \op) X = \Id  
% \end{equation}
% for some operator $\op$.
Our main goal for the remainder of the section is to show that $\nlf_+$ and $\nlf_- := \nlf \nlf_+^*$ are in $\mathbf{H}_\alpha^+$ and $\mathbf{H}_{\alpha, 0} ^-$.

In the special case when $\nlf \in \mathbf{B}_\alpha ^{\epsilon}$, we can take 
\begin{equation}\label{eq:M_defn_L2}
\op = \phl \begin{pmatrix}
   0 & A^{-1} B \\ -(A^{-1}B)^* & 0 
\end{pmatrix} \, ,
\end{equation}
where $\phl$ denotes the projection onto the space 
\[
\spce := \begin{pmatrix}
    H^2 (\D) & H^2 (\D) \\ H^2 (\D^*) & H^2 (\D^*)
\end{pmatrix} \, .
\]
From \eqref{eq:M_defn_L2} it is clear that $i\op$ is self-adjoint as an operator on $\spce$.

Things become more technical for general $\nlf \in \mathbf{B}_{\alpha}$. If $\nlf$ does not belong to $\mathbf{B}_\alpha ^{\epsilon}$ for any $\epsilon>0$, then we define $\op$ as an unbounded, densely defined operator, which will equal the operator in \eqref{eq:M_defn_L2} on a dense subspace $\dense$.

We point out that while the following subsection is involved, it becomes trivial once $\nlf \in \mathbf{B}_\alpha ^{\epsilon}$ for some $\epsilon > 0$. We also refer the reader to \cite[Section 4]{Alexis+25} for the simpler proof in the case $n=1$.

\subsection{Definition and properties of the operator \texorpdfstring{$\op$}{}}\label{subsection:defn_M}

Given $\nlf \in \mathbf{B}_{\alpha}$, we would like to define the operator $\op$ as in \eqref{eq:M_defn_L2}. 
However, if $A^{-1}B$ is unbounded, then the composition of multiplication by $A^{-1}B$ with the Fourier projection operator $\phl$ is not well-defined. Instead, we  rewrite the operator $\op$ as follows.

%Recall that the adjugate $\adj{T}$ of a matrix $T$ is the transpose of its cofactor matrix.
By the adjugate formula \eqref{eq:adjugate_2},
we may formally write the operator $\op$ in \eqref{eq:M_defn_L2} as
\begin{equation}\label{eq:formula_op_1}
\phl \begin{pmatrix}
    0 &  \frac{1}{\of^*} J^*  \\ - \frac{1}{\of} J  & 0
\end{pmatrix} \, ,
\end{equation}
where we fix 
\begin{equation}\label{def:J_o}
J := B^* (\adj A)^*  \, , \qquad  \of := \det A^* \, .
\end{equation}
We formally rewrite \eqref{eq:formula_op_1} as the limit
\[
\lim\limits_{\eta \to 0 ^+} \phl \begin{pmatrix}
    0 & \frac{1}{\of _{\eta} ^* } J^*  \\ - \frac{1}{\of_{\eta} } J & 0
\end{pmatrix} \, ,
\]
where, given a threshold $\eta >0$, we let $\of_{\eta}$ be the outer function on $\D$ whose absolute value on $\T$ satisfies
\begin{equation}\label{eq:omega_eta}
\log |\of_{\eta}| = \mathbf{1}_{\{ |\of| > \eta\}} \log |\of|  \, . 
\end{equation}

Define the set $\ddense$ consisting of elements $X \in \spce$ for which the limit
\begin{equation}\label{eq:unbded_op}
\op X := \lim\limits_{\eta \to 0} \phl \begin{pmatrix}
    0 & \frac{1}{\of _{\eta} ^* } J^*  \\ - \frac{1}{\of_{\eta}}  J  & 0
\end{pmatrix} X 
\end{equation}
exists in the weak sense on 
$L^2 := L^2 (\T ; \C^{4 n^2})$. This means that, for all $Z \in L^2 $, 
\[
\lim\limits_{\eta \to 0} \langle \phl \begin{pmatrix}
    0 & \frac{1}{\of _{\eta} ^* } J^*  \\ - \frac{1}{\of_{\eta}}J  & 0
\end{pmatrix} X , Z \rangle =  \langle \op X , Z \rangle \, .
\]
The map  $X \mapsto \op X$ in \eqref{eq:unbded_op} defines a linear operator $\op$ with domain $\mathcal{D} (\op) := \ddense$. We also define the space \[
\dense := \begin{pmatrix}
    \of & 0 \\ 0 & \of^*
\end{pmatrix} \spce \, .
\]

\begin{remark}\label{rmk:easy_case_spaces}
If $\nlf \in \mathbf{B} _{\alpha} ^{\epsilon}$, then $|\of|$ is bounded from below and hence
\[
\op = \phl \begin{pmatrix}
    0 & \frac{1}{\of ^*} J^*  \\ - \frac{1}{\of} J  & 0
\end{pmatrix} \, .
\]
In other words, the limit in $\eta$ disappears. In particular, $\dense = \ddense = \spce$, and so the operator $\op$ is defined and bounded on the whole space $\spce$. 
\end{remark}

\begin{lemma}\label{lem:density}

\begin{enumerate}[label=\arabic*)]
    \item \label{item:ratio_outer_conv_measure}  For every $\eta>0$, the function 
\begin{equation}\label{eq omega quotient}
    |{\of }{\of_{\eta} }^{-1}  -1|
\end{equation}
is bounded by $2$, and tends to zero in measure
as $\eta\to 0$.

\item \label{item:ratio_outer_conv_weak} For every $f\in L^2(\T)$, the function
\begin{equation}\label{eq omega quotient f}
    |{\of }{\of_{\eta} }^{-1}  -1|f
\end{equation}
converges to $0$ in $L^2(\T)$ as $\eta \to 0$.

\item \label{item:formula_D_op} For every $Y$ in $\spce$,
  \begin{equation}\label{eq strong Y}
\left \| \phl \begin{pmatrix}
    0 & \frac{\of ^*}{\of_{\eta} ^* } J^*  \\ - \frac{\of}{\of_{\eta} } J  & 0
\end{pmatrix} Y -\phl \begin{pmatrix}
    0 & J^* \\ -J & 0
\end{pmatrix} Y \right \|_{L^2}      
  \end{equation}
converges to $0$ as $\eta\to 0$. In particular, for every $Y \in \spce$, we have 
\begin{equation}\label{eq:strong_Y_formula}
\op \begin{pmatrix}
    \of & 0   \\ 0 & \of^*
\end{pmatrix} Y =\phl \begin{pmatrix}
    0 & J^* \\ -J & 0
\end{pmatrix} Y \, .
\end{equation}

\item \label{item:D_dense} Finally, 
   \[
   \dense \subset \ddense \subset \spce \, .
   \]
In particular, $\op$ is densely defined.

\end{enumerate}

\end{lemma}

\begin{proof}
\ref{item:ratio_outer_conv_measure} By construction, $|\of|\le |\of_\eta|$, so  \eqref{eq omega quotient} is
bounded by $2$.
By dominated convergence, $v_\eta:=\log |\of_\eta|$ converges to $ v:=\log |\of|$ in $L^1$. By
outerness, we have  $\log \of_\eta=v_\eta+iHv_\eta$ and $ \log \of= v+iHv$, where $H$ denotes  the Hilbert transform.
The weak-$(1,1)$ boundedness of the Hilbert transform then implies that 
$\log \of_\eta$ converges to $ \log \of$ in measure.
Taking exponentials and using continuity of the exponential function at $0$, it follows that  
\eqref{eq omega quotient} also tends to $0$ in measure, which was the second claim regarding \eqref{eq omega quotient}.

\ref{item:ratio_outer_conv_weak} As the first factor in \eqref{eq omega quotient f} is bounded, by an approximation argument
 it suffices to show that, for every  $N>0$,  
\begin{equation}
 \left |{\of }{\of_{\eta} }^{-1}  -1 \right | \min(|f|,N) 
\end{equation}
converges to zero in $L^2(\T)$.  As both factors are now bounded, this follows from
convergence  to zero in measure of \eqref{eq omega quotient}.

\ref{item:formula_D_op} As for convergence of \eqref{eq strong Y},
we estimate it by
\[
\left \| \begin{pmatrix}
    0 & (\frac{\of ^*}{\of_{\eta} ^*}  -1 ) J^*  \\ - (\frac{\of}{\of_{\eta}}-1)J  & 0
\end{pmatrix} Y \right \|_{L^2} 
 \leq \|J\|_{L^\infty} \left \| \left |\frac{\of }{\of_{\eta} }  -1 \right | Y \right \|_{L^2} \, ,
 %= \|J \|_{L^\infty} \left ( \int\limits_{\T} \left |\frac{\of}{\of_{\eta}}  -1 \right |^2 \|Y\|^2 _2 \right )^{\frac 1 2} \, 
\]
where we used that projections have operator norm $1$ on $L^2$.
Convergence to zero of the right side then follows from the previous item. Trivially,  \eqref{eq:strong_Y_formula} follows from \eqref{eq strong Y}.

\ref{item:D_dense} If $X\in \dense$, then by definition there exists $Y\in  \spce$ with 
  \begin{equation}\label{eq:X_to_Y_def}
   X = \begin{pmatrix}
       \of & 0 \\ 0 & \of ^*
   \end{pmatrix} Y  \, .
\end{equation}
Applying the strong convergence \eqref{eq strong Y} to $Y$ shows that $X$ is in the domain of $\op$. Hence $\dense \subset \ddense$.
 
As $\of$ is outer, Beurling's theorem on invariant subspaces \cite[Corollary 7.3, Chapter 2]{garnett} implies $\dense$ is dense in $\spce$ and hence $\op$ is densely defined.
\end{proof}

\begin{lemma}
    \label{lem:selfadjoint}
    The unbounded operator $i\op$ is self-adjoint.
\end{lemma}
\begin{proof}
The proof consist of upgrading the observation that, for each $\eta>0$, the bounded operator        \[
        A_\eta  := i\phl \begin{pmatrix} 
       0 & \frac{1}{\of_{\eta} ^*} J^*  \\ -  \frac{1}{\of _{\eta}} J & 0 
    \end{pmatrix} =\phl \begin{pmatrix} 
       0 & \frac{i}{\of_{\eta} ^*} J^*  \\ -  \frac{i}{\of _{\eta}} J & 0 
    \end{pmatrix} \phl
    \]
is evidently self-adjoint on the Hilbert space $\spce$.

We first observe that $i\op$ is symmetric, i.e., for all $x,y \in \mathcal{D} (\op) = \ddense$,
\begin{equation}\label{eq:symmetric}
\langle i \op x, y\rangle = \lim_{\eta \to 0} \langle A_\eta x, y\rangle = \lim_{\eta \to 0} \langle  x, A_\eta y\rangle = \langle x, i \op y \rangle\, .
\end{equation}
 In particular, this implies $\mathcal{D} (\op)  \subset\mathcal{D}(\op^*)$. We must now argue the two sets are equal.

To see that $D(\op^*)\subset D(\op)=\ddense$, fix $y\in D(\op^*)$. By definition, this means that there exists $\op^* y \in \spce$ such that, for all $x \in \ddense$,
\begin{equation}\label{eq ext bdly}
    \langle \op x, y\rangle = \langle x, \op ^* y\rangle \, .
\end{equation}
By definition of $\op$, it suffices to show  that, for every $x\in \spce$, 
\begin{equation}\label{eq x A eta y}
    \lim\limits_{\eta \to 0} \langle x, A_\eta y\rangle = \langle x, -i \op^* y \rangle \, .
\end{equation}
 So we write for 
\eqref{eq x A eta y}
\begin{equation}\label{eq d a star}
    \langle x, A_\eta y\rangle = \langle A_\eta x,  y\rangle =  \langle i \op M_\eta x,  y\rangle = \langle M_\eta x, -i \op ^* y \rangle \, ,
\end{equation}
where 
  \[
        M_\eta x = \begin{pmatrix}
        \frac{\of}{\of_{\eta} } & 0 \\ 0 & \frac{\of ^*}{\of_{\eta} ^*}
    \end{pmatrix} x \,,
    \]
and in the last step we applied \eqref{eq ext bdly} while noting that $M_{\eta} x\in \ddense$.  
Using that $M_\eta x$ has $L^2$-limit $x$ by Lemma \ref{lem:density}  and recalling \eqref{eq ext bdly}, we see that the right side of 
\eqref{eq d a star} converges  to the right side of \eqref{eq x A eta y} as $\eta\to 0$, which completes the proof that $D(\op^*)\subset D(\op)$. We conclude that $D(\op) =  D(\op^*)$ and that $i \op$ is self-adjoint.
\end{proof}

\begin{lemma}\label{lem:inverse_bd}
    If $\lambda \in \R$, then the densely defined operator 
    \begin{equation}\label{eq i l a}
     \Id + \lambda \op : \ddense \to \spce     
    \end{equation}
    is invertible, with operator norm bound
   \begin{equation}\label{eq:antisymmetric_inverse_bd}
    \left \| (\Id + \lambda \op)^{-1} \right \|_{\spce \to \spce} \leq 1 \, .
    \end{equation}
\end{lemma}
\begin{proof}
 Since $i\op$ is self-adjoint,  $\lambda \op$ has purely imaginary spectrum and $1$ belongs to its resolvent set, see \cite[Definition 13.26]{rudin}, which implies \eqref{eq i l a}.
As for the bound \eqref{eq:antisymmetric_inverse_bd}, we compute, for any $V \in \ddense$, that
\[
\left \| (\Id + \lambda \op)V\right \| ^2 = \|V\|^2 + \|\lambda \op V\|^2  \geq \|V\|^2 ,
\]
where mixed terms cancelled because $i\op$ is self-adjoint.
% \[
% \|V\|^2 + \|\op V\|^2 \geq \|V\|^2 \, .
% \]
% Altogether, we obtain
% \[
% \left \| (\Id + \op)V \right \| \geq \|V \| \, .
% \]
%Taking $V= (\Id + \op)^{-1}X$ for $X \in \spce$ yields 
%\[
%\|X\| \geq \left \| (\Id + \op)^{-1}X \right \| \, , 
%\]
%which 
This implies \eqref{eq:antisymmetric_inverse_bd}. 
\end{proof}

%\subsection{Solving the equation \texorpdfstring{$(\Id + \op)X = \Id$}{(Id + A)X = Id}}

\subsection{Uniqueness of the Factorization \texorpdfstring{\ref{eq:factorization_0}}{}}
We prove the uniqueness part of Theorem \ref{thm:NLFT_outer_inverse}, which is a consequence of the following lemma.  
\begin{lemma}\label{lem:RH_nec} 
    If $\nlf \in \mathbf{B} _{\alpha}$, then there exists at most one factorization 
    \eqref{eq:factorization_0}
    where $\nlf_{-} \in \mathbf{H}_{\alpha, 0} ^-$ and $\nlf_+ \in \mathbf{H}_{\alpha} ^+$.
\end{lemma}
\begin{proof} Assume a factorization exists. We show it is unique by using outerness  of the diagonal blocks of $\nlf$ and  holomorphicity to show that $M_+$ must satisfy an equation involving the invertible operator $\Id + \op$, which then has at most one solution. 

We  turn to the details.
First note that $\nlf_{-}$ is uniquely determined by the unitary $\nlf_{+}$ from  \eqref{eq:factorization_0}. We now show $\nlf_+$ is uniquely determined. 
% Let us write
% \begin{equation}\label{eq:label_M_all}
%  \begin{pmatrix}
%     A_{-} & B_{-} \\ C_{-} & D_{- }
% \end{pmatrix}:= \nlf_{-}   \, , \qquad \begin{pmatrix}
%     A_{+} & B_{+} \\ C_{+} & D_{+}
% \end{pmatrix}  := \nlf_{+}  \, , \qquad \begin{pmatrix}
%     A & B \\ C & D
% \end{pmatrix} :=\nlf \, .
% \end{equation}
Following our usual conventions, we rewrite \eqref{eq:factorization_0} as 
\[
\begin{pmatrix}
A_{-} & B_{-} \\ C_{-} & D_{-}   
\end{pmatrix} = \begin{pmatrix}
A & B \\ C & D    
\end{pmatrix} \begin{pmatrix}
A_{+} ^* & C_{+} ^* \\ B_{+} ^* & D_{+} ^*   
\end{pmatrix} \, .  
\]
Multiply both sides on the left by the block diagonal matrix 
\[
\begin{pmatrix}
A^{-1} & 0 \\ 0 & D^{-1}    
\end{pmatrix}
\]
to get
\begin{equation}\label{eq:RH_pre_proj}
\begin{pmatrix}
A^{-1} A_{-} & A^{-1} B_{-} \\ D^{-1} C_{-} & D^{-1} D_{-}   
\end{pmatrix} = \begin{pmatrix}
\Id & A^{-1} B \\ D^{-1} C & \Id    
\end{pmatrix} \begin{pmatrix}
A_{+} ^* & C_{+} ^* \\ B_{+} ^* & D_{+} ^*   
\end{pmatrix} \, .
\end{equation}
Because $\nlf$
is unitary, then 
\begin{equation}\label{eq:AB_to_CD}
(D^{-1} C)^* = - A^{-1} B\,.
\end{equation}
Combining this with the adjugate formula \eqref{eq:adjugate_2} and  taking $J, \of$ as in \eqref{def:J_o},   %\begin{equation}\label{eq:data_Jo_uniqueness}
    %J :=  B^* (\adj A)^* \, , \qquad \of := \det A^* \, .
%\end{equation}
% and so \eqref{eq:RH_pre_proj} becomes \lb{there is potential to save some displays here}
% \[
% \begin{pmatrix}
% A^{-1} A_{-} & A^{-1} B_{-} \\ D^{-1} C_{-} & D^{-1} D_{-}   
% \end{pmatrix} = \begin{pmatrix}
% \Id & A^{-1} B \\ -(A^{-1}B)^* & \Id    
% \end{pmatrix} \begin{pmatrix}
% A_{+} ^* & C_{+} ^* \\ B_{+} ^* & D_{+} ^*   
% \end{pmatrix} \, .
% \]
% Using the adjugate formula \eqref{eq:adjugate_2}, 
we may thus rewrite \eqref{eq:RH_pre_proj} as
\[
\begin{pmatrix}
A^{-1} A_{-} & A^{-1} B_{-} \\ D^{-1} C_{-} & D^{-1} D_{-}   
\end{pmatrix} = \begin{pmatrix}
\Id & \frac{1}{\of ^*} J^*  \\ - \frac{1}{\of}J  & \Id    
\end{pmatrix} \begin{pmatrix}
A_{+} ^* & C_{+} ^* \\ B_{+} ^* & D_{+} ^*   
\end{pmatrix} \, .
\]
% where   \begin{equation}\label{eq:data_Jo_uniqueness}
%     J :=  B^* (\adj A)^* \, , \qquad \of := \det A^* \, .
%     \end{equation}
   Taking $\dense, \ddense$ and $\op$ as in Section \ref{subsection:defn_M},  for all $W \in \dense$
we have that
\[
\langle \begin{pmatrix}
A^{-1} A_{-} & A^{-1} B_{-} \\ D^{-1} C_{-} & D^{-1} D_{-}   
\end{pmatrix}, W \rangle = \langle \begin{pmatrix}
\Id & \frac{1}{\of ^*} J^*  \\ - \frac{1}{\of}J  & \Id    
\end{pmatrix} \begin{pmatrix}
A_{+} ^* & C_{+} ^* \\ B_{+} ^* & D_{+} ^*   
\end{pmatrix} , W \rangle \, ,
\]
and the integral within the inner products from \eqref{eq:inner_prod_def} is well-defined because $W \in \dense$.  By duality and then applying \eqref{eq strong Y} for $\op$ on elements of $\dense$, we may write this last inner product as
% \[
% \langle  \begin{pmatrix}
% A_{+} ^* & C_{+} ^* \\ B_{+} ^* & D_{+} ^*   
% \end{pmatrix}, \begin{pmatrix}
% \Id & - \frac{1}{\of ^*} J^* \\ \frac{1}{\of} J  & \Id    
% \end{pmatrix} Z \rangle = \langle  \begin{pmatrix}
% A_{+} ^* & C_{+} ^* \\ B_{+} ^* & D_{+} ^*   
% \end{pmatrix}, (\Id -\op) Z \rangle \, ,
% \]
\begin{equation}\label{eq:inner_prods}
\langle \begin{pmatrix}
A^{-1} A_{-} & A^{-1} B_{-} \\ D^{-1} C_{-} & D^{-1} D_{-}   
\end{pmatrix}, W \rangle=\langle  \begin{pmatrix}
A_{+} ^* & C_{+} ^* \\ B_{+} ^* & D_{+} ^*   
\end{pmatrix}, (\Id -\op) W \rangle = \langle  \nlf_+ ^*, (\Id -\op) W \rangle \, .
\end{equation}
% and so we obtain \lb{here there is also some room to save space. This equation is just repeated twice}
% \begin{equation}\label{eq:pre_MVT_nec_conds}
% \langle \begin{pmatrix}
% A^{-1} A_{-} & A^{-1} B_{-} \\ D^{-1} C_{-} & D^{-1} D_{-}   
% \end{pmatrix}, Z \rangle = \langle  \nlf_+ ^*, (\Id -\op) Z \rangle \, .
% \end{equation}
Observe that in the leftmost matrix of \eqref{eq:inner_prods}, the first row is antiholomorphic, and the upper right entry vanishes at $\infty$. Similarly, the second row is holomorphic, and the lower left entry vanishes at $0$. It follows from the definition of the inner product $\langle\cdot,\cdot \rangle$ that the projection of such functions onto $\spce$ equals the projection onto the space of constant diagonal block matrices. Since $W \in \spce$, we can apply this reasoning to the leftmost inner product of \eqref{eq:inner_prods} and, combined with \eqref{eq:mult_mean_value}, we obtain
\[
\langle \begin{pmatrix}
    A_+ (\infty)^{-1} & 0 \\ 0 & D_+ (0)^{-1} 
\end{pmatrix}, W \rangle = \langle  \nlf_+ ^* , (\Id -\op) W \rangle \, .
\]
%Using the block labeling
% \[
% \begin{pmatrix}
%     Z_{11} & Z_{1 2} \\ Z_{2 1} & Z_{21}
% \end{pmatrix} := Z \, .
% \]
% we may write the inner product on the left of \eqref{eq:pre_MVT_nec_conds} \lb{The inner product is a number, the integral below is a matrix. The number is the trace of that matrix. I think it is still correct if we add traces in the next three equations}  as
% \[
% \int\limits_{\T} (A^{-1} A_{-} Z_{11} ^* + A^{-1} B_{-} Z_{1 2} ^*) + (D^{-1} C_{-}  Z_{2 1}^* + D^{-1} D_{-}  Z_{2 2}^* ) \, ,
% \]
% where the term in the first parentheses belongs to $H^2 (\D^*)$ and the term in the second parentheses belongs to $H^2 (\D)$. By the mean value property and the fact that $B_{-}$ and $C_{-}$ have mean $0$, the above equals
% \begin{equation} \label{eq:post_MVT_nec_conds}
% A^{-1} (\infty) A_{-} (\infty) Z_{11} ^* (\infty) +  D^{-1} (0) D_{-}(0)  Z_{2 2}^* (0) 
% \end{equation}
% \[= A_{+} (\infty)^{-1}  Z_{11} ^* (\infty) +  D_{+}(0) ^{-1}  Z_{2 2}^* (0)
% =\langle\begin{pmatrix}
%     A_+ (\infty)^{-1} & 0 \\ 0 & D_+ (0)^{-1} 
% \end{pmatrix}, Z \rangle \, , 
% \]
% where we used the identities
% \[
% A(\infty) = A_{-} (\infty) A_+ (\infty) \, , \qquad D(0) = D_{-} (0) D_+ (0) \, . 
% \]
% Thus altogether from \eqref{eq:pre_MVT_nec_conds}--\eqref{eq:post_MVT_nec_conds}, we have
% \[
% \langle \begin{pmatrix}
%     A_+ (\infty)^{-1} & 0 \\ 0 & D_+ (0)^{-1} 
% \end{pmatrix}, Z \rangle = \langle  \nlf_+ ^* , (\Id -\op) Z \rangle \, . 
% \]
Replacing $W$ by $W \begin{pmatrix}
    A_+ (\infty) ^*  & 0 \\ 0 & D_+ (0) ^* 
\end{pmatrix}$ and applying duality as in \eqref{eq:duality} yields 
\begin{equation}
\label{e:ZD}
\langle \begin{pmatrix}
    \Id & 0 \\ 0 & \Id
\end{pmatrix}, W \rangle = \langle  X, (\Id -\op) W \rangle \, ,
\end{equation}
for all $W \in \dense$, where we define
\begin{equation}\label{eq:X_Y+_unique}
X := \nlf_+ ^* \begin{pmatrix}
    A_+ (\infty) & 0 \\ 0 & D_+ (\infty)  
\end{pmatrix} .
\end{equation}
Equation \eqref{e:ZD} in fact continues to hold for all $W' \in \ddense$. To see this, fix such $W'$, let $\epsilon > 0$ and set 
\[
W_\epsilon :=\begin{pmatrix}
    \frac{\of}{\of_{\epsilon}} & 0 \\ 0 & \frac{\of ^*}{\of_{\epsilon} ^*}  
\end{pmatrix}  W' \in \dense \, .
\]
Apply \eqref{e:ZD} to $W_\epsilon$ and take $\epsilon \to 0$, using dominated convergence and the fact that $W' \in \ddense$ to obtain \eqref{e:ZD} for $W'$.
% \[
% \langle \begin{pmatrix}
%     \Id & 0 \\ 0 & \Id
% \end{pmatrix}, \begin{pmatrix}
%     \frac{\of}{\of_{\epsilon}} & 0 \\ 0 & \frac{\of ^*}{\of_{\epsilon} ^*}  
% \end{pmatrix}  Z' \rangle = \langle  X, \begin{pmatrix}
%     \frac{\of}{\of_{\epsilon}} & - \frac{1}{\of_{\epsilon} ^*}J^*  \\  \frac{1}{\of_{\epsilon}} J & \frac{\of ^*}{\of_{\epsilon} ^*}  
% \end{pmatrix}  Z'  \rangle \, .
% \]
% Letting $\epsilon \to 0^+$ and using the fact that $Z' \in \ddense$ yields
% \[
% \langle \begin{pmatrix}
%     \Id & 0 \\ 0 & \Id
% \end{pmatrix},  Z' \rangle = \langle  X, (\Id - \op)  Z'  \rangle \, .
% \]
By Lemma \ref{lem:inverse_bd}, as $W'$ ranges across $\ddense$, then $(\Id - \op) W'$ ranges across $\spce$. Thus for all $V \in \spce$, we have
\[
\langle (\Id + \op)^{-1}\begin{pmatrix}
    \Id & 0 \\ 0 & \Id
\end{pmatrix},   V \rangle = \langle \begin{pmatrix}
    \Id & 0 \\ 0 & \Id
\end{pmatrix},  (\Id - \op)^{-1} V \rangle = \langle  X, V  \rangle \, ,
\]
where we invoked Lemma \ref{lem:inverse_bd} to invert $\Id + \op$.
% or rather invoking $\op^* = -\op$ we get
% \[
% \langle (\Id + \op)^{-1}\begin{pmatrix}
%     \Id & 0 \\ 0 & \Id
% \end{pmatrix},   V \rangle = \langle  X, V  \rangle \, .
% \]
Thus
\[
X = (\Id + \op)^{-1} \begin{pmatrix}
    \Id & 0 \\ 0 & \Id
\end{pmatrix} 
\]
 is uniquely determined. Because $\nlf_+$ is unitary, then by \eqref{eq:X_Y+_unique} we have
\[
X^* X = \begin{pmatrix}
    A_+ (\infty) A_+ (\infty) ^* & 0 \\ 0 &  D_+ (0) D_+ (0) ^*
\end{pmatrix}
\]
is a uniquely determined positive definite matrix. By uniqueness of the Cholesky factorization, both $A_+ (\infty) \in \grp_{\alpha(0)}$ and $D_+ (0) \in \grp_{\alpha(1)}$ are unique. From \eqref{eq:X_Y+_unique} it follows that $\nlf_+$ is uniquely determined.
\end{proof}

\subsection{Construction of the Factorization \texorpdfstring{\ref{eq:factorization_0}}{}}

Motivated by the uniqueness proof, we construct a unitary-valued candidate for $\nlf_+ \in \mathbf{L}_+$, where $\mathbf{L}_+$ denotes the mvfs $\nlf$ on $\T$ for which $M^* \in \spce$. In what follows, we continue using the convention \eqref{eq:convention} and denote the diagonal blocks of $\nlf_+$ by $A_+$ and $D_+$.

\begin{lemma}\label{lem:solve_fixed_pt_eqn}
    Let $\alpha: \Z_2 \to \Z_2$ and let $\nlf \in \mathbf{B}_{\alpha}$. If 
    \begin{equation}\label{eq:RH_Hilbert_eqn}
 X := (\Id + \op )^{-1} \Id \, ,
    \end{equation}
then there exists a unique $U(2n)$-valued function $\nlf_+ \in \mathbf{L}_+$ satisfying $A_+(\infty) \in \grp_{\alpha(0)}$ and $D_+(0) \in \grp_{\alpha(1)}$, and
\begin{equation}\label{eq:X_to_Y_+}
X = \nlf_+ ^* \begin{pmatrix}
    A_+ (\infty) & 0 \\ 0 & D_+ (0)
\end{pmatrix} \, .
\end{equation}

\end{lemma}
\begin{proof}
\underline{\textbf{Step 1:} $X^* X$ is block diagonal and constant.}
Because 
\[
\Id + \op: \ddense \to \spce
\]
is invertible by Lemma \ref{lem:inverse_bd}, then $X \in \ddense \subset \spce$. We claim that
\begin{equation}\label{eq:X_star_X_formula}
X^* X = \begin{pmatrix}
    Q(0) & 0 \\ 0 & T(\infty)
\end{pmatrix} \, ,
\end{equation}
where $Q, T$ are as in the following labeling 
\begin{equation}\label{eq:comps_X}
\begin{pmatrix}
    Q & R \\ S  & T 
\end{pmatrix} := X \, .
\end{equation}
 Indeed, from this labeling and \eqref{eq:RH_Hilbert_eqn}, or equivalently
 \begin{equation}\label{eq:fixed_pt_eqn}
X = \Id -\op X \, ,
\end{equation}
it follows that 
\begin{equation}\label{eq:X_comps_forms}
X = \begin{pmatrix}
    Q & R\\ S & T
\end{pmatrix}= \begin{pmatrix}
    \Id - \lim\limits_{\eta \to 0} P_{\D} \frac{1}{\of_{\eta} ^ *} J ^*   S & -  \lim\limits_{\eta \to 0}P_{\D} \frac{1}{\of _{\eta} ^*} J ^*  T \\ \lim\limits_{\eta \to 0} P_{\D^*}    \frac{1}{\of_{\eta}} J  Q & \Id + \lim\limits_{\eta \to 0} P_{\D^*}   \frac{1}{\of_{\eta}} J R
\end{pmatrix} \, ,
\end{equation}
where the limit is in the weak sense. 
Taking conjugate transposes yields
\begin{equation} \label{eq:X_star_comps_forms}
X^* = \begin{pmatrix}
    Q^* & S^* \\ R^* & T^*
\end{pmatrix} = \begin{pmatrix}
    \Id - \lim\limits_{\eta \to 0} P_{\D^*} \frac{1}{\of_{\eta} } S^* J & \lim\limits_{\eta \to 0} P_{\D} \frac{1}{\of_{\eta} ^*} Q^* J ^*  \\ -  \lim\limits_{\eta \to 0} P_{\D^*} \frac{1}{\of_{\eta}} T^* J  & \Id + \lim\limits_{\eta \to 0} P_{\D} \frac{1}{\of_{\eta} ^*} R^* J ^*
\end{pmatrix} \, .
\end{equation}

We now verify one by one that the blocks of $X^*X$ are given by \eqref{eq:X_star_X_formula}, and begin with the diagonal blocks. For the upper left block $Q^* Q + S^* S$ of $X^* X$, substituting $Q$ and $S^*$ via \eqref{eq:X_comps_forms}--\eqref{eq:X_star_comps_forms}, we get
\begin{equation}\label{eq:QQSS_lim}
 Q^* Q + S^* S = \lim\limits_{\eta \to 0} Q^* (\Id - P_{\D} \frac{1}{\of _{\eta} ^*} J ^*  S ) + (P_{\D}  \frac{1}{\of _{\eta} ^*} Q^* J^* ) S\, , 
\end{equation}
where the limit is in the weak sense in $L^1(\T)$. We add and subtract $ \frac{1}{\of_{\eta} ^*}Q^* J^*  S$ to see that the right side of  \eqref{eq:QQSS_lim} equals 
\begin{equation}\label{eq:GG*+JJ*_useful}
 \lim\limits_{\eta \to 0} Q^* (\Id + (\Id -  P_{\D}) \frac{1}{\of_{\eta} ^*} J^*  S ) - ((\Id - P_{\D}) \frac{1}{\of_{\eta} ^*} Q^* J^*    ) S \, .
\end{equation}
Because the image of the operator $\Id - P_{\D}$ equals the subspace $H^2 _0 (\D^*)$ of mean zero functions in $H^2 (\D^*)$, then \eqref{eq:GG*+JJ*_useful} is the weak limit of a sum of products of $H^2 (\D^* )$ functions and so is the weak limit of a sequence in $H^1 (\D^* )$. Thus 
\[
Q^* Q + S^* S \in H^1 (\D^*) \, .
\]
Applying the $*$-operation and noting that this matrix is pointwise hermitian, we get it also belongs to $H^1 (\D)$. Thus it must be constant. Using that every function in the image of  $\Id-P_{\D}$ has mean zero, then  \eqref{eq:GG*+JJ*_useful} evaluated at $z = \infty$ equals
$Q^* (\infty)$.
 Thus the upper left entry of $X^*  X$ is given by $Q^*( \infty)$. Similarly, the lower-right block $R^* R  + T^* T$ of $X^* X$ equals $T(\infty)$.

We turn to the off-diagonal blocks. By substituting $Q$ and $T^*$ by their limits in \eqref{eq:X_comps_forms}--\eqref{eq:X_star_comps_forms}, we may write the bottom left block of $X^* X$ as
\[
R^*Q + T^* S
=\lim\limits_{\eta \to 0} R^* (\Id - P_{\D})  \frac{1}{\of_{\eta} ^*}J^* S  +  ( (P_{\D} - \Id) \frac{1}{\of_{\eta} ^*} R^*  J^*  )S   + R^* + S\, .
\]
The right side is clearly an element of $H^1 (\D^*)$. On the other hand, substituting $R^*$ and $S$ by their limits in \eqref{eq:X_comps_forms}--\eqref{eq:X_star_comps_forms} reveals that
\[
R^*Q + T^* S 
= \lim\limits_{\eta \to 0} ((\Id - P_{\D^*}) \frac{1}{\of_{\eta} } T^* J  ) Q - T^*  (\Id - P_{\D^*})  \frac{1}{\of_{\eta} } J Q  
\]
 is an element of the subspace $H^1 _0 (\D)$ of mean zero functions in $H^1 (\D)$. Since $R^* Q + T^* S$ belongs to both $H^1 (\D^*)$ and $H^1 _0 (\D)$, it must be identically zero. Similarly, the upper right block is also zero. 
 This proves Claim \eqref{eq:X_star_X_formula} and completes Step 1 of the proof.

\underline{\textbf{Step 2:}  $X^* X$ is positive definite  for all $z \in \T$.}  
%By \eqref{eq:X_star_X_formula}, $X^* X$ is positive semidefinite. \lb{We dont need (\eqref{eq:X_star_X_formula}) to see that every matrix $W^*W$ is positive semidefinite, right?} We now show it is positive definite. 
Note that if $v$ is in the kernel of $W^*W$ for any matrix $W$, then $v$ is also already in the kernel of $W$. Suppose that  $v$ belongs to the kernel of the constant matrix $X^*X$ (recall \eqref{eq:X_star_X_formula}). Taking $W = X(z)$ for any $z \in \mathbb{T}$, it follows that 
\[
    X(z) v = 0
\]
for all $z \in \mathbb{T}$.
But multiplying both sides of \eqref{eq:fixed_pt_eqn} on the right by $v$ yields
\[
0  = v - (\op X) v \, ,
\]
or rather using that $v$ is constant in $z$ and slightly abusing notation\footnote{ %\lb{I know that we talked about this abuse of notation in Oberwolfach. I still dont understand what the proper way to do this is. Christoph, do you know a more elegant way?}
 This abuse of notation can be made rigorous as follows. We have that $(\op X)v$ vanishes if and only if $(\op X) (v, \ldots, v)$ vanishes, where $(v,\ldots, v)$ denotes the $2n \times 2n$ matrix with each column given by $v$, and then using that $Xv = 0$, we have 
\[
(\op X) (v, \ldots, v) = \op (X (v, \ldots, v)) = \op (0) = 0 \, .
\]}, 
\begin{equation}\label{eq:v_vanishes_1}
v = (\op X) v  = \op (Xv) = 0 \, .
\end{equation}
Therefore $v=0$ by \eqref{eq:v_vanishes_1}.
Hence $X^* X$ is positive definite.

\underline{\textbf{Step 3:} defining $\nlf_+$.} By Step 2, the right side of \eqref{eq:X_star_X_formula} is positive definite.
Let $\beta_0, \beta_1 \in \{0, 1\}$ be numbers that will be chosen later. By the Cholesky factorization of Lemma \ref{lem_cholesky}, there exist a unique pair $(U,V) \in \grp_{\beta_0} \times \grp_{\beta_1}$ for which
\begin{equation}\label{eq:def_UV}
\begin{pmatrix}
    Q(0) & 0 \\
    0 & T (\infty)
\end{pmatrix} = \begin{pmatrix}
    U^* & 0 \\
    0 & V^*
\end{pmatrix} \begin{pmatrix}
    U & 0 \\
    0 & V
\end{pmatrix} \, .
\end{equation}
For this choice of $(U,V)$, we then define
\begin{equation}\label{eq:defn_Y+}
\nlf_+ ^* := X \begin{pmatrix}
    U^{-1} & 0 \\ 0 & V ^{- 1} 
\end{pmatrix} \, .
\end{equation}
Thus $\nlf_+ \in \mathbf{L} _+$ and is a.e.\ unitary by \eqref{eq:X_star_X_formula}.
Furthermore,  
\begin{equation}\label{eq:UV_to_AD}
 \begin{pmatrix}
A_+ ^* (0) & 0 \\ 0 & D_{+} ^* (\infty)    
\end{pmatrix} = \begin{pmatrix}
    U^* & 0 \\ 0 & V^*
\end{pmatrix}\ ,
\end{equation}
and so \eqref{eq:X_to_Y_+} holds. Finally, since $U^* \in \grp_{1-\beta_0}$ and $V^* \in \grp_{1-\beta_1}$, the proof is complete once we set $\beta_0 = 1-\alpha(0)$ and $\beta_1 = 1-\alpha(1)$, for then $A_+ (\infty) \in \grp_{\alpha(0)}$ and $D(0) \in  \grp_{\alpha(1)}$ by \eqref{eq:UV_to_AD}. 
\end{proof}

The following now shows existence of a factorization as in \eqref{eq:factorization_0}.
\begin{lemma}\label{lem:RH_ex} 
    If $\nlf \in \mathbf{B} _{\alpha}$, then there exists a factorization \eqref{eq:factorization_0}
    where $\nlf_{-} \in \mathbf{H}_{\alpha, 0} ^-$ and $\nlf_+ \in \mathbf{H}_{\alpha} ^+$.
\end{lemma}

%\lb{Does the following idea simplify the proof: The factorization $M = M_- M_+$ is equivalent to the factorization $\ad(z^{-1})(M^*) = \ad(z^{-1}) (M_+)^* \ad(z^{-1}) (M_-)^*$. Can this be used to directly get the analyticity properties for the left factor? }
%\ma{Maybe? If I'm understanding your thought correctly, one would have to show then that $\ad(z^{-1}) (M_{-}) ^*$ then has to necessarily satisfy the fixed point equation, following the uniqueness proof. In the uniqueness proof, everything till \eqref{eq:inner_prods} holds. But then in the next display, we need $A(\infty) = A_{-} (\infty) A_+ (\infty)$, which feels like it uses the holomorphicity property of $B_{-}$. So maybe just prove the holomorphicity property of $B_{-}$ and then run the scheme of the uniqueness proof? It certainly has potential, I'd have to write it up to see.}
%\ma{follow-up: it seems even \eqref{eq:inner_prods} doesn't even hold without knowing holomorphicity, since to write down the operator $\Id - \op$, you need to have a projection $\phl$ on the outside, which you can only insert if you know the holomorphicity properties. So I don't think the idea leads to a simplification :(}

\begin{proof}
We take advantage of the  functional analysis done in this section to define $\nlf_+$ and $\nlf_{-}$. The labor here is checking that $\nlf_{-}$ and $ \nlf_+$ belong to the mandated spaces by verifying holomorphicity properties, normalizations and nonexistence of inner factorizations of type \eqref{eq:inner_factor}.

We turn to the details. Let $J$ and $\of$ be as in \eqref{def:J_o}, and let $\nlf_+ \in \mathbf{L}_+$ be the function from Lemma \ref{lem:solve_fixed_pt_eqn}, so that $A_+ (\infty) \in \grp_{\alpha(0)} $ and $D_+(0) \in \grp_{\alpha(1)}$. 

Define the $U(2n)$-valued function
\[
\nlf_{-} := \nlf \nlf_+ ^{-1} = \nlf \nlf_+ ^* \, .
\]
Rewriting this using our labeling conventions from \eqref{eq:convention} yields
\[
 \begin{pmatrix}
    A_- & B_- \\ C_- & D_-
\end{pmatrix}= \begin{pmatrix}
    A & B \\ C & D
\end{pmatrix} \begin{pmatrix}
    A_+ ^* & C_+ ^* \\ B_+ ^* & D_+ ^*
\end{pmatrix}= \begin{pmatrix}
    A A_+ ^* + B B_+ ^* & A C_+ ^* + B D_+ ^* \\ C A_+ ^* + D B_+ ^* & C C_+ ^* + D D_+ ^*
\end{pmatrix} \, . 
\]
Let us check that $\nlf_{-}$ has the correct holomorphicity conditions.
% Start by writing
% \[
%  \begin{pmatrix}
%     A_- & B_- \\ C_- & D_-
% \end{pmatrix}= \begin{pmatrix}
%     A & B \\ C & D
% \end{pmatrix} \begin{pmatrix}
%     A_+ ^* & C_+ ^* \\ B_+ ^* & D_+ ^*
% \end{pmatrix}= \begin{pmatrix}
%     A A_+ ^* + B B_+ ^* & A C_+ ^* + B D_+ ^* \\ C A_+ ^* + D B_+ ^* & C C_+ ^* + D D_+ ^*
% \end{pmatrix} \, . 
% \]
Using \eqref{eq:X_to_Y_+} and \eqref{eq:X_comps_forms}, we see that
\[
\begin{pmatrix}
    A_+ ^* & C_+ ^* \\ B_+ ^* & D_+ ^*
\end{pmatrix} = \begin{pmatrix}
    A_+ (\infty)^{-1} - \lim\limits_{\eta \to 0} P_{\D} \frac{1}{\of_{\eta} *} J^*   B_+ ^* & -  \lim\limits_{\eta \to 0}P_{\D}  \frac{1}{\of _{\eta} ^*} J^* D_+ ^* \\ \lim\limits_{\eta \to 0} P_{\D^*} \frac{1}{\of_{\eta}} J   A_+ ^* & D_+ (0)^{-1} + \lim\limits_{\eta \to 0} P_{\D^*}  \frac{1}{\of_{\eta}} J  C_+ ^*
\end{pmatrix} \, .
\]
To see that $A_- \in H^2 (\D^*)$, we write
\[
A_- = A A_+ ^* + B B_+ ^*
=\lim\limits_{\eta \to 0} A ( A_+ (\infty) ^{-1} - P_{\D} \frac{1}{\of ^* _{\eta}} J ^*  B_+ ^*) +  B  B_+ ^* 
\]
\[= A A_+ (\infty)^{-1} + \lim\limits_{\eta \to 0} A(\Id - P_{\D})  \frac{1}{\of ^* _{\eta}} J^* B_+ ^* - \frac{1}{\of_{\eta} ^*} A   J^* B_+ ^* +  B B_+ ^* \, .
 \]
 By \eqref{def:J_o} and Lemma \ref{lem:density}, we have
 \begin{equation}\label{eq:weak_limit_AJ}
   \frac{1}{\of_{\eta} ^*} A  J^* =  \frac{\of ^*}{\of_{\eta} ^*} B  \to B \, ,
 \end{equation}
 strongly in  $L^2$ as $\eta \to 0$. Thus 
 \[
 A_{-} = A A_{+} (\infty)^{-1} + \lim\limits_{\eta \to 0} A(\Id - P_{\D}) \frac{1}{\of ^* _{\eta}} J^*  B_+ ^*
 \]
 is a weak limit of $H^2 (\D^*)$ functions and so is in $H^2 (\D^*)$. Evaluating both sides at $\infty$ yields
 \[
A_{-}( \infty) =A (\infty) A_+ (\infty)^{-1} \, . 
 \]
Each term on the right side belongs to $\grp_{\alpha(0)}$ and hence $A_{-} (\infty) \in \grp_{\alpha(0)}$.
 
 Similarly, 
 \[
 D_{-} = C C_+ ^* + D D_+ ^* = C C_+ ^* + D (D_+ (0)^{-1} + \lim\limits_{\eta \to 0} P_{\D^*} \frac{1}{\of_{\eta} } J  C_+ ^* ) 
 \]
 \[
 =D D_+ (0)^{-1} + \lim\limits_{\eta \to 0} D (\Id -  P_{\D^*})   \frac{1}{\of_{\eta}} J C_+ ^* + \lim\limits_{\eta \to 0}   \frac{1}{\of_{\eta}} D J C_+ ^*  + C C_+ ^*  \, . 
 \]
From \eqref{def:J_o} and \eqref{eq:AB_to_CD}, we have  
\begin{equation}\label{eq:strong_limit_DJ}
\frac{1}{\of_{\eta}} D J = \frac{\of }{\of_{\eta}} D \left ( \frac{1}{\of} J \right )   = - \frac{\of }{\of_{\eta}} D(A^{-1}B)^{*}   = -  \frac{\of }{\of_{\eta}} C  \to -C  \, ,
\end{equation}
strongly in  $L^2$  as $\eta \to 0$, 
and hence
\[
D_{-} =D D_+(0)^{-1} + \lim\limits_{\eta \to 0} D (\Id -  P_{\D^*}) J^*  \frac{1}{\of_{\eta}} C_+ ^* \, .
\]
Therefore $D_{-}$ is a weak limit of $H^2 (\D)$ functions and hence belongs to $H^2 (\D)$. Evaluating both sides at $0$ yields
\[
D_{-} (0) = D(0) D_+ (0)^{-1} \, .
\]
The right side is a product of terms in $\grp_{\alpha(1)}$ and so $D_{-} (0) \in \grp_{\alpha(1)}$. 

As for $B_{-}$, we use the strong limit \eqref{eq:weak_limit_AJ} to write 
\[
B_{-} = A C_+ ^* + B D_+ ^* = -\lim\limits_{\eta \to 0} A P_{\D} \frac{1}{\of ^* _{\eta}} J  D_+ ^* + B D_+ ^*  = \lim\limits_{\eta \to 0} A( \Id - P_{\D}) \frac{1}{\of ^* _{\eta}}  J  D_+ ^* \, , 
\]
which is a weak limit of $H^2 _0 (\D^*)$ functions  and so belongs to $H^2 _0 (\D^*) $. Similarly, using the strong limit \eqref{eq:strong_limit_DJ}, we have
\[
C_{-} = C A_+ ^* + D B_+ ^* = C A_+ ^* - \lim\limits_{\eta \to 0} D P_{\D^*} J^* \frac{1}{\of_{\eta} } A_+ ^*=   \lim\limits_{\eta \to 0} D(\Id -  P_{\D^*}) J^* \frac{1}{\of_{\eta} } A_+ ^* \, , 
\]
which is a weak limit of $H^2 _0 (\D)$ functions and so is again in $H^2 _0 (\D)$.

Thus Properties \ref{item:1st_prop_H}--\ref{item:2nd_prop_H} of Definition \ref{defH} hold for $\nlf_+$, and their analogs for $\nlf_{-}$. We now turn to Property \ref{item:3rd_prop_H} of Definition \ref{defH} and its analog for $\nlf_{-}$, and show there are no nontrivial factorizations \eqref{eq:inner_factor} and \eqref{eq:inner_factor_H-}. Let a factorization like \eqref{eq:inner_factor} be given, so that for an inner factor $I_1$ 
\[
     A_+ ^* \, ,  C_+ ^* \in I_1 H^2 (\D) \, .
\]
Then, from examining the upper left block of the matrix equation 
\[
\nlf = \nlf_{-} \nlf_{+}\,,
\]
we obtain
\[
    A^* = A_+ ^* A_{-} ^*  + C_+ ^* B_{-} ^*  \in I_1 H^\infty(\D) \, .
\]
% Multiplying both sides on the left by $I_1 ^{-1}$ we get that
% \[
%  I_1 ^{-1} A^* = (I_1 ^{-1}  A_+ ^*) A_{-} ^*  + (I_1 ^{-1} C_+ ^*) B_{-} ^*
% \]
% is a sum of products of $H^2 (\D)$ functions and hence is an element of $H^1 (\D )$. 
We claim that $\lvert \det I_1 \rvert = 1$ on all of $\D$: if not, then because $I_1 ^{-1} A^* \in H^{\infty} (\D)$, we estimate 
\begin{equation}\label{outer_func_mean_value_fails}
\log \lvert \det A^* (0) \rvert < \log \lvert \det I_1 ^{-1}(0) A^* (0) \rvert
\leq  \int\limits_{\T}  \log \lvert \det I_1 ^{-1} A^* \rvert = \int\limits_{\T}  \log \lvert \det A^* \rvert \, ,
\end{equation}
which contradicts the fact that $\det A ^*$ is outer on $\D$. Because $\lvert \det I_1 \rvert =1$ on $\D$, then by Lemma \ref{lem:cst_det_matrix_inner}  it follows that $I_1$ is constant on $\D$. But, for $A(\infty)$ and $ A (\infty) I_1 ^* (\infty)$ to both be elements of $\grp_{\alpha(0)}$, it must hold that $I_1= \Id$. Similar reasoning with $B_+$ and $D_+$ yields $I_2 = \Id$.  Analogously, the only way the factorization in  \eqref{eq:inner_factor_H-} can hold is if $I_1 = I_2 = \Id$.

Thus $\nlf_+$ is a.e.\ $U(2n)$-valued, and satisfies properties \ref{item:1st_prop_H}--\ref{item:3rd_prop_H} of Definition \ref{defH}.  By Corollary \ref{cor:H_equals_U}, we then have $\nlf_{+} \in \mathbf{H}_{\alpha} ^+$. Similarly, one sees that $\nlf_{-} \in \mathbf{H}_{\alpha, 0} ^-$. This concludes the proof of  existence.
\end{proof}

\subsection{Lipschitz bounds for the factorization}

If we restrict ourselves to the elements $\nlf \in \mathbf{B} _{\alpha}$ arising from some $B \in \mathbf{S}_{\epsilon}$, then we in fact have Lipschitz continuity of the Riemann--Hilbert factorization \eqref{eq:factorization_0}.

\begin{lemma}\label{lem:Lip_cty_RH}
    Let $\alpha:\Z_2 \to \Z_2$. For every $\epsilon > 0$, there exists a constant $C_{\epsilon, n} < \infty $ for which we have the Lipschitz bound
    \[
    \|\nlf_+ - \nlf_+ '\|_{L^2} \leq C_{\epsilon, n} \|B - B'\|_{L^2}
    \]
    for all $B, B' \in \mathbf{S}_{\epsilon}$,  where $\nlf := Y_{\alpha} (B)$, $\nlf' := Y_{\alpha} (B')$, and $\nlf_+$ and $\nlf_+ '$ are the resulting right factors of \eqref{eq:factorization_0} in $\mathbf{H}_{\alpha} ^+$. 
\end{lemma}
\begin{proof}
In what follows, $C_{\epsilon}$ and $c_{\epsilon}$ denote positive constants that may change from instance to instance.

Let $B, B' \in \mathbf{S}_{\epsilon}$. Then $\nlf, \nlf ' \in \mathbf{B}_{\alpha} ^{\epsilon}$.
Define 
    \[
 J : = B^* (\adj A) ^* \, , \qquad J' :=  (B')^* (\adj A')^* \, , 
 \]
 and 
 \[
 \of := \det A^* \, , \qquad \of ' := \det (A')^* \, , 
 \]
  where $A$ and $A'$ are the top left blocks of $Y_{\alpha} (B)$ and $Y_{\alpha} (B')$. 
By Remark \ref{rmk:easy_case_spaces}, the operators $\op$ and $\op'$, as defined using data $J, \of$ and $J', \of'$, respectively, are bounded  on $\spce$.
 %of the form
%\[
%\op= \phl \begin{pmatrix}
%    0 & K \\ -K^* & 0
%\end{pmatrix} \, , \qquad  \op'= \phl \begin{pmatrix}
%    0 & K' \\ -(K')^* & 0
%\end{pmatrix} \, , 
%\]
%where 
%\begin{equation}\label{eq:K_def}
%K:= A^{-1} B \, , \qquad K' := (A')^{-1} B' \, .
%\end{equation}
 Taking $U, V$ and $U', V'$ as defined in \eqref{eq:def_UV}, we compute using \eqref{eq:defn_Y+},
   \[
  (\nlf_+ - \nlf_+ ')^* =  X \begin{pmatrix}
    U ^{-1} & 0 \\ 0 & V ^{-1} 
\end{pmatrix} - X' \begin{pmatrix}
    (U ') ^{-1} & 0 \\ 0 & (V')^{-1} 
\end{pmatrix}
   \]
   \begin{equation}\label{eq:diff_M}
    = (X - X') \begin{pmatrix}
    U ^{-1} & 0 \\ 0 & V ^{-1} 
\end{pmatrix} + X'  \begin{pmatrix}
    U^{-1} - (U') ^{-1} & 0 \\ 0 & V^{-1}- (V ')^{-1} 
\end{pmatrix} \, .
   \end{equation}
To estimate this difference, we proceed in three steps. In the first two steps, we bound each summand in \eqref{eq:diff_M} by $C_{\epsilon} \|K- K'\|_{L^2 }$, where
\begin{equation}\label{eq:def_KK'}
K:= \begin{pmatrix}
   0 & A^{-1}B \\
   -(A^{-1}B)^* & 0
\end{pmatrix}  \, , \qquad K':= \begin{pmatrix}
   0 & (A')^{-1}B' \\
   -((A')^{-1}B')^* & 0
\end{pmatrix} \, ,
\end{equation}
and then in the last step we show 
   \[
   \|K-K'\|_{L^2 } \leq C_{\epsilon} \|B- B'\|_{L^2} \, .
   \]

\underline{   \textbf{Step 1:}  the first summand of \eqref{eq:diff_M} has norm at most $\| K- K'\|_{L^2}$.} Indeed, for the first summand, 
\begin{equation}\label{eq:X_diff}
\left (X- X' \right ) \begin{pmatrix}
    U ^{-1} & 0 \\ 0 & V^{-1} 
\end{pmatrix} = \left \{ \left ((\Id + \op)^{-1} - (\Id + \op')^{-1} \right ) \Id  \right \} \begin{pmatrix}
    U^{-1} & 0 \\ 0 & V^{-1} 
\end{pmatrix}  
\end{equation}
\[
=(\Id + \op')^{-1} (\op' - \op)  \left \{ (\Id + \op)^{-1} \Id \right \} \begin{pmatrix}
    U ^{-1}& 0 \\ 0 & V ^{-1} 
\end{pmatrix}
\]
\[
=(\Id + \op')^{-1} (\op' - \op) \nlf_{ +} ^* =(\Id + \op')^{-1} \phl(K' - K) \nlf_{ +} ^* \, , 
\]
where the last equality follows from \eqref{eq:M_defn_L2}.

Combining \eqref{eq:X_diff} with the operator bound \eqref{eq:antisymmetric_inverse_bd} on $(\Id + \op)^{-1}$ and on the projection $\phl$, we obtain the $L^2$ estimate 
\begin{equation}\label{eq:X_lip_bd}
\left \| \left (X - X' \right ) \begin{pmatrix}
    U^{-1} & 0 \\ 0 & V ^{-1} 
\end{pmatrix} \right \|_{L^2} \leq  \left \|   (K'-K)\nlf_{+} \right \|_{L^2}  = \left \| K'-K \right \|_{L^2} \, , 
\end{equation}
where in the last step we used unitarity of $\nlf_+$.

\underline{\textbf{Step 2:} the second summand of \eqref{eq:diff_M} has norm at most $C_{\epsilon} \|K-K'\|_{L^2}$.} We write
\[
X' \begin{pmatrix}
    U^{-1} - (U')^{-1} & 0 \\ 0 & V^{-1} - (V') ^{-1}
\end{pmatrix} 
\]
\[
= X' \begin{pmatrix}
    (U')^{-1} & 0 \\ 0 & (V') ^{-1}
\end{pmatrix} \begin{pmatrix}
    U'  U ^{-1}- \Id & 0 \\ 0 & V'  V ^{-1}- \Id  
\end{pmatrix} 
\]
\begin{equation}\label{eq:second_term_diff}
    = (\nlf_{+} ')^*  \begin{pmatrix}
    U'  U ^{-1} - \Id & 0 \\ 0 & V '  V ^{-1} - \Id  
\end{pmatrix} \, .
\end{equation}
Since $M_+'$ is unitary, we obtain
\[
\left \|X' \begin{pmatrix}
    U ^{-1} - (U') ^{-1} & 0 \\ 0 & V ^{-1} - (V') ^{-1}
\end{pmatrix} \right \|_{L^2} \leq \left \| \begin{pmatrix}
    U'  U  ^{-1} - \Id & 0 \\ 0 & V'  V ^{-1} - \Id  
\end{pmatrix} \right \|_{L^2}  
\]
\[
\leq \left \| U' U ^{-1} - \Id \right \|_{L^2} + \left \| V' V ^{-1} - \Id \right \|_{L^2} \, .
\]
We only estimate the term involving $U$, since the term involving $V$ is  similar. Using \eqref{eq:UV_to_AD} and \eqref{eq:mult_mean_value}, we have
\[
U = A_{+}  (\infty)  = A_{-}  (\infty) ^{-1} A  (\infty)  \, .
\]
Applying the mean value property to the outer matrix function $A$ then yields
\begin{equation}\label{eq:U_inv_est1}
\|U^{-1}\|_{\infty} \leq \| A_{-} (\infty)\|_{\infty} \|A (\infty) ^{-1}\|_{\infty}  \leq  \|A (\infty) ^{-1}\|_{\infty} \leq \int\limits_{\T} \| A ^{-1}\|_{\infty} \, , 
\end{equation}
which, because $\nlf \in \mathbf{B}_{\alpha}^{\epsilon}$, may then be estimated by
\begin{equation} \label{eq:U_inv_est2}
  \int\limits_{\T} \sqrt{ \| (A^* A) ^{-1}\|_{\infty}}= \int\limits_{\T} \sqrt{ \| (\Id - B^* B) ^{-1}\|_{\infty}} \leq C_{\epsilon} \, .
\end{equation}
Therefore, $U^{-1} $ has operator norm at most $C_{\epsilon}$, and so we have
\[
\left \| U' U ^{-1} - \Id \right \|_{2} \leq \left \|U^{-1} \right \|_{\infty} \left \| U'  - U \right \|_{2}  \leq C_{\epsilon} \left \| U'  - U  \right \|_{2} \, .
\]
Because $\nlf_+$ and $\nlf_+'$ are unitary, then by \eqref{eq:X_to_Y_+} we must have that $X$ and $X'$ have norm at most $1$. In particular, $Q(0)$ and $Q'(0)$ have norm at most $1$. We claim that their eigenvalues are bounded below by some $c_{\epsilon} > 0$. To see this, given a \emph{constant} vector $v \in \C^{2n}$,  write
\begin{equation}\label{eq:Qv_to_Xv}
\left \langle \begin{pmatrix}
    Q(0) & 0 \\ 0 & T(\infty)
\end{pmatrix} v, v \right \rangle = \|X(z) v\|_2 ^2
\end{equation}
for every $z \in \T$,  where $\|u\|_2$ denotes the usual euclidean norm of a vector.
Then multiplying \eqref{eq:fixed_pt_eqn} by $v$ on the right, we get
\[
v = \left [ (\Id +\op)(X v) \right ](z) 
\]
for every  $ z \in\T$, where we use the same abuse of notation as in \eqref{eq:v_vanishes_1}.
Because $B, B' \in \mathbf{S}_{\epsilon}$, then $A^{-1}$ and $(A')^{-1}$ have operator norms bounded by a constant $C_{\epsilon}$, and so we may estimate
\begin{equation}\label{eq:v_to_Xv}
\|v\|_2 = \left ( \int\limits_{\T} \|v\|_2 ^2 \right )^{\frac 1 2} \leq \|\Id + \op\|_{L^2 \to L^2} \|Xv \|_{L^2} \leq C_{\epsilon}\|Xv\|_{L^2} \, .
\end{equation}
Combining \eqref{eq:v_to_Xv} and \eqref{eq:Qv_to_Xv} yields 
\[
\left \langle \begin{pmatrix}
    Q(0) & 0 \\ 0 & T(\infty)
\end{pmatrix} v, v \right \rangle \geq c_{\epsilon} \|v\|^2 _2 \, ,
\]
which implies that all eigenvalues of $Q(0)$ are at least $c_{\epsilon}$. Similarly for  $Q'(0)$.

Because all eigenvalues of $Q(0)$ and $Q'(0)$ are between $c_{\epsilon}$ and $1$, then by the Lipschitz continuity of the Cholesky factorization as in Lemma \ref{lem_cholesky}, we get
\begin{equation}\label{eq:U_to_X_Lip}
\|U-U'\|_2 \leq  C_{\epsilon} \|Q(0) - Q'(0)\|_2 \leq C_{\epsilon} \int\limits_{\T} \|Q - Q'\|_2 \leq C_{\epsilon} \|X-X'\|_{L^2} \,.
\end{equation}
% \begin{equation}
% \leq C_{\epsilon}  \int\limits_{\T} \|X - X'\|_2 \leq C_{\epsilon} \|X-X'\|_{L^2} \, .
% \end{equation}
Multiplying \eqref{eq:X_diff} on the right by $\begin{pmatrix}
    U & 0 \\ 0 & V
\end{pmatrix}$ yields
\[
X - X' = (\Id + \op')^{-1} \phl (K' - K) X\,,
\]
and following the estimates in \eqref{eq:X_lip_bd}, we get that
\[
\|X - X'\|_{L^2} \leq  \|K - K'\|_{L^2 } \, . 
\]
Combining all estimates yields
\[
\left \|\nlf_+ - \nlf_+ ' \right \|_{L^2} \leq C_{\epsilon} \|K - K'\|_{L^2} \, .
\]
   
   \underline{\textbf{Step 3:} showing  $\|K - K'\|_{L^2} \leq C_{\epsilon} \|B-B'\|_{L^2}$.} By  \eqref{eq:def_KK'}, 
   \[
  \|K - K'\|_{L^2} =
  \sqrt{2} \|A^{-1} B - (A')^{-1}B'\|_{L^{2}} \, .
   \]
   Because $B, B'$ have operator norm at most $1-\epsilon$, then $A^{-1}$ and $(A ')^{-1}$ both have operator norm at most $C_{\epsilon}$ and so by the Lipschitz continuity of spectral factors provided by Lemma \ref{lem:cty_spec_factors}, we have 
   \[
   \| A^{-1} B - (A')^{-1} B'\|_{L^2} 
   \leq \|A^{-1} (B-B')\|_{L^2} + \|A^{-1} - (A')^{-1}\|_{L^2} \|B'\|_{L^\infty}
   \]
   \[ 
   \leq C_{\epsilon} \|B- B'\|_{L^2} + C_{\epsilon} \|A' - A\|_{L^2}    \leq C_{\epsilon,n} \|B-B'\|_{L^2} \, ,
   \]
   which completes the proof. 
\end{proof}

\begin{cor}\label{cor:final_Lip_bd}
    Let $\alpha: \Z_2 \to \Z_2$, let $\epsilon > 0$ and let $B, B' \in \mathbf{S}_{\epsilon}$ be the upper right entries of the NLFTs of the square summable sequences $F, F'$. Then the Lipschitz bound \eqref{eq:Lip_bds} holds. 
\end{cor}
\begin{proof}
By the translation symmetry \ref{item:translation} of the NLFT in Lemma \ref{lem:properties}, it suffices to show
\begin{equation}
\|F_0 - F_0 '\|_{\infty} \leq C_{\epsilon, n} \|B-B'\|_{L^2}
\end{equation}
for all $B, B' \in \mathbf{S}_{\epsilon}$.
 Let $\nlf_+, \nlf_+ ' \in \mathbf{H}_{\alpha} ^+$ denote the right Riemann--Hilbert factors of the $\mathbf{B}_{\alpha} ^{\epsilon}$ elements 
\[
\nlf := Y_{\alpha} (B) \, , \qquad  \nlf' :=Y_{\alpha} (B') \, . 
\]
We claim it suffices to show $D_+(0), D_+ '(0)$ have singular values bounded below by some $c_{\epsilon} > 0$. Indeed, assuming this claim, then Lemma \ref{lem:first_layer_strip_Lip} yields
\[
\|F_0 - F_0 '\|_{\infty} \leq C_{\epsilon, n} \|\nlf_{+} - \nlf_+ '\|_{L^2} \, ,
\]
which by Lemma \ref{lem:Lip_cty_RH} is at most $C_{\epsilon, n} \|B-B'\|_{L^2}$, which is exactly what we want to show.

We now show the claim. By the analog of \eqref{eq:mult_mean_value} for $D$ and then the mean value property for $D^{-1}$, which is holomorphic because $D$ is outer, we have 
\begin{equation}\label{eq:estimate_D+_inv_0_1}
\|D_+ (0) ^{-1}\|_{\infty} =  \| D(0)^{-1} D_{-} (0) \|_{\infty} \leq \|D(0)^{-1}\|_{\infty} \leq \int\limits_{\T} \| D  ^{-1}\|_{\infty} \, .
\end{equation}
By unitary of $\nlf$ and the fact that $\|B\|_{L^{\infty}} \leq  1 - \epsilon$, term \eqref{eq:estimate_D+_inv_0_1} is at most
\begin{equation} \label{eq:estimate_D+_inv_0_2}
\int\limits_{\T} \| D^{-1} (D ^*)^{-1}\|^{ \frac 1 2} _{\infty}= \int\limits_{\T} \|(\Id - B^* B)^{-1}\|_{\infty} ^{\frac 1 2} \leq C_{\epsilon} \, .
\end{equation} The proof for $D_+ '(0)$ goes similarly. 
\end{proof}

\subsection{Completing the proof of Theorem \ref{thm:NLFT_outer_inverse}: the Plancherel identity}

We now complete the proof of Theorem \ref{thm:NLFT_outer_inverse}. 

Let $\alpha:\Z_2 \to \Z_2$ and let $B \in \mathbf{S}$. Existence and uniqueness of $F$ follows from existence and uniqueness of the factorization \eqref{eq:factorization_0} for $\nlf_+ \in \mathbf{H}_{\alpha} ^+$ and $\nlf_{-} \in \mathbf{H}_{\alpha, 0} ^-$, which in turn follows from Lemmas \ref{lem:RH_ex} and \ref{lem:RH_nec}. The Lipschitz bound \eqref{eq:Lip_bds} follows from Corollary \ref{cor:final_Lip_bd}. 

We are left with proving \eqref{eq:Plancherel_QSP}. Let $M \in \mathbf{B}_{\alpha}$. By the existence and uniqueness part of Theorem \ref{thm:NLFT_outer_inverse}, there exists a unique sequence $F \in \ell^2 (\Z ; \con)$ for which $
\nlf = \mathcal{F}_{\alpha} (F)$.
By the Plancherel identity \eqref{plancherel}--\eqref{plancherel_ineq}, we have
\[
\frac{1}{2} \sum\limits_{j \in \Z} \log \det (\Id - F_j F_j ^*) = \log  \det A (\infty) = \int\limits_{\T} \log | \det A(z)|  \, ,
\]
where in the last step we applied Lemma \ref{lem:propertiesl2} and the fact that $\det A^*$ is outer and hence has only trivial inner factors.
Then \eqref{eq:Plancherel_QSP} follows by the identity
\[
\lvert \det A \rvert^2 = \det A A^* =  \det (\Id - B B^*) \, ,
\]
where in the last step we used unitarity of $\nlf$.

\appendix

\section{Relation to Quantum Signal Processing}
\label{section:qsv}

When $n=1$, the NLFT was identified with QSP in \cite{Alexis+24}, and in \cite{Alexis+25} the spectral factorization of Theorem \ref{factorization-theorem} was combined with the  Riemann--Hilbert solution of Theorem \ref{thm:NLFT_outer_inverse} to generate the first provably numerically stable algorithm for computing phase factors in QSP. The runtime of the algorithm was subsequently improved in \cite{ni2024fast}, and and then in \cite{ni2025inverse} with a near linear runtime.

QSP in higher dimensions has been far less studied. There are a few natural ways to increase the dimension. One way would be to replace $z \in \T$, which may be considered a $1 \times 1$ unitary matrix, by an $n \times n$ unitary. When $n \geq 2$, this is known as the quantum singular value transformation or quantum eigenvalue transformation, which was shown in  \cite{Gily_2019} to reduce to (one-dimensional) QSP, and so we do not study it here.  Another option is to increase the number of variables, as in the multivariate QSP introduced in \cite{rossi2022multivariable}. This does not correspond to our $SU(2n)$-valued NLFT, which rather increases the dimension of the codomain as the QSP protocols of \cite{La24, lu2024quantum}. However, following \cite{lu2024quantum}, we present an application of our main theorems to multivariate QSP. 

We first turn to the $SU(2^s)$ QSP protocol of \cite[Theorem 3, Corollary 4, Theorem 5]{La24}. In this protocol, one is given a finitely supported sequence $\Psi$ of matrices in $ SU(2^s)$, with which one defines the $SU(2^s)$-matrix valued function
\[
U_{d, \Psi} (z):=\Psi_0 Z^2 \Psi_1 Z^2 \ldots Z^2 \Psi_d \, ,
\]
where we recall 
\[
Z := \begin{pmatrix}
   z^{\frac 1 2 } \Id & 0 \\ 0 & z^{- \frac 1 2} \Id 
\end{pmatrix} \, .
\]
We first note that our NLFT fits within this $SU(2^s)$ QSP protocol.

\begin{lemma}\label{lem:transferrence_QSP_laneve}
Let $\alpha :\Z_2 \to \Z_2$, let $s \geq 0$, and let $F$ be a sequence of $SU_{\alpha} (2^s)$ matrices,  supported on $[-d,d]$. Then
\begin{equation}\label{eq:transferrence_QSP_laneve_NLFT}
  \mathcal{F}_{\alpha} (F) (z^2) = Z^{-2d} U_{d, Y_{\alpha} (F)} (z) Z^{2d} \, , 
\end{equation}
where the sequence $Y_{\alpha}(F)$ is defined by  
\[
Y_{\alpha}(F)_j := Y_{\alpha} (F_j) \, .
\]
\end{lemma}
\begin{proof}
    We write
    \[
    \mathcal{F}_{\alpha} (F)(z^2) = \prod\limits_{j =-d} ^d Z^{2 j}  Y_{\alpha} (F_j)Z^{-2j} = Z ^{-2d} Y_{\alpha} (F_{-d}) \left ( \prod\limits_{j=-d+1} ^d Z^2 Y_{\alpha}(F_j) \right ) Z^{2 d} \, , 
    \]
    which we recognize as the right side of \eqref{eq:transferrence_QSP_laneve_NLFT}.
\end{proof}

\cite[Corollary 4]{La24} characterizes the vector-valued Laurent polynomials 
\[
P(z) := \begin{pmatrix}
    P_1(z) & \ldots & P_{2^s} (z)
\end{pmatrix}^T
\]
of degree at most $d$ for which there is a sequence $\Psi$ supported on $[-d,d]$ such that $P(z)$ occurs as the first column of $U_{d, \Psi} (z)$. While this generalizes results in the $SU(2)$ QSP literature about finding complementary polynomials pairs, we provide an alternate generalization of the complementary polynomials problem. In particular, Lemma \ref{lem:bijection_finite_supp} characterizes which matrix polynomials $\nlf$ can be realized as the NLFT of a finitely supported sequence which, when translated to the QSP protocol via Lemma \ref{lem:transferrence_QSP_laneve}, yields a wide variety of polynomial matrices which can be realized as $U_{d, \Psi} (z)$ for some $d \geq 0$ and $\Psi$. However, since we do not know whether every $SU(2^s)$ QSP protocol is representable as an NLFT protocol,  we do not have a characterization of all the possible polynomials. 

The model of \cite{lu2024quantum} is similar to that of \cite{La24}. However, the latter paper contains an interesting application to bivariate QSP that we here translate into our language of NLFTs. Let $z, w$ be two variables in $\T$. We define the \emph{QSP protocols on $s$ qubits} inductively by the following axioms:
\begin{itemize}
    \item For every $\alpha: \Z_2 \to \Z_2$ and $m \geq 0$, all matrices in $SU_{\alpha} (2^s)$ are QSP protocols.
    \item The $2^{s} \times 2^s$ matrices 
    \[
    Z := \begin{pmatrix}
        z^{\frac 1 2}\Id & 0 \\ 0 & z^{-\frac 1 2} \Id
    \end{pmatrix} \, , \qquad W := \begin{pmatrix}
        w^{\frac 1 2}\Id  & 0 \\ 0 & w^{-\frac 1 2} \Id
    \end{pmatrix}
    \]
    and their inverses are QSP protocols.
    \item If $M$ is a QSP protocol on $s$ qubits, then tensoring $M$ with the $2 \times 2$ matrix $\Id$ yields a QSP protocol on $s+1$ qubits. 
\end{itemize}
A question arising from multivariate QSP asks the following: When can an mvf $B(w,z)$ be represented as the top right block of a finite product of QSP protocols over $s$ qubits, for some $s \geq 1$? We can extend this question to the case of infinite products by taking limits, as we did for the NLFT. We call such mvfs {\it $s$-attainable}. The following lemma is from \cite{lu2024quantum} and  originates in \cite{Gily_2019}.
\begin{lemma}
    If $B_1(z,w)$ and $B_2 (z,w)$ are $s$-attainable mvfs taking values in $\mathcal{M}$, then $B_1 (z,w) B_2 (z,w)$ is $(s+1)$-attainable. 
\end{lemma}
\begin{proof}
Because $B_1$ and $B_2$ are $s$-attainable, then each of the matrices 
\[
\begin{pmatrix}
    * & B_1 (z,w) \\ * & *
\end{pmatrix} \, , \qquad \begin{pmatrix}
    * & B_2 (z,w) \\ * & *
\end{pmatrix} 
\]
can be written as an (infinite) product of QSP protocols over $s$ qubits. By tensoring with the $2 \times 2$ identity matrix in two different ways, we get that the matrices
\[
\begin{pmatrix}
    * & B_1 & 0 &0 \\
    * & * & 0 & 0 \\
    0 & 0 & * & B_1 \\
    0 & 0 & * & *
\end{pmatrix} \, , \qquad \begin{pmatrix}
    * & 0 & B_2 & 0 \\
    0 & * & 0 & B_2 \\
    * & 0 & * & 0 \\
    0 & * & 0 & *
\end{pmatrix}
\]
are  products of QSP protocols over $s+1$ qubits. Multiplying both matrices together, we see the top right block equals $B_1 B_2$, and so $B_1 B_2$ is $(s+1)$-attainable. 
\end{proof}

By Theorem \ref{thm:NLFT_outer_inverse}, if $B$ is a $2^{s-1} \times 2^{s-1}$ mvf of one variable satisfying the Szeg\H o condition, then $B$ is $s$-attainable. Taking $B_1, B_2$ in the previous lemma to be Szeg\H o functions in one variable, we then obtain the following corollary.
\begin{cor} \label{cor:prods_attain}
    Let $n = 2^{s-1}$. If $B_1, B_2 \in \mathbf{S}$, then $B_1 (z) B_2 (w)$ is $(s+1)$-attainable. 
\end{cor}

In \cite{lu2024quantum}, the authors considered the case $B(z,w)$ was a $1 \times 1$ mvf, and worked with an even larger set of protocols than we do in this appendix. Since not all such functions $B(z,w)$ are $s$-attainable for any $s$, and because they lacked an $\ell^2$ theory for QSP, the authors limited themselves to the case where
\begin{equation}\label{eq:scalar_poly}
B(z,w) = \sum\limits_{j=1} ^{2 ^{s-2}} p_j (z) q_j (w) \, ,
\end{equation}
where $p_j$, $q_j$ are polynomials in one variable, both satisfying
\[
\sum\limits_{j=1} ^{2^{s-2}} |p_j (z)| \leq 1 \, , \qquad \sum\limits_{j=1} ^{2 ^{s-2}} |q_j (w)| \leq 1 \, . 
\]
They showed that $B(z,w)$ was $s$-attainable using a finite number of their QSP protocols. Their strategy was to first set $B_1(z)$ to be the $2^{s-2} \times 2^{s-2} $ mvf whose first row is $(p_1 (z), \ldots , p_{2^{s-2}}(z))$ and whose remaining entries are all zero, and to set $B_2$ to be the $2^{s-2} \times 2^{s-2} $ mvf whose last column is \[\sum\limits_{j=1} ^{2^{s-2}} q_j (z) e_j,\] and whose remaining entries are all zero. Then applying the analog of Corollary \ref{cor:prods_attain} for their enlarged QSP protocols yields that the product $B_1(z) B_2 (w)$ has upper left entry exactly the scalar function \eqref{eq:scalar_poly}, showing \eqref{eq:scalar_poly} is attainable. Our Corollary \ref{cor:prods_attain} recovers their result whenever the functions $p_j$ and $q_j$, $j=1, \ldots, 2^{s-2}$ are not constant, but its true novelty   is that it allows for non-polynomial functions.

\section*{Acknowledgements}

The first author is thankful to Lorenzo Laneve for helpful conversations which helped form the appendix of this paper.

%\subsection*{Funding and Financial Interests}

MA and CT were supported by Deutsche Forschungsgemeinschaft (DFG,
German Research Foundation) under DFG CRC 1060 – 211504053 and Germany’s Excellence Strategy project 390685813 – Hausdorff Center for Mathematics. CT was additionally supported under DFG CRC 1720 – 539309657.
DOS was supported by the Portuguese government through FCT -- Fundação para a Ciência e a Tecnologia, I.P., project UIDB/04459/2020 with DOI identifier 10-54499/UIDP/04459/2020 (CAMGSD), and project 2023.17881.ICDT with DOI identifier 10.54499/2023.17881.ICDT (project SHADE), as well as by the Deutsche Forschungsgemeinschaft 
 (DFG, German Research Foundation) under Germany's Excellence Strategy – EXC-2047/1 – 390685813.

%\subsection*{Nonfinancial interests}
%The authors have no relevant nonfinancial interests to disclose.

%\subsection*{Data Availability}
%We do not analyze or generate any datasets, because our work proceeds within a theoretical and mathematical approach. One can obtain the relevant materials from the references below.

\bibliographystyle{plain} % We choose the "plain" reference style
\bibliography{biblio}

\newpage
\section*{Glossary}
\label{sec:glossary}
\addcontentsline{toc}{section}{\nameref{sec:glossary}}

%\ma{should include $Z$ here and other new notation.}

\begin{tabular}{l p{0.75\textwidth}}
$\|\cdot \|_{p}$ & Matrix Schatten $p$-norm \\
$\|\cdot \|_{L^p}$ & Matrix $L^p$ norm, $\|M\|_{L^p} := (\int\limits_{\T} \|M\|_p ^p)^{\frac1p}$\\
$\mathcal{M}$ & The set of $n \times n$ matrices \\
$\con$ & The set of $n \times n$ matrices $F$ for which $\|F\|_{\infty} < 1$\\
$\ad(z)$ & Formally, $\ad(z) X := \begin{pmatrix}
    z^{\frac 1 2} & 0 \\ 0 & z^{- \frac 1 2}
\end{pmatrix} X \begin{pmatrix}
    z^{-\frac 1 2} & 0 \\ 0 & z^{\frac 1 2}
\end{pmatrix}$\\ 
$Z$ & Formally, $Z:= \begin{pmatrix}
    z ^{\frac 1 2}  & 0 \\ 0 & z^{- \frac 1 2}
\end{pmatrix}$ \\
    $\grp_{0}$ (or $\grp_1$) & Group of $n \times n$ upper (or lower) triangular matrices with positive diagonal entries \\
    $U(m)$ & $m \times m$ unitary matrices \\
   $SU(m)$ & $m \times m$ unitary matrices with determinant $1$ \\ 
   $\outr_{a} (\weight)$ & Outer  mvf solving $\weight = \outr_a (\weight) ^* \outr_a (\weight)$ with  $\outr_a (\weight) (0) \in \grp_a$\\
   $\sqrtt_a (P)$ & Matrix solving $\sqrtt_a (P) \sqrtt_a (P) ^* = P$ and $\sqrtt_a (P) \in \grp_a$.\\
    $SU_{\alpha} (2n)$ & $SU(2n)$-matrices with upper left and lower right diagonal blocks in $\grp_{\alpha(0)}$ and $\grp_{\alpha(1)}$, respectively \\
          
$\mathbf{L}_{\alpha}$  & $SU(2n)$-valued functions $\begin{pmatrix}
    A & B \\ C & D
\end{pmatrix} \in \begingroup\renewcommand{\arraystretch}{1.0} \begin{pmatrix}
    H^2 (\D^*) & L^2 (\T) \\ L^2 (\T) & H^2 (\D)
\end{pmatrix} \endgroup$ such that $A (\infty) \in \grp_{\alpha(0)}$ and $D(0) \in \grp_{\alpha(1)}$ \\
 
    $\mathbf{B}_{\alpha}$ (or $\mathbf{B}_{\alpha} ^{\epsilon})$ & $\mathbf{L}_{\alpha}$-functions for which $A, D$ are outer (and $\|B\|_{L^{\infty}} < 1- \epsilon$) \\

$\mathbf{S}$  & Functions $B: \T \to \mathcal{M}$ with $\|B\|_{L^\infty} \leq 1$ satisfying the Szeg\H o condition $\int_{\T} \log \det (\Id - B B^*) > - \infty$ \\
      $\mathbf{S}^{\epsilon}$  & Functions $B: \T \to \mathcal{M}$ for which $\|B\|_{L^\infty} < 1 - \epsilon$ \\
      $Y_{\alpha}$ & The unique map $\mathbf{S} \to \mathbf{B}_{\alpha}$ which embeds $B \in \mathbf{S}$ as the upper right block of $Y_{\alpha}(B) \in \mathbf{B}_{\alpha}$ \\
      $\mathcal{F}_{\alpha}$ & The $\alpha$-$SU(2n)$ nonlinear Fourier transform \\
      $\mathbf{H}_{\alpha}^+$ (or $\mathbf{H}_{\alpha, 0} ^{-}$)  & Image of $\ell^2 (\Z_{\geq 0}; \con)$ (or $\ell^2(\Z_{<0}; \con)$) under $\mathcal{F}_{\alpha}$\\
      $\mathbf{U}_{\alpha} ^+$ (or $\mathbf{U}_{\alpha, 0} ^{-}$) & Space equal to $\mathbf{H}_{\alpha} ^+$ (or $\mathbf{H}_{\alpha, 0} ^{-}$), but defined by \emph{a priori} weaker constraints \\
      $\mathbf{L}_{\alpha} ^+$ (or $\mathbf{L}_{\alpha, 0} ^-$)  & Intersection of $\mathbf{L}_{\alpha}$ with $ \begin{pmatrix}
      H^2 (\D^*) & H^2  (\D) \\ H^2  (\D^*) & H^2 (\D)
  \end{pmatrix}$ (or  $\begin{pmatrix}
      H^2 (\D^*) & H^2 _0 (\D^*) \\ H^2 _0 (\D) & H^2 (\D)
  \end{pmatrix}$) \\ 
$\mathbf{L}^+$ and $\spce$ & The spaces $\begin{pmatrix}
    H^2 (\D^*) & H^2 (\D) \\ H^2 (\D ^*) & H^2 (\D)
\end{pmatrix}$ and  $\begin{pmatrix}
    H^2 (\D) & H^2 (\D) \\ H^2 (\D ^*) & H^2 (\D^*)
\end{pmatrix}$ \\
$\phl$ & Hilbert space projection $L^2 (\T) \to \spce$\\
$P_{\D}$ (or $P_{\D^*}$) & Fourier projection $L^2 (\T) \to H^2 (\D)$ (or $L^2 (\T) \to H^2 (\D^*)$) \\
$\op$ & Densely defined operator on $\spce$ \\
$\ddense$ & Domain of definition of the densely defined operator $\op$ \\
$\dense$ & The space $\begin{pmatrix}
    \of & 0 \\ 0 & \of ^*
\end{pmatrix} \spce$, dense in $\ddense$ and $\spce$ \\

  $\weights$ & Positive semidefinite matrix-valued functions satisfying $\int\limits_{\T} \log \det P > - \infty$
\end{tabular}
\end{document}